\documentclass[a4paper,10pt, reqno]{amsart}
\usepackage{amsthm, amsmath, amssymb, amsfonts, amsopn, color, mdwlist}
\usepackage{enumerate}
  \newcommand{\e}{\varepsilon}
\newcommand{\RN}{\mathbb{R}^2}

\newcommand{\intr}{\int_{\mathbb{R}^2}}

\newtheorem{theorem}{Theorem}[section]
\newtheorem{definition}{Definition}[section]
\newtheorem{corollary}[theorem]{Corollary}

\newtheorem{proposition}[theorem]{Proposition}
\newtheorem{lemma}[theorem]{Lemma}

\numberwithin{equation}{section}

\begin{document}
\title[Blow up  at infinity in the  SU(3) Chern-Simons model, part I]{Blow up  at infinity in the  SU(3) Chern-Simons model, part I}
\author{Ting-Jung Kuo}
\address[Ting-Jung Kuo]{
Department of mathematics, National Taiwan Normal University, Taipei, 11677, Taiwan}
\email{tjkuo1215@ntnu.edu.tw}
\author{Youngae Lee}
\address[Youngae Lee]{ Department of Mathematics Education, Teachers College, Kyungpook National University, Daegu, South Korea}
\email{youngaelee@knu.ac.kr}
\author{Chang-Shou Lin}
\address[Chang-Shou Lin]{Taida Institute for Mathematical Sciences, Center
for Advanced Study in Theoretical Sciences, National Taiwan University,
No.1, Sec. 4, Roosevelt Road, Taipei 106, Taiwan}
\email{cslin@math.ntu.edu.tw }

\begin{abstract}
We consider  non-topological solutions of a nonlinear elliptic system problem (see \eqref{cs4} below) derived from the $SU(3)$
Chern-Simons models in $\mathbb{R}^2$.
The existence of  non-topological solutions even for radial  symmetric case has been a long standing open problem.
 Recently, Choe, Kim, and Lin in \cite{CKL2,CKL3}  showed the existence of radial symmetric non-topological solution when the vortex points collapse.
However, the arguments in \cite{CKL2,CKL3} cannot work for an arbitrary configuration of vortex points.  In this paper, we develop a new approach by using  different scalings
for different components of the system to construct a family of non-topological solutions, which blows up at infinity.
\end{abstract}

\date{\today }
\maketitle

\section{Introduction}

The relativistic self-dual Abelian Chern-Simons-Higgs model has been
proposed by Hong, Kim, Pac in \cite{HKP} and independently by Jackiw,
Weinberg in \cite{JW} to explain the high critical temperature
superconductivity.  The following nonlinear elliptic equation is derived
from the Euler-Lagrangian equation of the Chern-Simons-Higgs model via a
vortex ansatz (see \cite{HKP, JW, T2, yang1}):
\begin{equation}
\Delta v\left( x\right) +e^{v\left( x\right) }\left( 1-e^{v\left( x\right)
}\right) =4\pi \sum_{j=1}^{N}\delta_{p_{j}}\left( x\right) \mbox{ in }%
\mathbb{R}^{2},  \label{cs1}
\end{equation}
where $p_{j}^{\prime}s$ are called vortex points and $\delta_{p_{j}}$ denotes the
Dirac measure at the vortex point $p_{j}$. Here $p_{j}^{\prime}s$ are not
necessary to be distinct. We also refer to \cite{CK, CKL,   LY2, NT} for
recent developments of \eqref{cs1}. A solution $v\left( x\right) $ of (\ref%
{cs1}) is called a topological solution if $v\left( x\right) \rightarrow0$
as $\left \vert x\right \vert \rightarrow+\infty$, and is called a
non-topological solution if $v\left( x\right) \rightarrow-\infty$ as $\left
\vert x\right \vert \rightarrow+\infty$.

Let
 $\beta(v)=\int_{\mathbb{R}^{2}}e^{v\left( x\right) }\left( 1-e^{v\left( x\right) }\right) dx$
be the total magnetic flux (or mass) of $v$. Equivalently, $v$ satisfies
\begin{equation*}
v(x)=-\Big(\frac{\beta(v)}{2\pi}-2N\Big)\ln|x|+O(1) \ \mbox{at}\ \infty.
\end{equation*}

It is easy to show that if $v$ is a topological solution of \eqref{cs1},
then $\beta (v)=4\pi N$. For any configuration of vortex points, Wang in
\cite{wang}, and Spruck and Yang in \cite{S-W1} proved the existence of
topological solution of \eqref{cs1}. To find topological
solutions of \eqref{cs1}, they used the monotone scheme and the variational method
together with the Moser-Trudinger inequality.

However, it is more difficult to find non-topological solutions of \eqref{cs1}.
Generally speaking, the variational method together with the
Moser-Trudinger inequality could not work well for finding solutions
with logarithmic growth   at infinity. Similar situation in the
literature has been occurred in the study of the scalar curvature equation
in $\mathbb{R}^{2}$, where the coefficient $k(x)$ is always assumed to have
certain decays at infinity and the range of the logarithmic growth of a
solution has certain constraint. For example, see \cite{M}.

Spruck and Yang in \cite{S-W2} proved that if $v(x)$ is a radial  symmetric nontopological
solution of \eqref{cs1}, then its flux satisfies $\beta (v)>8\pi \left(
1+N\right) $. Since \cite{S-W2}, the following question to the equation \eqref{cs1}
has been arised:\medskip

\emph{For any configuration} $\{p_{j}\}_{j=1}^{N}$ \emph{and} $\beta >8\pi
\left( 1+N\right) $, \emph{does there exist a non-topological solution of} \eqref{cs1} \emph{such that} $\beta (v)=\beta ?$ \medskip

In \cite{CI}, by viewing (\ref{cs1}) as a perturbation of the Liouville
equation, Chae and Imanuvilov obtained the existence of non-topological
solution when $\beta $ is closed to $8\pi \left( 1+N\right) $. But it is unknown whether $\beta >8\pi (1+N)$ or not for this particular solution. Later, Chan, Fu, and Lin in \cite{CFL} solved the
problem \eqref{cs1} for the case $p_{j}=0$ for any $j$. Moreover, they proved the
existence and uniqueness of  radial  symmetric non-topological solutions for
any $\beta >8\pi (1+N)$. Finally, Choe, Kim and Lin \cite{CKL} used the
degree theory and proved the following theorem :

\medskip
\noindent
\textbf{Theorem A.}\cite{CKL} \emph{For any configuration of vortex points
and any number }$\beta >8\pi  ( 1+N ) $\emph{\ satisfying }$\beta \notin
 \{ \frac{8\pi N k}{k-1}|k=2,\cdots,N \} $\emph{,
there is a non-topological solution }$v$\emph{\ to (\ref%
{cs1}) satisfying }%
$\beta (v)=\beta$.
\medskip

The idea in \cite{CKL} is to obtain   a priori bound  for solutions when (%
\ref{cs1}) is deformed by collapsing vortex points into a single one, and
then they could compute the topological degree by the result of radial
symmetric solutions in \cite{CFL}. However, the existence   still
remains open for $\beta \in \left\{ 8\pi N\frac{k}{k-1}|k=2,3,\cdot \cdot
\cdot ,N\right\} $.

The non-Abelian relativistic Chern-Simons models was proposed by Kao-Lee
\cite{KL} and Dunne \cite{D1, D2, D3}. The model is defined in the $\left(
2+1\right) $ Minkowski space $\mathbb{R}^{1,2}$, and the gauge group is a
compact Lie group with a semi-simple Lie algebra $(\mathcal{G},[,])$. The
Chern-Simons Lagrangian density in $\left( 2+1\right) $ dimensional
spacetime involves the Higgs field $\phi $ with values in $\mathbb{C}$ and
the gauge potential $A=(A_{0},A_{1},A_{2})$. We restrict to consider the
energy minimizers of the Lagrangian functional, and  obtain the
self-dual Chern-Simons equations:
\begin{equation}
\left\{
\begin{array}{l}
D_{-}\phi  =0,  \label{cs2} \\
F_{+-} =\frac{1}{\kappa ^{2}}\left[ \phi -\left[ \left[ \phi ,\phi ^{+}%
\right] ,\phi \right] ,\phi ^{+}\right] ,
\end{array}%
\right.
\end{equation}%
where $D_{-}=D_{1}-iD_{2}$ and $F_{+-}=\partial
_{+}A_{-}-\partial _{-}A_{+}+\left[ A_{+},A_{-}\right] $ with $A_{\pm
}=A_{1}\pm iA_{2}$ and $\partial _{\pm }=\partial _{1}\pm i\partial _{2}$.
Dunne considered a simplified form of the self-dual system \eqref{cs2}, in
which the fields $\phi $ and $A$ are algebraically restricted:%
\begin{equation*}
\phi =\sum_{a=1}^{r}\phi ^{a}E_{a},\ A_\mu=-i\sum_{a=1}^r A^a_{\mu}H_a,
\end{equation*}%
where $r$ is the rank of the gauge Lie algebra, $E_{a}$ is the simple root
step operator, $H_a$ is Cartan subalgebra elements,  $\phi ^{a}$ are complex-valued functions, and $A^a_{\mu}$ is real value function. Let
$v_{a}=\log \left\vert \phi ^{a}\right\vert \mbox{, }a=1,\cdot \cdot \cdot ,r.$
Then the equation (\ref{cs2}) can be reduced to the following system of
equations:%
\begin{equation}
\Delta v_{a}+\frac{1}{\kappa ^{2}}\left(
\sum_{b=1}^{r}K_{ab}e^{v_{b}}-\sum_{b=1}^{r}%
\sum_{c=1}^{r}e^{v_{b}}K_{bc}e^{v_{c}}K_{ac}\right) =4\pi
\sum_{j=1}^{N_{a}}\delta _{p_{j}^{a}}\mbox{ in }\mathbb{R}^{2},a=1,\cdot
\cdot \cdot ,r,  \label{cs3}
\end{equation}%
where $K=\left( K_{ab}\right) $ is the Cartan matrix of a semi-simple Lie
algebra, $\left\{ p_{j}^{a}\right\} $ are zeros of $\phi ^{a}$. We refer to
\cite{yang2} for the detail from (\ref{cs2}) to (\ref{cs3}).

We want to study the case when the compact Lie group is $SU\left(
3\right) $ and the associated Cartan matrix $K=\left(
\begin{array}{cc}
2 & -1 \\
-1 & 2%
\end{array}%
\right) $. Without loss of generality, we assume $\kappa =1$. Then (\ref{cs3}) becomes%
\begin{equation}
\left\{
\begin{array}{l}
\Delta v_{1}+2e^{v_{1}}\left( 1-2e^{v_{1}}\right) -e^{v_{2}}\left(
1-2e^{v_{2}}-e^{v_{1}}\right) =4\pi \sum_{j=1}^{N_{1}}\delta _{p_{j}} \\
\Delta v_{2}+2e^{v_{2}}\left( 1-2e^{v_{2}}\right) -e^{v_{1}}\left(
1-2e^{v_{1}}-e^{v_{2}}\right) =4\pi \sum_{j=1}^{N_{2}}\delta _{q_{j}}%
\end{array}%
\right. \mbox{in }\mathbb{R}^{2}.  \label{cs4}
\end{equation}%
 According to the
asymptotic behavior of solutions at infinity, we define
three types of solutions as follows:
\begin{definition}
\label{D2} Let $\left( v_{1},v_{2}\right) $ be a solution of (\ref{cs4}).

(i) $\left( v_{1},v_{2}\right) $ is called a topological solution if $%
\lim_{\left \vert x\right \vert \rightarrow+\infty}v_{i}(x)=0$, $i=1,2.$

(ii) $\left( v_{1},v_{2}\right) $ is called a non-topological solution if $%
\lim_{\left \vert x\right \vert \rightarrow+\infty}v_{i}(x)=-\infty$, $i=1,2.$

(iii) $\left( v_{1},v_{2}\right) $ is called a mixed-type solution if either%
\begin{equation*}
\lim_{\left\vert x\right\vert \rightarrow+\infty }\left( v_{1}(x),v_{2}(x)\right)
=\left( \ln \frac{1}{2},-\infty \right) \text{ or }\lim_{\left\vert
x\right\vert \rightarrow+\infty }\left( v_{1}(x),v_{2}(x)\right) =\left( -\infty
,\ln \frac{1}{2}\right).
\end{equation*}
\end{definition}

For any configuration $\{p_{j}^{a}\}$ in $\mathbb{R}^{2}$, Yang in \cite%
{yang2} proved the existence of topological solutions for (\ref{cs3}), again
by the variational method and the Moser-Trudinger inequality. Since \cite%
{yang2}, it has been a long standing open problem to find non-topological as
well as mixed-type solutions.

 In this paper, we mainly concern with the structure of non-topological solutions of \eqref{cs4}.
Given any pair $(\beta _{1},\beta _{2})$ with $\beta _{i}>1$, $i=1,2$, we
ask the following question about the existence of non-topological
solutions:\medskip

\emph{Given any  pair} $(\beta _{1},\beta _{2})$, \emph{is there a solution} $(v_{1},v_{2})$ \emph{of} \eqref{cs4} \emph{%
satisfying   }
\begin{equation}
\left\{
\begin{array}{l}
v_{1}(x)=-2\beta _{1}\ln |x|+O(1) \\
v_{2}(x)=-2\beta _{2}\ln |x|+O(1)%
\end{array}%
\right. \ \emph{as}\ |x|\rightarrow +\infty ?  \label{q1}
\end{equation}%

Very recently, the first partial  result for the existence of
non-topological solutions of \eqref{cs4} satisfying \eqref{q1} has been obtained by Ao, Lin, and Wei in \cite{ALW}%
. Similar to the single equation \eqref{cs1}, they considered \eqref{cs4} as a
perturbation of $SU(3)$ Toda system, and obtained non-topological solutions when $(\beta_1,\beta_2)$ is closed to $(N_2+2,N_1+2)$. However, (\ref{cs4}) is much more difficult to
solve than (\ref{cs1}) due to the large dimension of freedom for solutions
to $SU(3)$ Toda system, but they could overcome this difficulty by  assuming some conditions on $\{p_j\}$, $\{q_j\}$, $N_1$ and $N_2$ (see \eqref{alw_con} below).

At this point, one might ask  what is the possible range of $(\beta _{1},\beta _{2})$ for
existence of solutions to \eqref{cs4} satisfying (\ref{q1})? Obviously, $\beta _{i}$
must satisfy $\beta _{i}>1$ for $e^{v_i}\in L^1(\RN)$, $i=1,2$. To find the suitable range of $%
\beta _{i}$, it is natural to start with radial symmetric solutions of %
\eqref{cs4}. In \cite{HL}, Huang and Lin studied the system (\ref{cs4}) when
$p_{j}=q_{j}=0$, that is,%
\begin{equation}
\left\{
\begin{array}{l}
 \frac{d^2 v_{1}}{dr^2}+\frac{1}{r}\frac{d v_1}{dr} +2e^{v_{1}}\left( 1-2e^{v_{1}}\right) -e^{v_{2}}\left(
1-2e^{v_{2}}-e^{v_{1}}\right) =4\pi N_{1}\delta _{0} \\
\frac{d^2 v_{2}}{dr^2}+\frac{1}{r}\frac{d v_2}{dr}+2e^{v_{2}}\left( 1-2e^{v_{2}}\right) -e^{v_{1}}\left(
1-2e^{v_{1}}-e^{v_{2}}\right) =4\pi N_{2}\delta _{0}%
\end{array}%
\right. . \label{cs5}
\end{equation}%
It is convenient to rewrite the nonlinear terms of \eqref{cs5} as $%
2F_{1}-F_{2}$ and $2F_{2}-F_{1}$ respectively, where
$
F_{1}=e^{v_{1}}-2e^{2v_{1}}+e^{v_{1}+v_{2}},\
F_{2}=e^{v_{2}}-2e^{2v_{2}}+e^{v_{1}+v_{2}}.
$
Among others, they proved

\medskip
\noindent
\textbf{Theorem B.}\cite{HL} \emph{Let }$\left( v_{1}(r),v_{2}(r)\right) $
\emph{be an entire radial solutions of \eqref{cs5}. Then}
\begin{equation}
\int_{\mathbb{R}^{2}}|F_{1}|dx<+\infty ,\quad \int_{\mathbb{R}%
^{2}}|F_{2}|dx<+\infty, \label{bddint}
\end{equation}%
\emph{and} $\left( v_{1},v_{2}\right) $ \emph{must satisfy one of the boundary
condition at }$\infty $ $\emph{in}$ \emph{Definition \ref{D2}. Moreover, one of the
followings holds.}

\noindent
(1) \emph{If} $\left( v_{1},v_{2}\right) $ \emph{is a non-topological
solution, then}%
\begin{equation}
\begin{aligned}\label{b} \left\{ \begin{array}{ll} v_{1}\left( r\right) &
=-2\beta_{1}\ln r+O\left( 1\right)  \\ v_{2}\left( r\right) &
=-2\beta_{2}\ln r+O\left( 1\right)  \end{array}\right.\ \emph{at}\ \infty, \end{aligned}
\end{equation}%
 \emph{for some }$\beta _{1}>1$\emph{\ and }$\beta _{2}>1$%
\emph{. Furthermore, }%
\begin{equation}
J\left( \beta _{1}-1,\beta _{2}-1\right) >J\left( N_{1}+1,N_{2}+1\right)
,\ \textrm{where }\ J\left( x,y\right) =x^{2}+xy+y^{2}.
\label{1}
\end{equation}
\noindent
(2) \emph{If} $(v_{1},v_{2})$ \emph{is a mixed-type solution, then either}

$\quad v_1(r)\to\ln\frac{1}{2}$ \emph{and} $v_2(r)=-2\beta_2\ln r +O(1)$ \emph{%
for some} $\beta_2>1$ as $r\to+\infty$ \emph{or}

$\quad v_2(r)\to\ln\frac{1}{2}$ \emph{and} $v_1(r)=-2\beta_1\ln r +O(1)$ \emph{%
for some} $\beta_1>1$ as $r\to+\infty$.

\noindent
(3) \emph{If} $(v_{1},v_{2})$ \emph{is a topological solution, then}
$ (v_{1}(r),v_{2}(r))\to(0,0)$ \emph{exponentially as} $r\to+\infty$. \medskip

Compared to \eqref{cs1}, one of main difficulties of \eqref{cs4} is whether
the nonlinear terms $2F_{1}-F_{2}$ and $2F_{2}-F_{1}$ belong to  $L^{1}(\mathbb{R}%
^{2})$ or not. Indeed, it is clear that our system is neither cooperative
nor competitive in the conventional classification of systems of equations,   that is, each nonlinear term in \eqref{cs4} is not  monotone with respect to any of $v_1$ and $v_2$.
Even in the case for radial symmetric solutions, the proof of %
\eqref{bddint} is not easy. For non-radial  solution of %
\eqref{cs4}, the question whether \eqref{bddint} holds or not is an
important open problem for understanding the complete picture for the
structure of solutions to (\ref{cs4}).

According to \eqref{1}, we define%
\begin{equation*}
\Omega =\left\{ \left( \beta _{1},\beta _{2}\right) |\ \beta _{1}>1,\beta
_{2}>1,J\left( \beta _{1}-1,\beta _{2}-1\right) >J\left(
N_{1}+1,N_{2}+1\right) \right\} .
\end{equation*}%
Then naturally we ask:

\medskip
\noindent
\textbf{Question.} \emph{Suppose }$\left( \beta _{1},\beta _{2}\right) \in
\Omega $\emph{. Does there exist a radial symmetric solution of (\ref{cs5}%
) subject to the boundary condition (\ref{b})?\medskip }

In view of Theorem B, the condition for $\left( \beta _{1},\beta _{2}\right)
\in \Omega $ is necessary for the existence of radial symmetric
non-topological solutions of (\ref{cs5}). However, Choe, Kim, Lin in \cite{CKL2}
pointed out that $\left( \beta _{1},\beta _{2}\right) \in \Omega $ might not
be sufficient for the existence of solutions due to the phenomena of partial
blowing up, i.e. only one of components of solutions tends to $-\infty$ locally in   $\RN$,
but the other does not. In \cite{CKL2}, Choe, Kim and Lin studied the above
problem and introduced for any $N=(N_{1},N_{2})$,
\begin{equation*}
S_{N}=\left\{ \left( \beta _{1},\beta _{2}\right) \left\vert
\begin{array}{l}
\beta _{1}>1,\beta _{2}>1, \\
-2N_{1}-N_{2}-3<\beta _{2}-\beta _{1}<N_{1}+2N_{2}+3, \\
2\beta _{1}+\beta _{2}>N_{1}+2N_{2}+6\mbox{, }\beta _{1}+2\beta
_{2}>2N_{1}+N_{2}+6%
\end{array}%
\right. \right\} .
\end{equation*}%
We note that $S_{N}\subseteq \Omega $ (see \cite{CKL2} for the proof). They
derived a priori bound for all solutions by deforming the parameters $%
\beta _{1}$ and $\beta _{2}$, and applied the degree theory to prove the
existence of radial symmetric non-topological  solutions provided that $%
(\beta _{1},\beta _{2})\in S_{N}$:

\medskip
\noindent
\textbf{Theorem C.}\cite{CKL2} \emph{If} $\left( \beta_{1},\beta_{2}\right)
\in S_{N}$\emph{\ then (\ref{cs5}) has a radial symmetric solution subject
to the boundary condition (\ref{b}).}

\medskip

As we already
mentioned, the main step in the proof of Theorem C is to derive the
global a priori bounds for any solution with $(\beta _{1},\beta _{2})\in
S_{N}$. In fact, it has also been proved that $S_{N}$ is the best possible
region to obtain a priori bounds for radial solutions. Indeed, for $(\beta
_{1},\beta _{2})\in \partial S_{N}$, they did prove the existence of partial
blowing up solutions. More precisely, let $\ell _{i}$ be the line
defined by%
\begin{equation*}
\ \ \quad\ell _{1}:=\left\{ \ (\beta _{1},\beta _{2})\ |\ \beta _{1}>1,\beta
_{2}>1,\beta _{2}-\beta _{1}=N_{1}+2N_{2}+3\right\} ,\text{ and}
\end{equation*}%
\begin{equation*}
\ell _{2}:=\left\{ \ (\beta _{1},\beta _{2})\ |\ \beta _{1}>1,\beta
_{2}>1,2\beta _{1}+\beta _{2}=N_{1}+2N_{2}+6\right\} \text{.}
\end{equation*}%
Then  $\ell _{1}\cup \ell _{2}\subseteq \partial S_{N}$.
They used shooting method to prove

\medskip
\noindent
\textbf{Theorem D.}\cite{CKL2,CKL3}  \emph{Let }$\left( \beta _{1},\beta
_{2}\right) \in \ell _{1}\cup \ell _{2}$\emph{. Then there exists a sequence
of partial blowing up (radial) solutions }$\left( v_{1,n}\left( r\right)
,v_{2,n}\left( r\right) \right) $ \emph{ to (\ref{cs5}) such that}
 \begin{equation*}
 \left \{
\begin{array}{l} v_{i,n}=-2\beta_{i,n}\ln r+O\left( 1\right) \mbox{ as }r\rightarrow +\infty,
\ i=1,2,\\   \left( \beta_{1,n},\beta_{2,n}\right) \rightarrow \left( \beta_{1},\beta
_{2}\right) \in \ell_1\cup\ell_2,
\end{array}
\right.
\end{equation*}
\emph{and $\lim_{n\rightarrow  \infty}(v_{1,n}(r),v_{2,n}(r))=(V(r),-\infty)$  in $C_{\textrm{loc}}(\RN)$, where $V(r)+\ln2$ is a solution of  \eqref{cs1}. Let $V$ satisfy
 $ V(r)=-2\left( {\gamma} -N_{1}\right) \ln r
+O\left( 1\right) \mbox{ as }\  r\rightarrow+\infty$. Then  \\
 (i) if $(\beta_1,\beta_2)\in\ell_1$, then $\gamma>2N_1+2N_2+4$, and only $v_{2,n}$  blows up at infinity,
\\
(ii) if $(\beta_1,\beta_2)\in\ell_2$, then $\gamma<2N_1+2N_2+4$, and both $v_{1,n}$ and $v_{2,n}$  blow up at infinity \\
Here  we say $v_{i,n}$ blows up at infinity if the contribution of the mass of $v_{i,n}$ at infinity cannot be ignored.
}

\medskip


In view of  Theorem D,  we note that  $2N_1+2N_2+4$ is the critical value for $\gamma$ to determine the asymptotic behavior of $\left( v_{1,n},v_{2,n}\right)$. If  $\gamma>2N_1+2N_2+4$, then
    only  $v_{2,n}$ blows up at infinity, that is, although  the contribution of the mass of $v_{2,n}$ at infinity cannot be ignored, but this behavior of $v_{2,n}$ does not affect the total mass of $v_{1,n}$ such that $\lim_{n\to+\infty}\int_{\RN}2e^{v_{1,n}(x)}(1-2e^{v_{1,n}(x)})dx=\intr 2e^{V(x)}(1-2e^{V(x)})dx$.  We call this kind of solutions \textit{bubbling type I solutions}. \\ On the other hand, the case  $\gamma<2N_1+2N_2+4$  is rather more complicated since not only $v_{2,n}$ but also $v_{1,n}$ blow  up at infinity  on the different regions. We call this kind of solutions \textit{bubbling type II solutions}.

\medskip

After Theorem C and Theorem D, one might ask whether $S_{N}$ is the optimal
range of $(\beta _{1},\beta _{2})$ to guarantee the existence of solutions to %
\eqref{cs4} satisfying the boundary conditions \eqref{q1}. As in the
single equation  \eqref{cs1}, it is very difficult to use the variational method
to solve the problem \eqref{q1}. So we turn to the method of degree theory
as Choe, Kim, and Lin did in \cite{CKL}. Starting with this paper, we want
to initiate the program to study this problem. The first step is to
see whether a priori bound would fail if $(\beta _{1},\beta _{2})\in
\partial S_{N}$, in other words, we want to understand for a generic set of
vortex points  and $(\beta _{1},\beta _{2})\in \partial S_{N}$, whether
there is a sequence of blowing-up non-topological solutions of \eqref{cs4}  satisfying \eqref{q1} with $(\beta
_{1,n},\beta _{2,n})\to (\beta _{1},\beta _{2})$. Solving this
question would be helpful for us to understand the structure of
non-topological solutions to the equation \eqref{cs4}.

Let $\{p_{j}\}_{j=1}^{N_{1}}$ be the vortex points of the first equation of %
\eqref{cs4} and $v$ be a solution of \eqref{cs1} with $%
\{p_{j}\}_{j=1}^{N_{1}}$ satisfying
\begin{equation}
v\left( x\right) =-2\left( {\gamma} -N_{1}\right) \ln \left\vert x\right\vert
+O\left( 1\right) \mbox{ as }\left\vert x\right\vert \rightarrow+\infty .
\label{asymp}
\end{equation}%
We remark that  Theorem A implies the existence of $v$ if ${\gamma} > 2N_1+2$ and ${\gamma} \notin
\{2N_{1}\frac{k}{k-1}|k=2,3,\cdot \cdot \cdot ,N_{1}\}$. Now we define the non-degeneracy condition for the equation \eqref{cs1}.
\begin{definition}
\label{D1} A solution $v(x)$ of \eqref{cs1} is called a non-degenerate
solution if the associated linearized operator $L:=\Delta +e^{v(x)}(1-2e^{v(x)})$   has no
zero eigenvalue. In other words, there is no non-trivial kernel of $L$   in $L^\infty(\mathbb{R}^2)$.
\end{definition}

Firstly, we observe the limit behavior for   $(\beta_{1,\varepsilon},\beta_{2,\varepsilon})$
 of a family of bubbling solutions  to  \eqref{cs4} according to the range of $\gamma$.
\begin{theorem}
\label{T0}
   {Let} $\left(
v_{1,\varepsilon},v_{2,\varepsilon}\right) $  {be a family of non-topological solutions}  {to} (\ref{cs4}) satisfying \eqref{q1} with $(\beta_{1,\varepsilon},\beta_{2,\varepsilon})$. Let $(v_{1,\e}(x),v_{2,\varepsilon}(x))\to(V(x),-\infty)$ in $C^0_{\textrm{loc}}(\RN)$, where
a function $V$ satisfies \eqref{asymp}.
We   assume that
  $v_{2,\varepsilon}\left( \frac {x}{\varepsilon}\right) -2\ln \varepsilon\to W(x)\ \textrm{in}\  C^0(\mathbb{R}^2)$ for some function $W$, and $\sup_{x\in\mathbb{R}^2}(e^{v_{i,\varepsilon}(x)}|x|^2)<C\ \textrm{ for some constant} \ C>0, \textrm{independent of}\ \e>0.$
Then  we have

(i) if $\gamma>2N_1+2N_2+4$, then  $\lim_{\e\to0}\left( \beta_{1,\e},\beta_{2,\e}\right)\in\ell_1.$

(ii) if $2N_1+2<\gamma<2N_1+2N_2+4$  and  $\lim_{\e\to0}\int_{\RN}2e^{v_{1,\e}}(1-2e^{v_{1,\e}})dx=4\pi(2N_1+2N_2+4)$,  then  $\lim_{\e\to0}\left( \beta_{1,\e},\beta_{2,\e}\right)\in\ell_2.$
\end{theorem}
In the proof of Theorem \ref{T0},  we observe rigorously how the range of $\gamma$ determine the mass concentration of $v_{1,\e}$.
Since $(v_{1,\e}(x),v_{2,\varepsilon}(x))\to(V(x),-\infty)$ in $C^0_{\textrm{loc}}(\RN)$ and the function $V$ satisfies \eqref{asymp}, we see that   \[
\lim_{\varepsilon\to0}\int_{|x|\le R}2e^{v_{1,\varepsilon}(x)}%
\left(1-2e^{v_{1,\varepsilon}(x)}\right)dx=4\pi{\gamma}+o(1),\] where $o(1)$ as $R\to+\infty$.
We note that $\gamma$ is related to the decay of $e^{v_{1,\varepsilon}}$.
If $\gamma>2N_1+2N_2+4$, then the decay of $e^{v_{1,\varepsilon}}$ is   fast enough at infinity so that the behavior of  $v_{2,\e}$ (blowing up at infinity)
does not affect the total mass of the first component $\int_{\mathbb{R}%
^2}2e^{v_{1,\varepsilon}}(1-2e^{v_{1,\varepsilon}})dx$, although the
asymptotic behavior of $v_{1,\varepsilon}$ at infinity is affected. Thus   $v_{1,\e}$ does not blow up at infinity,  that is,  \begin{equation}  \label{addmass}
\lim_{\varepsilon\to0}\int_{\mathbb{R}^2}2e^{v_{1,\varepsilon}(x)}%
\left(1-2e^{v_{1,\varepsilon}(x)}\right)dx=4\pi{\gamma}.
\end{equation} However, if $2N_1+2<\gamma<2N_1+2N_2+4$, then \eqref{addmass} does not hold any more and the extra mass of $v_{1,\e}$ appears at infinity. Therefore,  not only $v_{2,\e}$ but also $v_{1,\e}$ blow up at infinity. This
observation plays a important role in our analysis for the behavior of
 blowing up at infinity.

Our main goal    is that for any  $\left( \beta_{1},\beta
_{2}\right)\in\ell_1\cup\ell_2,$  we want to construct
 partial blowing up solutions  $\left(
v_{1,\varepsilon},v_{2,\varepsilon}\right) $ to \eqref{cs4}  satisfying   \eqref{q1} with $(\beta_{1,\varepsilon},\beta_{2,\varepsilon})\to \left( \beta_{1},\beta
_{2}\right)\in\ell_1\cup\ell_2$ and   $ v_{2,\varepsilon}\to-\infty$ in $C^0_\textrm{loc}(\RN)$.
    In this paper,   we only construct
   bubbling type I solutions. We will discuss    bubbling type II solutions in another paper \cite{LL}.

For any $(\beta _{1},\beta _{2})\in \ell_1$, we define \begin{equation}\label{ga1}
\gamma_1  :=\frac{2}{3}\left( 2N_{1}+2\beta _{1}+N_{2}+\beta _{2}\right).\end{equation}  We see that $2\beta_1+\beta_2>N_1+2N_2+6$, which implies that
$
{\gamma_1} >2N_{1}+2N_{2}+4.
$
Now  we introduce our main result.
\begin{theorem}
\label{T121}Let $\left( \beta_{1},\beta_{2}\right) \in \ell_1$.  Assume that there exists a    non-degenerate solution $v(x)$ of \eqref{cs1} satisfying
\eqref{asymp} with
$\gamma=\gamma_1$.
Then   there exists
non-topological    solution    $\left(
v_{1,\varepsilon},v_{2,\varepsilon}\right) $ to  \eqref{cs4} with $%
\varepsilon>0$  satisfying $v_{1,\e}(x)\to v(x)-\ln2$ and  $v_{2,\e}(x)\to-\infty$ in $C^0_\textrm{loc}(\RN)$,
$
v_{i,\varepsilon}\left( x\right) =-2\beta_{i,\varepsilon}\ln \left \vert
x\right \vert +O\left( 1\right) \mbox{ as }\left \vert x\right \vert
\rightarrow+\infty,$ where  $\lim_{\varepsilon\to0}\left(
\beta_{1,\varepsilon},\beta_{2,\varepsilon}\right)= \left( \beta_{1},\beta
_{2}\right)\in\ell_1.
$
\end{theorem}We note that when $N_1=1$, any non-topological (radial) solution of \eqref{cs1} is non-degenerate (see
\cite{CFL}). Therefore, Theorem \ref{T121} guarantees the existence of
 non-topological solutions to    \eqref{cs4} under even  some special condition on vortices, which was not considered in \cite{ALW}.
\begin{corollary}
\label{T122}Let $\left( \beta_{1},\beta_{2}\right) \in \ell_1$.  Assume that \begin{equation}N_1=1, \ p_{1}\neq\sum^{N_2}_{i=1}q_i.\label{condiforvortex}\end{equation}
Then  there exists
non-topological    solution    $\left(
v_{1,\varepsilon},v_{2,\varepsilon}\right) $ to  \eqref{cs4} with $%
\varepsilon>0$  satisfying    $v_{2,\e}(x)\to-\infty$ in $C^0_\textrm{loc}(\RN)$,
$
v_{i,\varepsilon}\left( x\right) =-2\beta_{i,\varepsilon}\ln \left \vert
x\right \vert +O\left( 1\right) \mbox{ as }\left \vert x\right \vert
\rightarrow+\infty,$ where  $\lim_{\varepsilon\to0}\left(
\beta_{1,\varepsilon},\beta_{2,\varepsilon}\right)= \left( \beta_{1},\beta
_{2}\right)\in\ell_1.
$\end{corollary}  As far as we know, Corollary \ref{T122} is the only existence result under this particular assumption \eqref{condiforvortex}. For example, in \cite{ALW},  the existence of  non-topological solutions to \eqref{cs4} was proved  when $(\beta_1,\beta_2)$ is closed to $(N_2+2,N_1+2)$ under the following conditions:
\begin{equation}\label{alw_con}\textrm{either}\ \sum_{j=1}^{N_1}p_j=\sum_{j=1}^{N_2}q_j\ \textrm{ or}\ \sum_{j=1}^{N_1}p_j\neq\sum_{j=1}^{N_2}q_j,\ N_1, N_2>1,\ |N_1-N_2|\neq1.\end{equation}

To prove Theorem \ref{T121}, we use different scalings for
different components of the system (\ref{cs4}). To the best of our
knowledge, this idea seems new. In detail,  with a change of variables, we consider the following
functions.
\begin{equation}  \label{1.5}
u_{1,\varepsilon}\left( x\right) :=v_{1,\varepsilon}\left( x\right)\ \ \mbox{ and} \ \
u_{2,\varepsilon}\left( x\right) :=v_{2,\varepsilon}\left( \frac{x}{\varepsilon}\right)
-2\ln \varepsilon.
\end{equation}
Then $(v_{1,\varepsilon}(x),v_{2,\varepsilon}(x))$ is a solution of (\ref{cs4}) if and only if $\left(
u_{1,\varepsilon},u_{2,\varepsilon}\right) $ is a solution of
\begin{equation}
\left\{
\begin{array}{l}
\Delta u_{1,\varepsilon}+2
e^{u_{1,\varepsilon}}\left(1-2e^{u_{1,\varepsilon}}\right)  =\varepsilon^{2}
e^{u_{2,\varepsilon}\left( \varepsilon x\right)
}\left(1-2\varepsilon^{2}e^{u_{2,\varepsilon}\left( \varepsilon x\right)
}-e^{u_{1,\varepsilon}(x) }\right)
+4\pi \sum_{j=1}^{N_{1}}\delta_{p_{j}},  \label{2} \\
\Delta u_{2,\varepsilon}+2 e^{u_{2,\varepsilon}}\left(1-2\varepsilon
^{2}e^{u_{2,\varepsilon}}\right)  = \frac{ e^{u_{1,\varepsilon}\left( \frac{x%
}{\varepsilon}\right) }(1-2e^{u_{1,\varepsilon}\left( \frac{x}{\varepsilon}%
\right) })}{\varepsilon^{2}} -e^{u_{1,\varepsilon}\left( \frac{x}{\varepsilon}%
\right) +u_{2,\varepsilon}} +4\pi \sum_{j=1}^{N_{2}}\delta_{\varepsilon
q_{j}}.
\end{array}%
\right.
\end{equation}%
Each equation in the  system (\ref{2}) can be
viewed as a perturbation of the following Chern-Simons-Higgs equation and Liouville
equation respectively ($V(x)=v(x)+\ln\frac{1}{2}$):%
\begin{equation}
\left \{
\begin{array}{l}
\Delta V+2e^{V}\left(1 -2e^{V}\right) =4\pi \sum_{j=1}^{N_{1}}\delta_{p_{j}},
\\
\int_{\mathbb{R}^2}2e^{V}\left(1 -2e^{V}\right) dx=4\pi {{\gamma_1}},%
\end{array}
\right.  \label{5}
\end{equation}
\begin{equation}
\left \{
\begin{array}{l}
\Delta W+2e^{W}=4\pi \left( \frac{{{\gamma_1}}}{2}+N_{2}\right) \delta_{0}, \\
\int_{\mathbb{R}^2}2e^{W}dx=8\pi(1+ \frac{{{\gamma_1}}}{2}+N_{2}),%
\end{array}
\right.  \label{6}
\end{equation}
where the extra singular source $\frac{{{\gamma_1}}}{2}\delta_0$ in the second equation appears  due to the
first component. It is well known that the solutions to (\ref{6}) have been
completely classified by Prajapat and Tarantello in \cite{PT} such that%
\begin{equation}
W\left( x\right) =W_{\lambda_1,a}\left( x\right) =\ln \frac{4\left( 1+\frac{%
{{\gamma_1}}}{2}+N_{2}\right) ^{2}e^{\lambda_1}\left \vert x\right \vert
^{{{\gamma_1}}+2N_{2}}}{\left( 1+e^{\lambda_1}\left \vert x^{1+\frac{{{\gamma_1}}}{2}%
+N_{2}}+a\right \vert ^{2}\right) ^{2}},  \label{17}
\end{equation}
where $\lambda_1 \in \mathbb{R}$,  $x =|x|(\cos\theta+i\sin\theta),\  a=a_{1} + i a_{2} \in \mathbb{C}$.  We denote $W_a(x):=W_{0,a}(x)$.

Since  the term  $\frac{  e^{u_{1,\varepsilon}\left( \frac{x}{\varepsilon}%
\right) }\left(1-2e^{u_{1,\varepsilon}\left( \frac{x}{\varepsilon}\right) }\right)}{
\varepsilon^{2}}$ in   \eqref{2}
cannot be simply replaced by a Dirac measure, we
 consider the following   as in \cite{LY}:
\begin{equation}\label{u2epstar}
 \Delta u_{2,\varepsilon}^{*}\left( x\right) =\frac{e^{V(\frac {x}{%
\varepsilon})}(1-2e^{V(\frac{x}{\varepsilon})})}{\varepsilon^{2}}.
\end{equation}
From Lemma \ref{diffgreen}, we have for any constant $R, r>0$, independent of $\varepsilon>0$,
\begin{equation*}
\Vert u_{2,\varepsilon}^{*}(x)-{{\gamma_1}} \ln|x|\Vert_{L^\infty(B_R(0)\setminus B_r(0))}=O(\varepsilon).
\end{equation*}
Due to the term $u_{2,\varepsilon}^{*}(x)$,  there are errors coming from not only the term $%
\prod_{j=1}^{N_2}|x-\varepsilon q_j|^2-|x|^{2N_2}$, but also the term $%
u_{2,\varepsilon}^*-{{\gamma_1}}\ln|x|$.  Indeed, the term $%
u_{2,\varepsilon}^*-{{\gamma_1}}\ln|x|$  gives us the errors including $-\frac{x\cdot \overrightarrow{c}}{|x|^2}\e$ where $\overrightarrow{c}:=\frac{1}{2\pi}\int_{\mathbb{R}^2}
ye^{V(y)}(1-2e^{V(y)})dy$. This is also one of the differences from
the previous works (for example, see \cite{ALW}). From this point of view,
by a shift of origin, without loss of generality, we always assume that
\begin{equation}  \label{sumqeqc}
2\sum^{N_2}_{i=1}q_i+\overrightarrow{c}=0,
\end{equation} to vanish $\varepsilon$-order term in the error
and get a suitable reduced problem.  Even when $N_2=0$,   \eqref{sumqeqc} is still satisfied (see the detail in Section \ref{sectionb}).

As we mentioned before, if ${\gamma_1}>2N_1+2N_2+4$, then  the term $\varepsilon^{2}e^{u_{2,\varepsilon}\left(
\varepsilon x\right) }$ in the first equation of \eqref{2} does not affect the mass of the first component $\int_{\mathbb{R}%
^2}2e^{u_{1,\varepsilon}}(1-2e^{u_{1,\varepsilon}})dx$,  but  the
asymptotic behavior of $u_{1,\varepsilon}$ at infinity is affected by $\varepsilon^{2}e^{u_{2,\varepsilon}\left(
\varepsilon x\right) }$.    So we  consider the following problem: \begin{equation}\label{uqepstar}
\Delta u_{1,\varepsilon}^*=\varepsilon^2e^{W_a(\varepsilon x)}.
\end{equation}
Denote   the regular part of  $W_a$ by $\tilde{W}_a$ and let   $u_{1,\varepsilon}^*(x)=\frac{\tilde{W}_a(0)-\tilde{W}_a(\varepsilon x)}{2}$. Then $u_{1,\e}^*$ controls the effect of  $\varepsilon^2e^{W_a(\varepsilon x)}$ although   $u_{1,\e}^*$ itself is small enough on any bounded subset in $\RN$.
From the above arguments, at this moment,  we write an 0-th order approximate solutions for bubbling type I solutions to \eqref{cs4} such that  \begin{equation}\label{apfortypeI}(V(x)+u_{1,\e}^*(x),\tilde{W}_a(\e x)+2\ln\e+ 2\sum_{j=1}^{N_2}\ln|\e x-\e q_j|+u_{2,\e}^*(\e x)).\end{equation}

However, the above form \eqref{apfortypeI} is not sufficiently accurate to carry out the next further steps. In order to improve the approximate solution to the next order,  we have to study the invertibility of
the linearized operator $\mathbb{QL}_{\varepsilon}^a$  (see \eqref{231}), which is most important part in our paper. The linearized operator $\mathbb{QL}_{\e}^a$ is defined from $%
X^{1}_{\alpha}\times E_{\alpha}$ to $Y_{\alpha}\times
F_{\alpha}$, where $X^{1}_{\alpha}$, $E_{\alpha} (\subseteq
X^{2}_{\alpha} )$, $Y_{\alpha}$, $F_{\alpha}$ are some suitable
weighted function spaces defined in Section 3 (see Definition \ref{D3}). For our linearized problem, we need to  use the different function
spaces $X^1_\alpha$ and $X^2_\alpha$ for different component because a limiting problem for the
first component of   (\ref{2}) is Chern-Simons-Higgs
equation, but a limiting problem for the second component of \eqref{2} is
Liouville equation.
It is worth to note that any function with logarithmic growth at infinity cannot
belong to $X^1_\alpha$ unlike $X^2_\alpha$. The non-degeneracy assumption about the linearized equation of \eqref{cs1} requires that any function with logarithmic growth does not allow to be in the function
space for the first component (see \cite[Theorem 4.1]{CFL}). On the other hand, the function space  for the second component should contain functions with logarithmic growth (see \cite[Lemma 2.2]{CI}).

In order to prove the invertibility of the linearized
operator $\mathbb{QL}_{\e}^a$, there is a major difficulty which is
caused by the different scaling factors for different   components  in \eqref{2}. We overcome this obstacle by using     the Green  representation formula and the property of $X^1_\alpha$.
We   note that the analysis in this paper  would provide an important insight for the studies of  mixed type entire solution in \cite{CKL4, FLL} and  blow up phenomena in various models (for example, see \cite{HL0, HLY, LLTY, LY3, LWZ}).

The organization of this paper is as follows: In Section 2, we prove Theorem \ref{T0}. In Section 3, we
construct  an approximate solution with a suitable change of variables.  In Section 4, we use
contraction mapping theorem,  solve a finite dimensional reduced problem, and complete the proof of Theorem \ref{T121} and Corollary
\ref{T122}.  In Section 5, we prove the invertibility of $\mathbb{QL}%
_{\e}^a$, which is the main part
in this paper. Finally,  we  put some basic estimates in Section \ref{sectionb}.
\section{Proof of Theorem \ref{T0}}
The arguments here will be a motivation to  construct approximate solutions for Theorem \ref{T121}.

\textbf{Proof of Theorem \ref{T0}} Since $\left(
v_{1,\varepsilon},v_{2,\varepsilon}\right) $ satisfies the system
\eqref{cs4}, the assumptions $v_{1,\varepsilon}(x)\to V(x)$ and $v_{2,\varepsilon}\to-\infty$ in $C_{\textrm{loc}}^0(\mathbb{R}^2)$ imply that   $V(x)$ satisfies
  \begin{equation*}
\begin{aligned}\left\{ \begin{array}{ll}  &\Delta V+2
e^{V}\left(1-2e^{V}\right) =4\pi \sum_{j=1}^{N_{1}}\delta_{p_{j}},
\\&\int_{\mathbb{R}^2}2e^{V(x)}\left(1 -2e^{V(x)}\right)dx=4\pi\gamma, \ \gamma>2N_1+2. \end{array}\right. \end{aligned}
\end{equation*}
Recall that  $(u_{1,\varepsilon}\left( x\right),u_{2,\varepsilon}\left( x\right)
)=(v_{1,\varepsilon}\left( x\right), v_{2,\varepsilon}\left( \frac {x}{\varepsilon}\right) -2\ln \varepsilon)$ satifies \eqref{2},  the assumption in Theorem \ref{T0} implies that $(u_{1,\varepsilon} ,u_{2,\varepsilon}
) \to(V ,W )$.

Now we claim that there exists constant $C>0$, independent of $\varepsilon>0$, such that \begin{equation}\int_{\mathbb{R}^2}e^{u_{i,\varepsilon}(x)}dx\le C.\label{eubdd}\end{equation}
After multiplying 2 and 1 (resp. 1 and 2) on both sides of the first and
second (resp. the second and first) equation in \eqref{2}, adding new
equations respectively, and taking the integration on $\mathbb{R}^2$, we get
that
\begin{equation}  \label{1flux}
\begin{aligned}\left\{ \begin{array}{ll}
\lim_{\varepsilon\to0}\int_{\mathbb{R}^2} 2e^{u_{1,\varepsilon}(x)}\left( 1
-2e^{u_{1,\varepsilon}(x)}+\varepsilon^2e^{u_{2,\varepsilon}(\varepsilon
x)}\right) dx=\frac{8\pi}{3}\left(
2N_{1}+2\beta_{1}+ N_{2}+ \beta_{2}\right),\\ \lim_{\varepsilon\to0}\int_{\mathbb{R}^2}
2e^{u_{2,\varepsilon}(x)}\left( 1
-2\varepsilon^2e^{u_{2,\varepsilon}(x)}+e^{u_{1,\varepsilon}(
x/\varepsilon)}\right) dx=\frac{8\pi}{3}\left(
N_{1}+\beta_{1}+2N_{2}+2\beta_{2}\right). \end{array}\right. \end{aligned}
\end{equation}
By the assumption $\sup_{x\in\mathbb{R}^2}(e^{v_{i,\varepsilon}(x)}|x|^2)<C$,
we see that if $|x| \gg1$, \begin{equation}  \label{flux00} (1-2e^{u_{1,\varepsilon}(x)}+\varepsilon^2e^{u_{2,\varepsilon}(\varepsilon x)}), \   (1-2\varepsilon^2e^{u_{2,\varepsilon}(x)}+e^{u_{1,\varepsilon}(x/\varepsilon)})>\frac{1}{2}.\end{equation}
From $(u_{1,\varepsilon}\left( x\right),u_{2,\varepsilon}\left( x\right)
) \to(V(x),W(x))$ and \eqref{1flux}-\eqref{flux00},  we prove  the claim \eqref{eubdd}.

Let \begin{equation}\label{defofft}f(t):=2e^t(1-2e^t)\ \textrm{for}\ t\in\mathbb{R}\ \textrm{and}\ \ \bar{u}_{1,\varepsilon}(|x|):=\frac{1}{2\pi|x|}\int_{\partial B_{|x|}}u_{1,\varepsilon}dS.\end{equation}
If $1\ll R <|x|< \frac{\delta}{\varepsilon}$ where  $R, \delta>0$ are   independent of $\e>0$,  then we see that as $\varepsilon\to0$,
 \begin{equation*}
\begin{aligned}
&2\pi|x|\Big(\frac{d\bar{u}_{1,\varepsilon}}{d|x|}\Big)
=-\int_{B_{|x|}(0)}f(u_{1,\varepsilon}(y)) dy +\int_{B_{\varepsilon|x|}(0)}e^{u_{2,\varepsilon}\left(
y\right) }\left( 1 -2\varepsilon^{2}e^{u_{2,\varepsilon}\left(
 y\right) }-e^{u_{1,\varepsilon}(y/\varepsilon)}\right)dy +4\pi N_1
 \\&\le-\int_{B_{R}(0)}f(u_{1,\varepsilon}(y)) dy+\int_{B_{\delta}(0)}e^{u_{2,\varepsilon}}dy +4\pi N_1
 \\&\to-\int_{B_{R}(0)}f(V(y))dy +O(\delta^2)+4\pi N_1=-4\pi\gamma +O(R^{-2})+O(\delta^2)+4\pi N_1.
  \end{aligned}
\end{equation*}
Since $\gamma>2N_1+2$, we get that   if $R<|x|< \frac{\delta}{\varepsilon}$, then \begin{equation}
\begin{aligned}
& |x|\Big(\frac{d\bar{u}_{1,\varepsilon}}{d|x|}\Big)
\le -(2+\nu)\ \textrm{for some constant} \ \nu>0,\ \textrm{independent of}\ \varepsilon>0.
  \end{aligned}\label{average}
\end{equation}
In view of Green representation formula (for example, see Lemma \ref{w}), there is a constant $c_{1,\varepsilon}$ satisfying $
u_{1,\varepsilon}(x)
=c_{1,\varepsilon}+ \frac{1}{2\pi}\int_{\mathbb{R}^2}\ln|x-y|\Delta u_{1,\varepsilon}(y)dy.
$
Note that if $|x|=|x'|>R$, then
\begin{equation*}\begin{aligned}
&u_{1,\varepsilon}(x)-u_{1,\varepsilon}(x')
\\&=2\sum_{i=1}^{N_1}\ln\frac{|x-p_{j}|}{|x'-p_{j}|}
 +\frac{1}{2\pi}\int_{\mathbb{R}^2}\ln\frac{|x'-y|}{|x-y|}\Big( f(u_{1,\varepsilon} )-\varepsilon^{2}e^{u_{2,\varepsilon}(
\varepsilon y) }( 1 -2\varepsilon^{2}e^{u_{2,\varepsilon}(
\varepsilon y) }-e^{u_{1,\varepsilon}})\Big)dy
\\&=
O\Big(\int_{|y|\le\frac{|x|}{2},|y|\ge2|x|}\Big| f(u_{1,\varepsilon} )-\varepsilon^{2}e^{u_{2,\varepsilon}(
\varepsilon y) }( 1 -2\varepsilon^{2}e^{u_{2,\varepsilon}(
\varepsilon y) }-e^{u_{1,\varepsilon}})\Big|dy\Big)
\\&+\Big(\max_{\frac{|x|}{2}\le|y|\le 2|x|}e^{u_{1,\varepsilon}(y)}+\max_{\frac{|x|}{2}\le|y|\le 2|x|}\varepsilon^{2}e^{u_{2,\varepsilon}(
\varepsilon y) }\Big)\int_{\frac{|x|}{2}\le|y|\le 2|x|}\left(\frac{|x-x'|}{|x-y|}+\frac{|x-x'|}{|x'-y|}\right)dy+O(1).
\end{aligned}\end{equation*}
Since   $\int_{\mathbb{R}^2}e^{u_{i,\varepsilon}(x)}dx<C$ and $\sup_{x\in\mathbb{R}^2}(e^{u_{i,\varepsilon}(x)}|x|^2)<C\ \textrm{for}\ i=1,2,$
if $|x|=|x'|>R$, then
\begin{equation}\label{O}u_{1,\varepsilon}(x)-u_{1,\varepsilon}(x')=O(1).\end{equation}
In view of \eqref{average}-\eqref{O}, we see that  for any fixed small $\delta>0$,
\begin{equation*}
\begin{aligned}&\lim_{\varepsilon\to0}\int_{B_\delta(0)}\frac{f(u_{1,\varepsilon}\left(\frac{x}{\varepsilon}\right) )}{2\varepsilon^{2}}dx=\lim_{\varepsilon\to0}\int_{B_{\frac{\delta}{\varepsilon}}(0)}\frac{f(u_{1,\varepsilon}(x))}{2} dx
\\&=\lim_{\varepsilon\to0}\Big[\int_{B_{R}(0)}\frac{f(u_{1,\varepsilon}(x))}{2}  dx+O\Big(\int_{B_{\frac{\delta}{\varepsilon}}(0)\setminus B_{R}(0)}e^{u_{1,\varepsilon}(x)-\bar{u}_{1,\varepsilon}(x)+\bar{u}_{1,\varepsilon}(x)} dx\Big)\Big]
\\&=\int_{B_{R}(0)}\frac{f(V(x))}{2} dx+O(R^{-\nu})\to2\pi\gamma\ \ \ \textrm{as}\ \ R\to+\infty.
\end{aligned}
\end{equation*}
Then  $u_{2,\varepsilon}\to W$ in $C^0(\mathbb{R}^2)$ implies that  $W$ would satisfy the following equation:
\begin{equation*}
\begin{aligned}
\Delta W+2e^{W(x)}=4\pi \left( N_2+ \frac{\gamma}{2}\right) \delta_{0},\ \ e^W\in
L^1(\mathbb{R}^2).   \end{aligned}
\end{equation*}
Since $W$ satisfies
Liouville equation, we have $\int_{\mathbb{R}^2}2e^{W(x)}dx=8\pi\Big(N_2+\frac{{{\gamma}}}{2}+1\Big).$
We note that  although the component $\varepsilon^{2}e^{u_{2,\varepsilon}\left( \varepsilon
x\right) }\rightarrow0$ in $C^0_{\text{loc}}(\mathbb{R}^2)$ as $\varepsilon\to0
$ in the first equation in \eqref{2},   the total mass $\varepsilon^{2}e^{u_{2,\varepsilon}\left( \varepsilon
x\right) }$ in $\mathbb{R}^2$  converges to a positive value such that
\begin{equation}\label{positivemassatinfty}
\lim_{\varepsilon\to0}\int_{\mathbb{R}^{2}}\varepsilon^{2}e^{u_{2,\varepsilon}\left( \varepsilon
x\right) }dx=\lim_{\varepsilon\to0}\int_{\mathbb{R}^{2}}e^{u_{2,\varepsilon}\left( x\right) }dx=\intr e^{W(x)}dx=
4\pi\Big(N_2+\frac{{{\gamma}}}{2}+1\Big)>0,
\end{equation}
since  $u_{2,\e}(x)=v_{2,\varepsilon}\left( \frac {x}{\varepsilon}\right) -2\ln \varepsilon\to W(x)\ \textrm{in}\  C^0(\mathbb{R}^2)$. Here \eqref{positivemassatinfty} implies that  the mass of   $e^{v_{2,\e}(x)}=\varepsilon^{2}e^{u_{2,\varepsilon}\left( \varepsilon x\right) }$
concentrates at infinity.

(i) We consider the case  ${{\gamma}}>2N_1+2N_2+4$.
 Then we see that  for any large $|x|> R$, as $\varepsilon\to0$,
 \begin{equation*}
\begin{aligned}
 &2\pi{|x|}\Big(\frac{d\bar{u}_{1,\varepsilon}}{d{|x|}}\Big)
 =-\int_{B_{|x|}(0)}f(u_{1,\varepsilon} ) dx +\int_{B_{\varepsilon |x|}(0)}e^{u_{2,\varepsilon} (
x ) } ( 1 -2\varepsilon^{2}e^{u_{2,\varepsilon} (
 x ) }-e^{u_{1,\varepsilon}(x/\varepsilon)} )dx +4\pi N_1,
 \\&\le-\int_{B_{R}(0)}f(u_{1,\varepsilon}) dx+4\pi(N_2+\frac{{{\gamma}}}{2}+1) +4\pi N_1
  \to-\int_{B_{R}(0)}f(V)dx +4\pi(N_1+ N_2+\frac{{{\gamma}}}{2}+1)
\\&=4\pi(N_1+ N_2-\frac{{{\gamma}}}{2}+1)+O(R^{-2})<-4\pi.
  \end{aligned}
\end{equation*}
Therefore, the decay of $e^{u_{1,\varepsilon}(x)}=O(|x|^{-(2 +\nu') })$ for some $\nu'>0$ is fast enough at infinity so that the concentration of $\varepsilon^{2}e^{u_{2,\varepsilon}\left( \varepsilon x\right) }$ at infinity
does not affect the mass of  $\int_{\mathbb{R}%
^2}f(u_{1,\varepsilon})dx$.
Thus we  see
\begin{equation}  \label{flux2}
\begin{aligned}\left\{ \begin{array}{ll}
\lim_{\varepsilon\to0}\int_{\mathbb{R}^2} 2e^{u_{1,\varepsilon}(x)}\left( 1
-2e^{u_{1,\varepsilon}(x)}+\varepsilon^2e^{u_{2,\varepsilon}(\varepsilon
x)}\right) dx=\int_{\mathbb{R}^2}f(V  )dx=4\pi\gamma,\\ \lim_{\varepsilon\to0}\int_{\mathbb{R}^2}
2e^{u_{2,\varepsilon}(x)}\left( 1
-2\varepsilon^2e^{u_{2,\varepsilon}(x)}+e^{u_{1,\varepsilon}(
x/\varepsilon)}\right) dx =\int_{\mathbb{R}^2}2e^{W  }dx=8\pi (N_2+\frac{\gamma}{2}+1 ). \end{array}\right. \end{aligned}
\end{equation}
By  \eqref{1flux},   \eqref{flux2}, and  $u_{2,\varepsilon}\to W$ in $C^0(\mathbb{R}^2)$, we get that
\begin{equation}  \label{magofgamma}
\begin{aligned}\left\{ \begin{array}{ll}
{{\gamma}}= \frac{2}{3}\left(
2N_{1}+2\beta_{1}+N_{2}+\beta_{2}\right)>2N_1+2N_2+4,\\ \beta_{2}-\beta_{1}=N_{1}+2N_{2}+3, \ \beta_i>1. \end{array}\right. \end{aligned}
\end{equation}
(ii) If $2N_1+2<{{\gamma}}<2N_1+2N_2+4$, then  the assumption $\lim_{\e\to0}\int_{\RN}f(v_{1,\e})dx=4\pi(2N_1+2N_2+4)$ and
 \eqref{1flux}  yield
$2\beta_{1}+\beta_{2}=N_{1}+2N_{2}+6, \ \beta_i>1.$
This completes the proof.\hfill$\square$
\section{Approximate solution and Linearized operator}

\subsection{0-th order approximate solution}\label{0th}
In view of the assumption of Theorem \ref{T121}, we choose a non-degenerated solution $V$ of the  following equation
  \begin{equation}\label{chernsimonseq}
\begin{aligned}\left\{ \begin{array}{ll}  &\Delta V+f(V(x)) =4\pi \sum_{j=1}^{N_{1}}\delta_{p_{j}},
\\&\int_{\mathbb{R}^2}f(V(x))dx=4\pi\gamma_1,\ \ V(x)=(-2\gamma_1+2N_1)\ln|x|+O(1)\ \ \textrm{as}\ \ |x|\to+\infty,\end{array}\right.\end{aligned}
\end{equation}where $\gamma_1   =\frac{2}{3}\left( 2N_{1}+2\beta _{1}+N_{2}+\beta _{2}\right) >2N_1+2N_2+4$ (see \eqref{ga1}). \\
 We also consider the following  Liouville equation: \begin{equation} \label{eqliou}
\begin{aligned}
\Delta W+2e^{W(x)}=4\pi \left(N_{2}+ \frac{{\gamma_1}}{2}\right) \delta_{0},\ \ e^W\in
L^1(\mathbb{R}^2).   \end{aligned}
\end{equation} As mentioned in the introduction (see \eqref{17}), we have
\begin{equation*}
W\left( x\right) =W_{\lambda_1,a}\left( x\right) =\ln \frac{4\left( 1+\frac{%
{\gamma_1}}{2}+N_{2}\right) ^{2}e^{\lambda_1}\left \vert x\right \vert
^{{\gamma_1}+2N_{2}}}{\left( 1+e^{\lambda_1}\left \vert x^{1+\frac{{\gamma_1}}{2}%
+N_{2}}+a\right \vert ^{2}\right) ^{2}}.
\end{equation*}   We choose
 $W\left( x\right) =W_{0,a}\left( x\right) $ (denoted by $W_{a}\left(
x\right) $) where $a=  a_{1}+i a_{2}$ would be determined later.
For the regular part for $W_a$, we use the following notation:  \begin{equation}
 \tilde{W}_a(x):=W_a(x)-({\gamma_1}+2N_2)\ln|x|.\label{regularpart}
\end{equation}
As mentioned in the introduction (see \eqref{u2epstar} and \eqref{uqepstar}),  we need to   consider the following problem: \begin{equation}\label{dirac}
\Delta u_{1,\varepsilon}^*=\varepsilon^2e^{W_a(\varepsilon x)}\quad\mbox{and}%
\quad \Delta u_{2,\varepsilon}^*=\frac{f(V(x/\varepsilon))}{2\varepsilon^2},
\end{equation}
Let
$
u_{1,\varepsilon}^*(x):=\frac{\tilde{W}_a(0)-\tilde{W}_a(\varepsilon x)}{2},$ and
\begin{equation}\label{definitionofu2e}
\begin{aligned}u_{2,\varepsilon}^*(x):=\left\{ \begin{array}{ll} -\frac{V ( \frac{x}{\varepsilon} )-2\ln|\frac{x}{\varepsilon}-p_1|}{2}+
\gamma_1\ln\varepsilon+\frac{I_{p_1}}{2} \ \textrm{if}\ \ N_1=1,
\\ \frac{1}{2\pi}\int_{\mathbb{R}^2}\ln|x-y|\frac{
e^{V(y/\varepsilon)}(1-2e^{V(y/\varepsilon)})}{\varepsilon^2}dy\ \ \textrm{if}\ \ N_1\ge2,\end{array}\right. \end{aligned}
\end{equation}
where  if $N_1=1$, then we choose  a (non-degenerate)  solution $V$  of \eqref{chernsimonseq} which is a radial symmetric with respect to $p_1$ and satisfies
$V(x)=(-2\gamma_1+2)\ln|x -p_{1} |+I_{p_1}+O(|x -p_1|^{-\sigma})$ for $|x |\gg1$ (see \cite[Lemma 2.6]{CFL}) where $\sigma>0$.   Here  we define $u_{2,\e}^*$ depending on $N_1$ to make  $|y|^6|f(V(y)|$ integrable in the proof of Lemma \ref{diffgreen}.
We note that  $u_{i,\varepsilon}^*$ satisfies \eqref{dirac} for $i=1,2$.  Moreover,  by the definition of $\tilde{W}_a$, we see that $u_{1,\varepsilon}^*\to0\ \ \textrm{in}\ \  C^0_{\textrm{loc}}(\mathbb{R}^2).$
We also note that   Lemma \ref{diffgreen} gives us $\|u_{2,\varepsilon}^*-{\gamma_1}%
\ln|x|\|_{C_{\text{loc}}(\mathbb{R}^2\setminus\{0\})}=O(\varepsilon).$

Now we define the $0$-th order approximate solution  for bubbling type I to \eqref{cs4}  such that \begin{equation}
\begin{aligned}\left\{ \begin{array}{ll}\label{approximatesolutionforbubblingI}
U_{1,\varepsilon}^{a}\left( x\right) & :=V\left( x\right)
 +u_{1,\varepsilon}^*(x) =V\left( x\right) +\frac{1}{2}\left( \tilde{W}%
_{a}\left( 0\right) -\tilde{W}_{a}\left( \varepsilon x\right) \right), \\
U_{2,\varepsilon}^{a}\left( x\right) & :=\tilde{W}_a\left(\e  x\right)+2\ln\e
+2\sum_{j=1}^{N_{2}}\ln \left \vert \e x-\varepsilon q_{j}\right \vert
+u_{2,\varepsilon}^{*}(\e x).
\end{array}\right. \end{aligned}
\end{equation}
By  the definition of ${\gamma_1}$, $\beta_2-\beta_1=N_1+2N_2+3$, and Lemma \ref{diffgreen},  it is easy to see that
\begin{equation}\label{rmk0}
\begin{aligned}\left\{ \begin{array}{ll} U_{1,\varepsilon}^{a}\left(
x\right) &=-2\beta_1\ln|x|+O(1)
\ \mbox{ at }\infty, \\ U_{2,\varepsilon}^{a}\left( x\right) &=-2\beta_2\ln|x|+O(1)
\ \mbox{ at }\infty.
\end{array}\right. \end{aligned}
\end{equation}
However,   as mentioned in the introduction,  we need to improve the approximate solution to the  next    order such that   \[\left( V_{1,\varepsilon}^{a}, V_{2,\varepsilon}^{a}\right)=\left( U_{1,\varepsilon}^{a}+\phi_{1,\varepsilon}^a,U_{2,\varepsilon}^{a}+\phi_{2,\varepsilon}^a\right),\]
where $(\phi^a_{1,\e},\phi^a_{2,\e})$ satisfies \eqref{26}.  Then
we have to study
the invertibility of the linearized operator, which is the most important part in our paper.
\subsection{Linearized operator}\label{linearizedoperator}
In order to define the suitable function space for the linearized operator,   let
\begin{equation}
\begin{aligned}\left\{ \begin{array}{ll}\label{z0}
Z_{0}(x):=\frac{\partial W_{\lambda_1,a}}{\partial \lambda_1}\Big|_{\lambda_1=0,a=0}
=\frac{1- |x|^{2N_{2}+\gamma_1+2}}{1+|x|^{2N_{2}+\gamma_1+2}},
\\ Z_{1}(x):=-\frac{1}{4}\frac{\partial W_{\lambda_1,a}}{\partial a_1}\Big|_{\lambda_1=0,a=0}
=\frac{ |x|^{N_{2}+\frac{{\gamma_1}}{2}+1}\cos[(N_{2}+\frac{{\gamma_1}}{2}%
+1)\theta] }{1+|x|^{2N_{2}+\gamma_1+2}},
\\ Z_{2}(x):=-\frac{1}{4}\frac{\partial W_{\lambda_1,a}}{\partial a_2}\Big|_{\lambda_1=0,a=0}
=\frac{ |x|^{N_{2}+\frac{{\gamma_1}}{2}+1}\sin[(N_{2}+\frac{{\gamma_1}}{2}%
+1)\theta] }{1+|x|^{2N_{2}+\gamma_1+2}}.
\end{array}\right. \end{aligned}
\end{equation}
 We denote by $\mathbb{N}$ the set of positive integers, and introduce weighted function spaces as follows:
\begin{definition}
\label{D3}For any $0<\alpha<\min\left\{2\Big(N_2+\frac{\gamma_1-2}{2}\Big),\frac{1}{3}\right\}$,

(a) A function $w(x)$ is said to be in $X_{\alpha}^{i}$ if
\begin{equation*}
\begin{aligned} \|w \|^2_{X^i_\alpha}= \int_{\mathbb{R}^2}|\Delta
w|^2(1+|x|)^{2+\alpha}+(w(x)\rho_i(x))^2dx<+\infty, \  \ i=1,2\end{aligned}.
\end{equation*}where $
\rho_{1}(x):=(1+|x|)^{-1}\big(\ln(2+|x|)\big)^{-1-\frac{\alpha}{2}},\ \rho_{2}(x):=(1+|x|)^{-1-\frac{\alpha}{2}}.$

(b) A function $h(x)$ is said to be in $Y_{\alpha}$ if
$$\|h \|_{Y_\alpha}^2 =\int_{\mathbb{R}^2}
h^2(x)(1+|x|)^{2+\alpha}dx<+\infty.$$

(c) $
{E}_{\alpha}:= \left \{
\begin{array}{l}
\{\eta \in X_{\alpha}^{2}\ |\ \int_{\mathbb{R}^{2}}(-\Delta
Z_{i}+2e^{W_{0}}Z_{i})\eta dx=0,\ i=0,1,2\} \ \ \mbox{if}\ \frac{{\gamma_1}}{2}%
\in\mathbb{N}, \\
\{\eta \in X_{\alpha}^{2}\ |\ \int_{\mathbb{R}^{2}}(-\Delta
Z_{0}+2e^{W_{0}}Z_{0})\eta dx=0\} \ \ \mbox{if}\ \frac{{\gamma_1}}{2}\notin%
\mathbb{N}.%
\end{array}
\right.
$

(d) $
{F}_{\alpha}:= \left \{
\begin{array}{l}
\{h\in Y_{\alpha}\ |\ \int_{\mathbb{R}^{2}}hZ_{i}dx=0,\ \ i=1,2\} \ \ %
\mbox{if}\ \frac{{\gamma_1}}{2}\in\mathbb{N}, \\
Y_{\alpha} \ \ \mbox{if}\ \frac{{\gamma_1}}{2}\notin\mathbb{N}.%
\end{array}
\right.
$
\end{definition}

Obviously,   $X_{\alpha}^{i}$ $(i=1,2)$,  $Y_{\alpha}$ are Hilbert
spaces. Here we see that $Z_{0},\ Z_{1},\ Z_{2}\in X^2_\alpha$. In  \cite{CI,CL,DEM}, it has been known that the kernel of  $\mathfrak{L}_0 u=\Delta u+2e^{W_0(x)}u$ in $X_\alpha^2$ is
\begin{equation}
\left \{\label{kerinx2}
\begin{array}{l}
\mbox{ker}\ \mathfrak{L}_0= \mbox{Span} \{Z_{0}, Z_{1}, Z_{2}\} \ \ \mbox{if}\ \frac{{\gamma_1}}{2}\in\mathbb{N}, \\
\mbox{ker}\ \mathfrak{L}_0= \mbox{Span} \{Z_{0}\} \ \ \mbox{if}\ \frac{{\gamma_1}}{2}\notin\mathbb{N}.
\end{array}
\right.
\end{equation}Since $\gamma_1>2N_1+2$, we also see that  if $0<\alpha<\min\left\{2\Big(N_2+\frac{\gamma_1-2}{2}\Big),\frac{1}{3}\right\}$, then   $Z_{1},\ Z_{2}\in Y_\alpha$, but $Z_{0}\notin Y_\alpha$.

Now we define the
operator ${Q}:Y_{\alpha}\rightarrow {F}_{\alpha}$ by%
\begin{equation}\label{proj_q}
{Q}h:= \left \{
\begin{array}{l}
h-\sum_{i=1}^{2}c_{i}Z_{i},\ \mbox{where}\  \{c_{i}\}_{i=1,2}\ \mbox{are chosen so that}\ {Q}h\in F_\alpha\
\mbox{if}\ \frac{{\gamma_1}}{2}\in\mathbb{N}, \\
h \ \ \mbox{if}\ \frac{{\gamma_1}}{2}\notin\mathbb{N}.%
\end{array}
\right.
\end{equation}
Define
$
\mathbb{QL}_{\e}^{a}:X_{\alpha}^{1}\times
{E}_{\alpha}\rightarrow Y_{\alpha}\times {F}_{\alpha}
$
by%
\begin{equation}\label{231}
\mathbb{QL}_{\varepsilon}^{a}(w_{1},w_{2}):=\left( \mathbb{L}%
_{\varepsilon}^{a,1}\left( w_{1},w_{2}\right) ,{Q}\mathbb{L}%
_{\varepsilon}^{a,2}\left( w_{1},w_{2}\right) \right),
\end{equation} and $
\mathbb{L}_{\varepsilon}^{a}\left( w_{1},w_{2}\right) :=\left( \mathbb{%
L}_{\e}^{a,1}\left( w_{1},w_{2}\right) ,\mathbb{L}%
_{\e}^{a,2}\left( w_{1},w_{2}\right) \right),
$ where
\begin{equation}
\begin{aligned}\left\{ \begin{array}{ll}\label{241}
\mathbb{L}_{\e}^{a,1}\left( w_{1},w_{2}\right) :=\Delta w
_{1}\left( x\right) +f^{\prime}\left( V\left( x\right) \right)
w_{1}\left( x\right) -\varepsilon^{2}e^{W_{a}\left( \varepsilon x\right)
}w_{2}\left( \varepsilon x\right),
\\
\mathbb{L}_{\e}^{a,2}\left( w_{1},w_{2}\right) :=\Delta w
_{2}\left( x\right) +2e^{W_{0}\left( x\right) }w_{2}\left( x\right) -%
\frac{f^{\prime}\left( V\left( \frac{x}{\varepsilon}\right) \right) }{%
2\varepsilon^{2}}w_{1}\left( \frac{x}{\varepsilon}\right).
\end{array}\right. \end{aligned}
\end{equation}
Now we state our main result in this paper.
\begin{theorem}
\label{thma} If $a=a(\e)\to0$ as $\e\to0$, then  for small $\e>0$, the map
$
\mathbb{QL}_{\e}^{a}:X_{\alpha}^{1}\times
{E}_{\alpha}\rightarrow Y_{\alpha}\times {F}_{\alpha}
$
is isomorphism. Moreover, for any $(w_{1},w_{2})\in
X_{\alpha}^{1}\times {E}_{\alpha}$ and $(h_{1},h_{2})\in
Y_{\alpha}\times {F}_{\alpha}$ satisfying
$
\mathbb{QL}_{\e}^{a}(w_{1},w_{2})=(h_{1},h_{2}),
$
  there is a constant $C>0$, independent of $\varepsilon>0$, such
that
\begin{equation}\label{importantineq}
|\ln \e|^{-1}\left(\Vert w_{1}\Vert_{L^{\infty}(\mathbb{R}^{2})}+ \Big\|\frac{w_2(x)}{\ln(2+|x|)}\Big\|_{L^\infty(\mathbb{R}^2)}\right)+\sum_{i=1}^2\Vert w_i(x)\rho_2(x)\Vert_{L^2(\mathbb{R}^2)}\leq C\left(\Vert h_{1}\Vert_{Y_{\alpha}}+\Vert h_{2}\Vert_{Y_{\alpha}}\right).
\end{equation}
\end{theorem}
 Since the proof is long and technical, we postpone it in Section \ref{sectiona}.
 We note that   even when $\Vert w_i\Vert_{X^i_\alpha}=O(1)$,     $\Vert\frac{f'(V(x/\varepsilon))w_1(x/\varepsilon)}{\varepsilon^2}\Vert_{Y_\alpha}=O(\varepsilon^{-1})$ and
 $\Vert\varepsilon^2e^{W_a(\varepsilon x)}w_2(\varepsilon)\Vert_{Y_\alpha}=O(\varepsilon^{-\frac{\alpha}{2}})$  are too big.
 To overcome this difficulty caused by the large norm of   the perturbations terms, we shall use   Green representation formula and the property of $X^1_\alpha$.

\begin{corollary}
\label{magofc} Let $\frac{{\gamma_{1}}}{2}\in\mathbb{Z}$. Then for any $%
(h_{1},h_{2})\in Y_{\alpha}\times F_{\alpha}$, there exist $%
(w_{1},w_{2})\in X_{\alpha}^{1}\times E_{\alpha}$ and $\left(
c_{1},c_{2}\right) \in \mathbb{R}^{2}$ such that
\begin{equation}
\begin{aligned}\left\{\label{withc} \begin{array}{ll}
\mathbb{L}_{1,\varepsilon}^{a}(w_{1},w_{2})=h_{1},\\
\mathbb{L}_{2,\varepsilon}^{a}(w_{1},%
w_{2})=h_{2}+c_{1}Z_{1}+c_{2}Z_{2}, \end{array}\right. \end{aligned}
\end{equation}
and $c_i=O(\varepsilon^{2-\frac{\alpha}{2}}\|w_1\rho_2\|_{L^2(\RN)})$.
\end{corollary}

\begin{proof}
In view of Theorem \ref{thma}, for any $(h_{1},h_{2})\in Y_{\alpha}\times
F_{\alpha}$, we get $(w_{1},w_{2})\in X_{\alpha}^{1}\times
E_{\alpha}$ and $\left( c_{1},c_{2}\right) \in \mathbb{R}^{2}$
satisfying the linear problem \eqref{withc}.

By   Lemma \ref{integrationbyparts},  $h_{2}\in F_{\alpha}$, and \eqref{withc}, we see that for $i\neq j$,
\begin{equation*}
\begin{aligned} &\Big|c_{i}\int_{\mathbb{R}^2} Z_{i}^2(x)dx+c_{j}\int_{\mathbb{R}^2} Z_{i}(x)Z_{j}(x)dx\Big|
=\Big|\int_{\mathbb{R}^2}
\frac{f'(V(x/\varepsilon))}{2\varepsilon^2}w_{1}(x/\varepsilon)Z_{i}(x)dx%
\Big|
\\&\le\frac{1}{2}
\|w_{1}\rho_2\|_{L^2(\mathbb{R}^2)} \Big\|| f'(V(x))\rho_2^{-1}(x)|\frac{|\varepsilon
x|^{N_2+\frac{{\gamma_1}}{2}+1}  }{1+|\varepsilon x|^{2N_2+\gamma_1+2}}\Big\|_{L^2(\RN)}
\\&\le O(
\|w_{1}\rho_2\|_{L^2(\mathbb{R}^2)})\Big(\|
(1+| x|)^{-2\gamma_1+2N_1+1+\frac{\alpha}{2}}|\varepsilon
x|^{N_2+\frac{{\gamma_1}}{2}+1}\|_{L^2(| x | \le\varepsilon^{-1})}
\\&\ +\|(1+| x|)^{-2\gamma_1+2N_1+1+\frac{\alpha}{2}}\|_{L^2(\e^{-1}\le| x |)}
\le O
(\varepsilon^{N_2+\frac{{\gamma_1}}{2}+1} +\e^{2\gamma_1-2N_1-2-\frac{\alpha}{2}})\|w_1\rho_2\|_{L^2(\RN)}.\end{aligned}
\end{equation*}
 Since $\int_{\RN} Z_{j}Z_{i}dx=\left(\int_{\RN} Z_{i}^2dx\right)\delta_{i,j}$,
 we get that  $c_i=O(\varepsilon^{2-\frac{\alpha}{2}}\|w_1\rho_2\|_{L^2(\RN)})$  for  $i=1,2$. This completes the proof of Corollary \ref{magofcc}.
\end{proof}

\subsection{Improved $1$-st order approximate solution}\label{1st}
Now we are going to improve our approximate solution  in the form  of $\left(
V_{1,\varepsilon}^{a},V_{2,\varepsilon}^{a}\right)=\left({U}_{1,\varepsilon}^{a}+
\phi_{1,\varepsilon}^{a},{U}_{2,\varepsilon}^{a}+\phi_{2,\varepsilon}^{a}\right) $. \\
Let $R_{0}>0$ such that $\{p_{j},q_{i}\} \subseteq B_{\frac{R_{0}}{2}}(0)$.
We denote
\begin{equation}\label{definitionofdx}d(x):=\int_{\mathbb{R}^2}\Big(\frac{2(x\cdot y)^2-|x|^2|y|^2}{8\pi|x|^4}\Big)f(V(y))dy,\  A(x):=
\sum_{j=1}^{N_{2}} \left \vert q_{j}\right \vert ^{2} -d\left( x\right)|x|^2 -2\sum_{j=1}^{N_{2}}\frac{\left( x\cdot
q_{j}\right) ^{2}}{\left \vert x\right \vert ^{2}}.\end{equation}
Let $x=|x|(\cos \theta +i\sin \theta)$,  $a=a_1+i(a_2)$, and $q_i=q_{i,1}+i(q_{i,2})$ for $1\le i\le N_2$. By the standard computation, we get
\begin{equation}\label{ang2}
\begin{aligned} A(x)&= \sum^{N_2}_{i=1}(-2q_{i_1}q_{i_2}\sin2%
\theta-(q_{i_1}^2-q_{i_2}^2)\cos2\theta)-\frac{1}{8\pi}(C_1\cos2\theta+C_2\sin2\theta), \end{aligned}
\end{equation}
where  $C_1:=\int_{\mathbb{R}^{2}}(y_{1}^{2}-y_{2}^{2})f(V(y))dy$, and $%
C_2:=2\int_{\mathbb{R}^{2}} y_{1}y_{2}f(V(y))dy$.

Now we consider the following linear problem:%
\begin{equation}
\left \{
\begin{array}{c}
\mathbb{L}_{1,\varepsilon}^a\left(
\phi_{1,\varepsilon}^{a},\phi_{2,\varepsilon}^{a}\right)
=k_{1,\varepsilon}^{a}, \\
Q\mathbb{L}_{2,\varepsilon}^a\left(
\phi_{1,\varepsilon}^{a},\phi_{2,\varepsilon}^{a}\right)
=k_{2,\varepsilon}^{a},%
\end{array}
\right.  \label{26}
\end{equation}
where%
\begin{align}
k_{1,\varepsilon}^{a}\left( x\right) & :=\varepsilon^{2}\left[ e^{\tilde {W}%
_{a}\left( \varepsilon x\right) +2\sum_{j=1}^{N_{2}}\ln \left \vert
\varepsilon x-\varepsilon q_{j}\right \vert +u_{2,\varepsilon}^{\ast}\left(
\varepsilon x\right) }-e^{W_{a}\left( \varepsilon x\right) }\right]
\label{27} \\
& -2\varepsilon^{4}e^{2\tilde{W}_{a}\left( \varepsilon x\right)
+4\sum_{j=1}^{N_{2}}\ln \left \vert \varepsilon x-\varepsilon q_{j}\right
\vert +2u_{2,\varepsilon}^{\ast}\left( \varepsilon x\right) },  \notag
\end{align}
and%
\begin{equation}
k_{2,\varepsilon}^{a}\left( x\right) :=\varepsilon^{2}\left(
4e^{2W_{0}\left( x\right) }-2e^{W_{0}\left( x\right) }\frac{A(x)}{|x|^2} \right) .
\label{28}
\end{equation}Here $k_{i,\varepsilon}^a$ can be regarded as the biggest part in the difference between $(u_{1,\varepsilon},u_{2,\varepsilon})$ and $(U^a_{1,\varepsilon},U^a_{2,\varepsilon})$.
It is easy to check that $k_{1,\varepsilon}^{a}\left( x\right),
k_{2,\varepsilon}^{a}\left( x\right) \in Y_{\alpha} $. Since $k_{2,\varepsilon}^{a}\left(
x\right)$ is a linear combination of radial symmetric function, $\cos 2\theta
$, and $\sin 2\theta$ by \eqref{ang2}, we can
see that $k_{2,\varepsilon}^{a}\left( x\right) \in F_{\alpha}$ from
the property of $Z_1$ and $Z_2$ (see \eqref{z0}). Then Corollary %
\ref{magofc} implies that there are $\left( \phi
_{1,\varepsilon}^{a},\phi_{2,\varepsilon}^{a}\right) \in
X_{\alpha}^{1}\times E_{\alpha}$ and two constants $%
c_{i,\varepsilon}^{a}=O(\varepsilon^{2-\frac{\alpha}{2}}\|\phi_{1,%
\varepsilon}^{a}\rho_2\|_{L^2(\RN)})$ $(i=1,2)$ satisfying
\begin{equation}\label{eqk}
\begin{aligned}\left\{ \begin{array}{ll}
\mathbb{L}_{1,\varepsilon}^{a}\left( \phi
_{1,\varepsilon}^{a},\phi_{2,\varepsilon}^{a}\right)=k_{1,\varepsilon}^{a},%
\\ \mathbb{L}_{2,\varepsilon}^{a}\left( \phi
_{1,\varepsilon}^{a},\phi_{2,\varepsilon}^{a}\right)=k_{2,%
\varepsilon}^{a}+c_{1,\varepsilon}^{a}Z_{1}+c_{2,\varepsilon}^{a}Z_{2}.
\end{array}\right. \end{aligned}
\end{equation}%
Furthermore, we prove that

\begin{proposition}
\label{T2}
(i) $|\ln \e|^{-1}\left(\Vert \phi_{1,\varepsilon}^{a}\Vert_{L^{\infty}(\mathbb{R}^{2})}+ \left\Vert \frac{\phi_{2,\varepsilon}^{a}}{\ln(2+|x|)}\right\Vert_{L^{\infty}(\mathbb{R}^{2})}\right)+\sum
_{i=1}^{2}\Vert \phi_{i,\varepsilon}^{a}\rho_2\Vert_{L^{2}(\mathbb{R}^{2})}=O(\varepsilon
^{2-\frac{\alpha}{2}}),$

(ii) $c_{i,\varepsilon}^{a}=O(\e^{4- \alpha })$ for $i=1,2$.
\end{proposition}
\begin{proof}
Since $\Vert k_{2,\varepsilon}^{a}\Vert_{Y_{\alpha}}=O(\varepsilon^{2})$ is
trivial, we consider $\Vert k_{1,\varepsilon}^{a}\Vert_{Y_{\alpha}}$. In
view of Lemma \ref{mean}, we see that
\begin{equation*}
\begin{aligned}  \|k^a_{1,\varepsilon}\|_{Y_\alpha}^2  & \le
 c \int_{\mathbb{R}^2} \varepsilon^2(e^{\tilde{W}_a(x) +2\sum_{i=1}^{N_2}\ln|x
-\varepsilon
q_i|+u_{2,\varepsilon}^*(x)}-e^{W_a(x)})^2\Big(1+\frac{|x|^{2+\alpha}}{%
\varepsilon^{2+\alpha}}\Big)dx
\\&+c \int_{\mathbb{R}^2}\varepsilon^6e^{4\tilde{W}_a(x)
+8\sum_{i=1}^{N_2}\ln|x -\varepsilon
q_i|+4u_{2,\varepsilon}^*(x)}\Big(1+\frac{|x|^{2+\alpha}}{\varepsilon^{2+%
\alpha}}\Big)dx \\&\le C\left[\varepsilon^{4+4N_2+2{\gamma_1}} +\int_{|x|\ge
\varepsilon
R_0}\varepsilon^{4-\alpha}\Big(e^{2W_a(x)}|x|^{-4}+e^{4W_a(x)}%
\Big) (1+|\ln|x||)^{12}|x|^{2+\alpha}dx\right]  \\& \le  O(\varepsilon^{4-\alpha}),\end{aligned}
\end{equation*}for some constants $c, C>0$.
Together with Theorem \ref{thma} and {Corollary}
\ref{magofc}, we complete the proof of Proposition \ref{T2}.\end{proof}

At this point, we can describe the final form of our approximate solution $(V_{1,\varepsilon}^a,V_{2,\varepsilon}^a)$ such that
\begin{equation}
\begin{aligned} \left \{ \begin{array}{ll}&V_{1,\varepsilon}^a(x):=V ( x ) +\frac{ \tilde{W}%
_{a} ( 0 ) -\tilde{W}_{a} ( \varepsilon x ) }{2}+\phi_{1,\varepsilon}^a(x),\\& V_{2,\varepsilon}^a(x):=\tilde{W}_a (\varepsilon x )+2\ln\e
+2\sum_{j=1}^{N_{2}}\ln   \vert \varepsilon x-\varepsilon q_{j} \vert
+u_{2,\varepsilon}^{\ast}(\varepsilon x)+\phi_{2,\varepsilon}^a(\varepsilon x),\end{array}\right. \end{aligned}\label{finalapp}
\end{equation}
where $(\phi_{1,\varepsilon}^a,\phi_{2,\varepsilon}^a)$ satisfies \eqref{26}.

\section{Proof of Theorem \ref{T121}  and Corollary \ref{T122}: Existence of solutions}
\subsection{Contraction mapping theorem}
We shall construct solutions  of \eqref{cs4} in the form of%
\begin{equation}
\left \{\label{finalapp2}
\begin{array}{l}
v_{1,\varepsilon}\left( x\right)=V ( x ) +\frac{ \tilde{W}%
_{a} ( 0 ) -\tilde{W}_{a} ( \varepsilon x ) }{2}  +\phi_{1,\varepsilon}^a\left(
x\right) +\eta_{1,\varepsilon}^a\left(
x\right)=V_{1,\varepsilon}^a(x)+\eta_{1,\varepsilon}^a\left(
x\right),  \\
v_{2,\varepsilon}\left( x\right) =\tilde{W}_a ( \e x )+2\ln \e
+2\sum_{j=1}^{N_{2}}\ln   \vert \e x-\varepsilon q_{j} \vert
+u_{2,\varepsilon}^{*}(\e x)   +\phi_{2,\varepsilon}^a\left(\e
x\right) +\eta_{2,\varepsilon}^a\left(\e
x\right)\\ \quad\quad\quad\ =V_{2,\varepsilon}^a(x)+\eta_{2,\varepsilon}^a\left(\e
x\right).
\end{array}
\right.
\end{equation}
To construct solutions of the form (\ref{finalapp2}), we need to find  $\left(
\eta_{1,\varepsilon}^{a},\eta_{2,\varepsilon}^{a}\right) $ solving the
following system:%
\begin{equation}
\left \{
\begin{array}{c}
\mathbb{L}_{\e}^{a,1}\left(
\eta_{1,\varepsilon}^{a},\eta_{2,\varepsilon}^{a}\right)
=g_{1,\varepsilon}^{a}\left(
\eta_{1,\varepsilon}^{a},\eta_{2,\varepsilon}^{a}\right),\quad \quad\quad \\
\mathbb{L}_{\e}^{a,2}\left(
\eta_{1,\varepsilon},\eta_{2,\varepsilon}\right)
=g_{2,\varepsilon}^{a}\left(
\eta_{1,\varepsilon}^{a},\eta_{2,\varepsilon}^{a}\right),\ \textrm{where}
\end{array}
\right.  \label{33}
\end{equation}
 \begin{equation*}
\begin{aligned}
& g_{1,\varepsilon}^{a}\left(\eta_{1,\varepsilon}^{a},\eta_{2\varepsilon
}^{a}\right)  :=-f\left( V_{1,\varepsilon}^{a}\left( x\right)  +\eta_{1,\varepsilon}^{a}\left( x\right) \right)
+f\left( V\left( x\right) \right) +f^{\prime}\left(V\left( x\right) \right) (\phi_{1,\e}^a(x)+\eta
_{1,\varepsilon}^{a}\left( x\right))
\\& +
e^{V_{2,\varepsilon}^{a}(x)+\eta_{2,\varepsilon}^{a}\left( \varepsilon x\right) }
-e^{V_{2,\varepsilon}^{a}(x)-\phi^a_{2,\e}(\e x)}-\e^2e^{W_{a}\left( \varepsilon x\right)
} (\phi_{2,\varepsilon}^{a}\left( \varepsilon x\right)+\eta_{2,\varepsilon}^{a}\left( \varepsilon x\right))
\\&-2 e^{2V_{2,\varepsilon}^{a}(x)  +2\eta_{2,\varepsilon}^{a}\left( \varepsilon x\right) }
+2e^{2V_{2,\varepsilon}^{a}(x)-2\phi^a_{2,\e}(\e x) }
\\& -e^{V_{1,\varepsilon}^{a}\left( x\right)
+ \eta_{1,\varepsilon}^{a}\left(
x\right) +V_{2,\varepsilon}^{a}\left( x\right)
+\eta_{2,\varepsilon}^{a}\left( \varepsilon x\right) },\ \ \textrm{and}
\\
 &g_{2,\varepsilon}^{a}\left( \eta_{1,\varepsilon}^{a},\eta_{2\varepsilon
}^{a}\right) :=
\frac{f\left( V_{1,\varepsilon}^{a}\left( \frac {x}{%
\varepsilon}\right)  +\eta_{1,\varepsilon}^{a}\left( \frac{x}{\varepsilon }\right)
\right) -f\left( V\left( \frac {x}{\varepsilon}%
\right) \right)
-f^{\prime}\left(V\left( \frac{x}{\varepsilon }%
\right) \right)\left(\phi_{1,\varepsilon}^{a}\left( \frac{x}{\varepsilon }\right)+\eta_{1,\varepsilon}^{a}\left( \frac{x}{\varepsilon }\right)\right)}{2\e^2}
 \\
& -  2e^{\tilde{W}_{a}\left( x\right) +2\sum_{j=1}^{N_{2}}\ln \left \vert
x-\varepsilon q_{j}\right \vert +u_{2,\varepsilon}^{*}\left( x\right)
 +\phi_{2,\varepsilon}^{a}(x) +\eta_{2,\varepsilon}^{a}\left(
x\right) }\quad\\&+2e^{W_{a}\left( x\right) }\Big(1+
\frac{A(x)}{|x|^2} \varepsilon^{2}+\phi_{2,\varepsilon}^{a}\left( x\right)+
\eta_{2,\varepsilon}^{a}\left( x\right)\Big)-4\e^2e^{2W_a(x)}
 +4\varepsilon^{-2}e^{2V_{2,\varepsilon}^{a}(\frac{x}{\e})  +2\eta_{2,\varepsilon}^{a}\left( x\right) }
 \\
& -\varepsilon^{-2}e^{V_{1,\varepsilon}^{a}\left( \frac{x}{\varepsilon}\right)+\eta_{1,\varepsilon
}^{a}\left( \frac{x}{\varepsilon}\right) +V_{2,\varepsilon}^{a}\left(
\frac{x}{\e}\right)
+\eta_{2,\varepsilon}^{a}\left( x\right)}
 +\varepsilon^{2}\left[
4(e^{2W_{a}\left( x\right) }-e^{2W_{0}\left( x\right) })-2(e^{W_{a}\left( x\right) }-e^{W_{0}\left( x\right) })\frac{A(x)}{|x|^2} \right]\\&+2(e^{W_0(x)}-e^{W_a(x)})(\phi_{2,\e}^a(x)+\eta_{2,\e}^a(x))-R_\e(x),
\end{aligned}\end{equation*}
here
\begin{equation}
R_\varepsilon(x):= \left \{
\begin{array}{l}
\sum_{i=1}^2c^a_{i,\varepsilon}Z_i \ \mbox{if}\ \frac{{\gamma_1}}{2}\in\mathbb{Z}%
, \\
0 \ \ \mbox{if}\ \frac{{\gamma_1}}{2}\notin\mathbb{Z}.%
\end{array}
\right.
\end{equation}
In view of contraction mapping theorem, we get the following proposition.
\begin{proposition}
\label{existence_of_eta} Let $0<\alpha<\min\left\{2(\gamma_1-2N_1-2N_2-4),\frac{1}{3},2\left(N_2+\frac{\gamma_1-2}{2}\right)\right\}$. There is  $\varepsilon_0>0$ such that for
any $\varepsilon\in(0,\varepsilon_0)$ and $a=a(\e)=O(\e^{2\alpha})$, there is $%
(\eta_{1,\varepsilon}^a,\eta_{2,\varepsilon}^a)\in X^1_\alpha\times
E_{\alpha}$ satisfying
\begin{equation}\left \{
\begin{array}{l}  \label{eq_for_existence_of_eta}
\mathbb{QL}_{\varepsilon}^{a}(\eta_{1,\varepsilon}^a,\eta_{2,%
\varepsilon}^a)=(g_{1,\varepsilon}^a(\eta_{1,\varepsilon}^a,\eta_{2,%
\varepsilon}^a),
{Q}g_{2,\varepsilon}^a(\eta_{1,\varepsilon}^a,\eta_{2,\varepsilon}^a)),
\\
|\ln\e|^{-1}\left(\Vert \eta_{1,\varepsilon}^{a}\Vert_{L^{\infty}(\mathbb{R}^{2})}+ \Big\|\frac{\eta_{2,\varepsilon}^a(x)}{\ln(2+|x|)}\Big\|_{L^\infty(\mathbb{R}^2)}\right)+\sum_{i=1}^2\Vert\eta_{i,\varepsilon}^{a}(x)\rho_{2}(x)\Vert_{L^2(\mathbb{R}^2)}
\le\e^{2+\alpha}.\end{array}
\right.
\end{equation}
\end{proposition}
\begin{proof}
By using Theorem \ref{thma}, \eqref{eq_for_existence_of_eta} can be written
as
\begin{equation*}
\begin{aligned}
(\eta_{1,\varepsilon}^a,\eta_{2,\varepsilon}^a)&=(\mathbb{QL}^a_{\e})^{-1}(g_{1,\varepsilon}^a(\eta_{1,\varepsilon}^a,\eta_{2,%
\varepsilon}^a),
{Q}g_{2,\varepsilon}^a(\eta_{1,\varepsilon}^a,\eta_{2,\varepsilon}^a))
\\&=:B_\varepsilon^a(\eta_{1,\varepsilon}^a,\eta_{2,\varepsilon}^a)=(B_{1,%
\varepsilon}^a(\eta_{1,\varepsilon}^a,\eta_{2,\varepsilon}^a),B_{2,%
\varepsilon}^a(\eta_{1,\varepsilon}^a,\eta_{2,\varepsilon}^a)). \end{aligned}
\end{equation*}
From Theorem \ref{thma} and Lemma  \ref{pronorm}, we have  for a constant  $C>0$,  independent of $\varepsilon>0$,
\begin{equation}
\begin{aligned}\label{bnorm}
&|\ln \e|^{-1}\left(\|B_{1,\varepsilon}^a(\eta_{1,\varepsilon}^a,\eta_{2,\varepsilon}^a)\|_{L^\infty(\mathbb{R}^2)}+ \Big\|\frac{B_{2,\varepsilon}^a(\eta_{1,\varepsilon}^a,\eta_{2,%
\varepsilon}^a)(x)}{\ln(2+|x|)}\Big\|_{L^\infty(\mathbb{R}^2)}\right)+\sum_{i=1}^2\|B_{i,\varepsilon}^a(\eta_{1,\varepsilon}^a,\eta_{2,%
\varepsilon}^a)\rho_2\|_{L^2(\mathbb{R}^2)}\\& \le
C ( \|g_{1,\varepsilon}^a(\eta_{1,\varepsilon}^a,\eta_{2,%
\varepsilon}^a)\|_{Y_\alpha} + \|g_{2,\varepsilon}^a(%
\eta_{1,\varepsilon}^a,\eta_{2,\varepsilon}^a)\|_{Y_\alpha} ). \end{aligned}
\end{equation}Let
\begin{equation}
\begin{aligned}
&S_{\varepsilon}:=\Big\{(w_{1},w_{2})\in X^{1}_{\alpha}\times E_{\alpha}\ \Big|\\&
 |\ln \e|^{-1}\left(\| w_{1}\Vert_{L^{\infty}(\mathbb{R}%
^{2})}+   \Big\|\frac{w_{2} (x)}{\ln(2+|x|)}\Big\|_{L^\infty(\mathbb{R}^2)}\right)+ \sum_{i=1}^2\Vert
w_{i}\rho_{2}\Vert_{L^2(\mathbb{R}^2)}\le \e^{2+\alpha}\Big\}.
\end{aligned}
\end{equation}
We will prove that there exists  $\varepsilon_{0}>0$ such that for any
 $\varepsilon \in(0,\varepsilon_{0}]$ and $a=a(\e)=o(1)$, $%
B_{\varepsilon }^{a}=(B_{1,\varepsilon}^{a},B_{2,\varepsilon}^{a})$ is a
contraction mapping from $S_{\varepsilon}$ to $S_{\varepsilon}$.
First of all,  we claim that for small $\varepsilon>0$,
$B_{\varepsilon}^{a}: S_{\varepsilon}\longrightarrow S_{\varepsilon}.$
To prove it, we  shall  estimate $\|g_{\varepsilon
,i}^{a}(w_{1},w_{2})\|_{Y_{\alpha}}$, $ i=1,2,$ for $(w_{1},w_{2})\in S_{\varepsilon}$.
Denote
$
g_{1,\varepsilon}^{a}(w_{1},w_{2})
=I^{(1)}_{\varepsilon}+II^{(1)}_{\varepsilon}+III^{(1)}_{\varepsilon},\ \textrm{where}
$
\begin{equation}\label{123}
\begin{aligned}\left\{ \begin{array}{ll}  I^{(1)}_\varepsilon(x)
&:=-f\left(( V_{1,\varepsilon}^{a}  +w_1)\left( x\right) \right)
+f\left( V\left( x\right) \right) +f^{\prime}\left( V\left( x\right) \right)  (\phi^a_{1,\e}(x)+ w_1\left( x\right)
),
\\ II^{(1)}_\varepsilon(x)
&:=e^{V_{2,\varepsilon}^{a}(x)}(e^{ w_2\left( \varepsilon x\right) }-e^{-\phi^a_{2,\e}(\e x)})
-\e^2e^{W_{a}\left( \varepsilon x\right)}(\phi^a_{2,\e}(\e x)+ w_2\left( \varepsilon x\right))
 \\& -2e^{2V_{2,\varepsilon}^{a}(x)}(e^{ 2w_2\left( \varepsilon x\right) }-e^{-2\phi^a_{2,\e}(\e x)}),
 \\ III^{(1)}_\varepsilon(x)&:=-e^{V_{1,\varepsilon}^{a}(x)+w_1(x)+V_{2,\varepsilon}^{a}(x)+w_2(\e x)}.\end{array}\right. \end{aligned}
\end{equation}
Note that $e^{V(x)}=O((1+|x|)^{-2\gamma_1+2N_1})$ and   $\gamma_1>2N_1+2N_2+4$.
 Lemma \ref{mean} yields
\begin{equation*}
\begin{aligned} V_{1,\varepsilon}^{a}(x)&=V ( x ) +\frac{ \tilde{W}%
_{a} ( 0 ) -\tilde{W}_{a} ( \varepsilon x ) }{2}+\phi^a_{1,\e}(x) =V(x)+\ln \Big(\frac{  1+ \left \vert (\e x)^{1+\frac{{\gamma_1}}{2}%
+N_{2}}+a\right \vert ^{2}  }{  1+ \left \vert a\right \vert ^{2}  }\Big)+\phi^a_{1,\e}(x)\\&=V(x)+\phi^a_{1,\e}(x)+O(|\e x|^{2N_2+\gamma_1+2}+|a||\e x|^{N_2+\frac{\gamma_1}{2}+1}).\end{aligned}
\end{equation*}
We see that for some $\theta \in(0,1)$,
\begin{equation*}
\begin{aligned}&I^{(1)}_\varepsilon(x)
= -f'(V(x))\Big(\frac{\tilde{W}_a(0)-\tilde{W}_a(\e x)}{2}\Big)\\&-\frac{1}{2}f''(V(x)+\theta (V_{1,\varepsilon}^{a}(x)-V(x)+w_1(x)))\Big(\frac{\tilde{W}_a(0)-\tilde{W}_a(\e x)}{2}+\phi^a_{1,\e}(x)+w_1(x)\Big)^2.\end{aligned}
\end{equation*}By  $0<\alpha< \min\left\{2(\gamma_1-2N_1-2N_2-4),\frac{1}{3}\right\}$ and Proposition \ref{T2}, we get that
\begin{equation}\label{I0Y3}
\begin{aligned} &\|I^{(1)}_\varepsilon \|_{Y_\alpha}= O( \| w_1\|_{L^\infty(\RN)}^2+\|\phi^a_{1,\e}\|_{L^\infty(\RN)}^2+
\e^4+\e^{3}|a|)=O( \| w_1\|_{L^\infty(\RN)}^2 +
\e^{4-\alpha}|\ln \e|^2+\e^{3}|a|).
\end{aligned}\end{equation}
We see that
\begin{equation*}
\begin{aligned} II^{(1)}_\varepsilon(x)
&=\varepsilon^{2}\left[
e^{\tilde{W}_{a}\left( \varepsilon x\right) +2\sum_{j=1}^{N_{2}}\ln \left
\vert \varepsilon x-\varepsilon q_{j}\right \vert +u_{2,\varepsilon
}^{*}\left( \varepsilon x\right)}( e^{\phi^a_{2,\e}(\e x)+w_2\left( \varepsilon x\right) }-1)
 -e^{W_{a}\left( \varepsilon x\right)
}(\phi^a_{2,\e}(\e x)+ w_2\left( \varepsilon x\right))\right]
\\&-2\varepsilon^{4}
e^{2\tilde{W}_{a}\left( \varepsilon x\right) +4\sum_{j=1}^{N_{2}}\ln \left
\vert \varepsilon x-\varepsilon q_{j}\right \vert +2u_{2,\varepsilon
}^{*}\left( \varepsilon x\right)}(e^{ 2\phi^a_{2,\e}(\e x)+2w_2\left( \varepsilon x\right) }-1).
\end{aligned}
\end{equation*}
By $e^{W_a(x)}=O \left(\frac{|x|^{2N_2+\gamma_1}}{(1+|x| )^{4N_2+2\gamma_1+4}}\right)$, Lemma \ref{mean},   and $\gamma_1>2N_1+2N_2+4$,  we see that for some $\theta \in(0,1)$,
 \begin{equation}\label{I0Y1}
\begin{aligned} &\|II^{(1)}_\varepsilon\|_{Y_\alpha}= O \left( \||x| ^{1+\frac{\alpha}{2}}II^{(1)}_\varepsilon(x) \|_{L^2(|x|\ge R_0)}+\e^{2N_2+\gamma_1+1}( \| \phi^a_{2,\e}\rho_2\|_{L^2(\RN)}+\| w_2\rho_2\|_{L^2(\RN)})\right)
\\&\le O(\e^{-\frac{\alpha}{2}}) \Bigg\| |x|^{1+\frac{\alpha}{2}}e^{W_a}
 \Bigg[\phi^a_{2,\e}+ w_2\\&-\left\{\left(1+O\left( \frac{%
\varepsilon^2}{|x|^2} (1+|\ln|x||)^3\right)\right) \left(\phi^a_{2,\e}+w_2+e^{\theta(\phi^a_{2,\e}+ w_2 )}\frac{( \phi^a_{2,\e}+w_2)^2}{2}\right)\right\}\Bigg]\Bigg\|_{L^2(|x|\ge\e R_0)}
\\&+ O(\e^{2-\frac{\alpha}{2} })\Big\||x|^{1+\frac{\alpha}{2}}e^{2(W_a+\theta (\phi^a_{2,\e}+w_2))}(\phi^a_{2,\e}+w_2)
\left\{1+O\left( \frac{\varepsilon^2}{|x|^2}(1+|\ln|x||^6)\right)\right\}
\Big\|_{L^2(|x|\ge \e R_0)}
\\&+\e^{2N_2+\gamma_1+1} (\|\phi^a_{2,\e}\rho_2\|_{L^2(\RN)}+\| w_2\rho_2\|_{L^2(\RN)})
\\&=O (1)\left(\e^{-\frac{\alpha}{2}} \Big(\Big\|\frac{w_2(x)}{\ln(2+|x|)}\Big\|_{L^\infty(\RN)}^2+ \Big\|\frac{\phi^a_{2,\e}(x)}{\ln(2+|x|)}\Big\|_{L^\infty(\RN)}^2 \Big) \right)\\&+O\left(\e^{2-\frac{\alpha}{2}} (\| w_2\rho_2\|_{L^2(\RN)}+   \| \phi^a_{2,\e}\rho_2\|_{L^2(\RN)})\right)\\&=O (1)\left(\e^{-\frac{\alpha}{2}} \Big(\Big\|\frac{w_2(x)}{\ln(2+|x|)}\Big\|_{L^\infty(\RN)}^2+\e^{2-\frac{\alpha}{2}} \| w_2\rho_2\|_{L^2(\RN)}+\e^{4-\frac{3\alpha}{2}}|\ln \e|^2\right).
 \end{aligned}
\end{equation}
Note that $
III^{(1)}_\varepsilon(x)=-\varepsilon^2e^{\tilde{W}_a(\varepsilon x)
+2\sum^{N_2}_{i=1}\ln|\varepsilon x -\varepsilon
q_i|+u_{2,\varepsilon}^*(\varepsilon x )+\phi^a_{2,\e}(\e x)+w_2(\varepsilon x)+V(x)+\frac{\tilde{W}_a(0)-\tilde{W}_a(\varepsilon
x)}{2} +\phi^a_{1,\e}(x)+w_1(x)}. $\\
 By using    Lemma \ref{mean}, we see that
\begin{equation}\label{I0Y2}
\begin{aligned} &\|III^{(1)}_\varepsilon \|_{Y_\alpha}= O(1)\Big( \e^4+\e^4\|w_1\|_{L^\infty(\RN)}+\e^4\Big\|\frac{w_2(x)}{\ln(2+|x|)}\Big\|_{L^\infty(\RN)}\Big).\end{aligned}
\end{equation}
From \eqref{I0Y3}-\eqref{I0Y2},      we obtain that
\begin{equation}
\begin{aligned}\label{mag_of_gi2}
 \|g_{1,\varepsilon}^a(w_1,w_2)\|_{Y_\alpha}&\le      O(1)\Big( \| w_1\|_{L^\infty(\RN)}^2 + \e^{-\frac{\alpha}{2}}\Big\|\frac{w_2(x)}{\ln(2+|x|)}\Big\|_{L^\infty(\RN)}^2  +\e^{2-\frac{\alpha}{2}} \| w_2\rho_2\|_{L^2(\RN)}\Big)\\&+
 O(1)\Big( \e^3+\e^3\|w_1\|_{L^\infty(\RN)}+\e^3\Big\|\frac{w_2(x)}{\ln(2+|x|)}\Big\|_{L^\infty(\RN)}\Big). \end{aligned}
\end{equation}
In order to estimate $\|g_{2,\varepsilon}^a(w_1,w_2)\|_{Y_\alpha}$, we  see that if $|x|\ge \e R_0$, then Lemma \ref{mean} implies that \begin{equation}\label{II2}\begin{aligned}II_\e^{(2)}(x)&:=-2e^{\tilde{W}_{a}\left( x\right) +2\sum_{j=1}^{N_{2}}\ln \left \vert
x-\varepsilon q_{j}\right \vert +u_{2,\varepsilon}^{*}\left( x\right)
 +\phi_{2,\varepsilon}^{a}(x) +w_2\left(
x\right) }\quad\\&+2e^{W_{a}\left( x\right) }\Big(1+
\frac{A(x)}{|x|^2} \varepsilon^{2}+\phi_{2,\varepsilon}^{a}\left( x\right)+
w_2\left( x\right)\Big)-4\e^2e^{2W_a(x)}
 +4\varepsilon^{-2}e^{2V_{2,\varepsilon}^{a}(\frac{x}{\e})  +2w_2\left( x\right) }
 \\&=-2e^{W_a }\Big\{1+\frac{A(x)}{|x|^2}\e^2+O(\frac{\e^3}{|x|^3}(1+|\ln|x||)^3)\Big\}\Big\{1+ \phi_{2,\varepsilon}^{a} +w_2  +O(|\phi_{2,\varepsilon}^{a}|^2 +|w_2|^2)\Big\}
 \\&+2e^{W_{a}\left( x\right) }\Big(1+
\frac{A(x)}{|x|^2} \varepsilon^{2}+\phi_{2,\varepsilon}^{a}\left( x\right)+
w_2\left( x\right)\Big)
\\&-4\e^2e^{2W_a(x)}
 +4\e^2e^{2W_a(x)}\Big\{1+  O(\frac{\e^2}{|x|^2}(1+|\ln|x||)^6) + O(|\phi_{2,\varepsilon}^{a}|  +|w_2| )\Big\}\\&=O(e^{W_a})\Big(|\phi_{2,\varepsilon}^{a}|^2 +|w_2|^2+\frac{\e^2|\phi_{2,\varepsilon}^{a}|}{|x|^2}  +\frac{\e^2|w_2|}{|x|^2}+\frac{\e^3}{|x|^3}\Big)(1+|\ln|x||)^6.\end{aligned}\end{equation}
Together with the similar arguments in \eqref{I0Y3}-\eqref{I0Y2}, we obtain that
\begin{equation}
\begin{aligned}\label{mag_of_gi}
&\|g_{2,\varepsilon}^a(w_1,w_2)\|_{Y_\alpha}
\\&\le O(1)\Big(\frac{\|I_\e^{(1)}\|_{Y_\alpha}}{\e}+\|II_\e^{(2)}\|_{Y_\alpha}+\frac{\|III_\e^{(1)}\|_{Y_\alpha}}{\e}\Big)
\\& +O(1)\Big(|a|\e^2+|a|\|\phi^a_{2,\e}\rho_2\|_{L^2(\RN)}+|a|\|w_2\rho_2\|_{L^2(\RN)}+\sum_{i=1}^2|c_{i,\e}^a|\Big)
\\&
 \le   O(1)\Big( \Big\|\frac{w_{2}(x)}{\ln(2+|x|)}\Big\|_{L^\infty(\mathbb{R}^2)}^2+\varepsilon^2
 \| w_{2}\rho_2\|_{L^2(\mathbb{R}^2)}+\frac{\| w_1\|_{L^\infty(\RN)}^2}{\e}\Big)
\\&+ O(1)\Big(\e^{3-\alpha}|\ln\e|^2+\e^3\|w_1\|_{L^\infty(\RN)}+\e^3\Big\|\frac{w_2(x)}{\ln(2+|x|)}\Big\|_{L^\infty(\RN)}\Big)
 \\&+O(1)\Big(  |a|\e^2+|a|\|\phi^a_{2,\e}\rho_2\|_{L^2(\RN)}+|a|\|w_2\rho_2\|_{L^2(\RN)}\Big). \end{aligned}
\end{equation}
By using  \eqref{bnorm}, \eqref{mag_of_gi2}- \eqref{mag_of_gi}, and $a=O(\e^{2\alpha})$, we see that
\begin{equation*}
\begin{aligned} &|\ln \e|^{-1}\left(\|B_{1,\varepsilon}^a(w_1,w_2)\|_{L^\infty(\mathbb{R}^2
)}+ \Big\|\frac{B_{2,\varepsilon}^a(w_1,w_2)}{\ln(2+|x|)}\Big\|_{L^\infty(\mathbb{R}^2)}\right)+\sum_{i=1}^2\|B_{i,\varepsilon}^a(w_1,w_2)\rho_2\|_{L^2(\mathbb{R}^2)}
 =o(\e^{2+\alpha}),  \end{aligned}
\end{equation*}which implies the claim
$B_{\varepsilon}^{a}: S_{\varepsilon}\longrightarrow S_{\varepsilon}$ for small $\varepsilon>0$.
Similarly, if $(w_1,w_2)$, $(\tilde{w}_1,\tilde{w}_2)\in S_\e$, then
   for small $\e>0$,
\begin{equation*}
\begin{aligned}
&|\ln \e|^{-1}\left(\|B_{1,\varepsilon}^a(w_1,w_2)-B_{1,\varepsilon}^a(\tilde{w}_1,\tilde{w}_2)%
\|_{L^\infty(\mathbb{R}^2
)}+ \Big\|\frac{(B_{2,\varepsilon}^a(w_1,w_2)-B_{2,\varepsilon}^a(\tilde{w}_1,\tilde{w}_2)) (x)}{\ln(2+|x|)}\Big\|_{L^\infty(\mathbb{R}^2)}\right)\\&+\sum_{i=1}^2\|(B_{i,\varepsilon}^a(w_1,w_2)-B_{1,\varepsilon}^a(\tilde{w}_1,\tilde{w}_2))(x)\rho_2(x)\|_{L^2(\mathbb{R}^2)}
\\&\le
\frac{1}{2}\left(|\ln \e|^{-1}\left(\|w_1-\tilde{w}_1\|_{L^\infty(\mathbb{R}^2
)}+ \Big\|\frac{(w_2-\tilde{w}_2) (x)}{\ln(2+|x|)}\Big\|_{L^\infty(\mathbb{R}^2)}\right)+\sum_{i=1}^2\|(w_i-\tilde{w}_i)(x)\rho_2(x)\|_{L^2(\mathbb{R}^2)}
\right). \end{aligned}
\end{equation*}Therefore,  $%
B_{\varepsilon }^{a} : S_{\varepsilon} \rightarrow S_{\varepsilon}$  is the
contraction mapping
 for small
 $\varepsilon>0$ and $a=a(\e)=O(\e^{2\alpha})$.
By
 contraction mapping theorem,  there is a
unique $(\eta_{1,\varepsilon}^{a},\eta _{2,\varepsilon}^{a})\in
S_{\varepsilon}\subseteq X^{1}_{\alpha}\times E_{\alpha}$ satisfying \eqref{eq_for_existence_of_eta}.
 This completes the proof of  Proposition
\ref{existence_of_eta}.
\end{proof}
\subsection{Finite dimensional reduced problem}
If $\frac{{\gamma_1}}{2}\notin\mathbb{N}$,    there is no projection problem, and  Proposition \ref{existence_of_eta} with $\mathbb{QL}_{\e}^{0}$
directly implies the existence of solution of  \eqref{cs4}. In this subsection, we only
consider the case $\frac{{\gamma_1}}{2}\in\mathbb{N}$. By Proposition %
\ref{existence_of_eta}, there exists $\varepsilon_0>0$ such that for   $%
\varepsilon\in(0,\varepsilon_0]$ and any $a=a(\e)=O(\e^{2\alpha})$, there is $%
(\eta_{1,\varepsilon}^a,\eta_{2,\varepsilon}^a)\in X^1_\alpha\times
E_{\alpha}$ and $(C_{1,\varepsilon}^a,C_{2,\varepsilon}^a)\in\mathbb{R}%
\times\mathbb{R}$ such that
\begin{equation*}
\begin{aligned} \left\{ \begin{array}{ll} &\Delta
\eta^a_{1,\varepsilon}+f'(V(x))\eta^a_{1,\varepsilon}(x)-\varepsilon^2e^{W_a(%
\varepsilon x)}\eta^a_{2,\varepsilon}(\varepsilon x)
=g^a_{1,\varepsilon}(\eta_{1,\varepsilon}^a,\eta_{2,\varepsilon}^a),
\\&\Delta\eta^a_{2,\varepsilon}+2e^{W_0(x)}\eta^a_{2,\varepsilon}(x)-\frac{f'(V(%
\frac{x}{\varepsilon}))}{2\varepsilon^2}\eta^a_{1,\varepsilon}\Big(\frac{x}{%
\varepsilon}\Big)
=g^a_{2,\varepsilon}(\eta_{1,\varepsilon}^a,\eta_{2,\varepsilon}^a)
+\sum^2_{i=1}C_{i,\varepsilon}^aZ_{i}, \end{array}\right. \end{aligned}
\end{equation*}
and for $j=1,2$,
\begin{equation}\label{prored}
\begin{aligned}&\int_{\mathbb{R}^2}\Big\{\Delta\eta^a_{2,%
\varepsilon}+2e^{W_0}\eta^a_{2,\varepsilon}-\frac{f'(V(\frac{x}{%
\varepsilon}))\eta^a_{1,\varepsilon}(\frac{x}{%
\varepsilon})}{2\varepsilon^2}
-g^a_{2,\varepsilon}(\eta_{1,\varepsilon}^a,\eta_{2,\varepsilon}^a)
-\sum^2_{i=1}C_{i,\varepsilon}^aZ_{i}\Big\}Z_{j}dx=0.\end{aligned}
\end{equation}From now on, for the simplicity, we denote for $i=1,2$, $g^a_{i,\varepsilon}(\eta_{1,\varepsilon}^a,\eta_{2,\varepsilon}^a)$ by $g^a_{i,\varepsilon}$.
To complete the proof of  Theorem \ref{T121}, we are going to find $a\in\mathbb{R}^2$
satisfying  $C_{i,\varepsilon}^a\equiv0$ for $i=1,2$.
\begin{proposition}
\label{propgoal} Let $0<\alpha<\min\left\{2(\gamma_1-2N_1-2N_2-4),\frac{1}{3},2\left(N_2+\frac{\gamma_1-2}{2}\right)\right\}$.  There exists $\e_0>0$ such that for each $\e\in(0,\e_0)$, there exists $a_{\varepsilon}\in \mathbb{R}^{2}$ such that $%
|a_{\varepsilon}|=O(\e)$ and $C_{i,\varepsilon}^{a_\e}\equiv0$ for $i=1,2$.
\end{proposition}
\begin{proof}
 \textit{Step 1}.  We claim that\begin{equation}\label{cl_1}
\begin{aligned} C_{j,\e}^a\int_{\RN} Z_{j}^2dx= -\intr  g^{a}_{2,\varepsilon}Z_{j}dx+
O(\varepsilon^3).\end{aligned}
\end{equation}
Indeed, by $\int_{\RN} Z_{j}Z_{i}dx =\Big[\int_{\RN} Z_{i}^2dx\Big]\delta_{i,j}$,   Lemma \ref{integrationbyparts} and \eqref{prored}, we see that for $j=1,2$, \begin{equation}\label{reduction_1}
\begin{aligned} C_{j,\e}^a\int_{\RN} Z_{j}^2dx&=\intr(\sum_{i=1}^2 C_{i,\e}^a Z_{i}) Z_{j} dx\\&
 =\int_{\mathbb{R}^2}\Big\{ \Delta
\eta^{a}_{2,\varepsilon}+2e^{W_{a}}\eta^{a}_{2,%
\varepsilon}-\frac{f'(V(\frac{x}{\varepsilon}))}{2\varepsilon^2}\eta^{a_%
\varepsilon}_{1,\varepsilon}\Big(\frac{x}{\varepsilon}\Big)
-g^{a}_{2,\varepsilon}
\Big\}Z_{j}dx
\\&=-\intr  \left(\frac{f'(V(x))}{2}\eta^{a}_{1,\varepsilon}(x)Z_{j}(\e x)
+g^{a}_{2,\varepsilon}Z_{j}\right)dx
\\&=-\intr  g^{a}_{2,\varepsilon}Z_{j}dx+
O\left(\intr|f'(V(x))\eta^a_{1,\varepsilon}(x)|
\frac{|\varepsilon x|^{N_2+\frac{{\gamma_1}}{2}+1}}{1+|\varepsilon
x|^{2N_2+{\gamma_1}+2}}dx \right)\\&= -\intr  g^{a}_{2,\varepsilon}Z_{j}dx+
O(\|\eta^a_{1,\varepsilon}\|_{L^\infty(\RN)}(\varepsilon^{N_2+\frac{{\gamma_1}}{2}+1}+\e^{2\gamma_1-2N_1-2}))
\\&= -\intr  g^{a}_{2,\varepsilon}Z_{j}dx+
O(\varepsilon^3).\end{aligned}
\end{equation}

\medskip
\noindent
\textit{Step 2.}
We claim that for $j=1,2$,
 \begin{equation}
\begin{aligned}\label{eachterm5} -\int_{\mathbb{R}^2}
 g^{a}_{2,\varepsilon}Z_{j} dx&= 32a_j\e^2\intr e^{2W_0}  Z_{j}^2 dx-8\intr e^{W_0(x)}\left(\sum_{i=1}^2a_i\phi_{2,\e}^a(x)Z_{i}\right)Z_{j}dx\\&+  O(\e^{3}+|a|\e^{2+\alpha}+|a|^2\e^{2-\frac{\alpha}{2}}).\end{aligned}
 \end{equation}Indeed, by using \eqref{123}, \eqref{II2}-\eqref{mag_of_gi}, and  Lemma \ref{mean}, we get for $j=1,2$, \begin{equation*}
\begin{aligned} & -\int_{\mathbb{R}^2}
 g^{a}_{2,\varepsilon}Z_{j}  dx\\&=
   \int_{\RN}  \Big(2(e^{W_a(x)}-e^{W_0(x)})\frac{A(x)\e^2}{|x|^2}
-4\e^2 (e^{2W_a(x)}-e^{2W_0(x)})\Big) Z_{j}(x)dx+ 2\intr( e^{W_a(x)}-e^{W_0(x)})\phi_{2,\e}^a(x)Z_{j}dx \\&+O\Bigg(\intr \Big(\frac{|I_\e^{(1)}(\frac{x}{\e})|+|III^{(1)}_\e(\frac{x}{\e})|}{\e^2}\Big)|Z_j(x)|dx\Bigg)+  O\Big(\e^{3}+|a|\Big\|\frac{\eta_{2,\e}^a(x)}{\ln(2+|x|}\Big\|_{L^\infty(\RN)}\Big). \end{aligned}
\end{equation*}
Together with \eqref{ang2}, $(N_2+\frac{\gamma_1}{2}+1)>2$, and   radial symmetric property of $W_0$, we get that
\begin{equation*}
\begin{aligned} &-\int_{\mathbb{R}^2}
 g^{a}_{2,\varepsilon}Z_{j} dx\\&=
   -\int_{\RN}  4\e^2e^{2W_0(x)}\Big(\sum_{i=1}\Big(2a_i\frac{\partial W_a}{\partial a_i}\Big|_{a=0}Z_{j} \Big)dx  + 2\intr e^{W_0(x)}\left(\sum_{i=1}^2a_i\frac{\partial W_a}{\partial a_i}\Big|_{a=0}\right)\phi_{2,\e}^a(x)Z_{j}dx\\&+O\Big(\intr (|I_\e^{(1)}(x)|+|III^{(1)}_\e(x)|)|Z_j(\e x)|dx\Big)+   O\Big(\e^{3 } +|a|\Big\|\frac{\eta_{2,\e}^a(x)}{\ln(2+|x|}\Big\|_{L^\infty(\RN)}+|a|^2\e^2+|a|^2\|\phi^a_{2,\e}\rho_2\|_{L^2(\RN)}\Big)\\& = 32a_j\e^2\intr e^{2W_0}Z_{j}^2dx
  -8\intr e^{W_0(x)}\left(\sum_{i=1}^2a_i\phi_{2,\e}^a(x)Z_{i}\right)Z_{j}dx+  O(\e^{3}+|a|\e^{2+\alpha}+|a|^2\e^{2-\frac{\alpha}{2}}), \end{aligned}
\end{equation*}where we used \eqref{I0Y3}, \eqref{I0Y2}, Proposition \ref{T2}, and Proposition \ref{existence_of_eta} for the second indentity. Now we complete the proof of the claim \eqref{eachterm5}.

\medskip
\noindent
\textit{Step 3.}
We claim for $j=1,2,$
 \begin{equation}\label{step4_r}
\begin{aligned}&-8\intr e^{W_0(x)}\left(\sum_{i=1}^2a_i\phi_{2,\e}^a(x)Z_{i}\right)Z_{j}dx\\&=-  a_j  \int_{\mathbb{R}^2}k_{1,\e}^adx
-4a_j\intr \e^2e^{2W_0(x)}\left(\frac{1}{(1+|x|^{2N_2+\gamma_1+2})}+\frac{1}{(1+|x|^{2N_2+\gamma_1+2})^2}\right)dx\\&+O(|a|^2 \e^{2-\frac{\alpha}{2}}+\e^{4 }).\end{aligned}
\end{equation}
To prove the claim \eqref{step4_r}, firstly, we consider the following functions as in \cite[Lemma 2.5]{CI}:
\begin{equation*}
\xi_{r}(x)=\xi_{r}(|x|):=\frac{1}{4(1+|x|^{2N_{2}+{\gamma_1}+2})^{2}}=\frac{1}{4}-\frac{(2|x|^{2N_2+\gamma_1+2}+|x|^{4N_2+2\gamma_1+4})}{4(1+|x|^{2N_2+\gamma_1+2})^2},\ \ \textrm{and}
\end{equation*}
\begin{equation*}
\xi_{\theta }(x)=\xi_{\theta }(|x|):=-\frac{|x|^{2N_{2}+{\gamma_1} +2}}{4(1+|x|^{2N_{2}+{\gamma_1}+2})^{2}}.
\end{equation*}
Then $(\xi_{r},\xi_{\theta})$ satisfies
\begin{equation*}
\begin{aligned} \left \{ \begin{array}{ll}
&(\xi_r)''+\frac{(\xi_r)'}{r}+2e^{W_0}\xi_r=\frac{4(N_2+%
{\gamma_1}/2+1)^2|x|^{4N_2+2{\gamma_1}+2}}{(1+|x|^{2N_2+{\gamma_1}+2})^4},
\\&(\xi_\theta)''+\frac{(\xi_\theta)'}{r}-\frac{(2N_2+{\gamma_1}+2)^2}{r^2}\xi_%
\theta+2e^{W_0}\xi_\theta=\frac{4(N_2+{\gamma_1}/2+1)^2|x|^{4N_2+2%
{\gamma_1}+2}}{(1+|x|^{2N_2+{\gamma_1}+2})^4}\ \ \mbox{where} \  r=|x|.
\end{array}\right. \end{aligned}
\end{equation*}
Define $\xi_{i,j}$ as follows:
\begin{equation}  \label{ang3}
\begin{aligned} \left \{ \begin{array}{ll}&\xi_{1,1}=\xi_r+\xi_\theta
\cos[(2N_2+{\gamma_1}+2)\theta], \\&\xi_{1,2}=\xi_{2,1}=\xi_\theta
\sin[(2N_2+{\gamma_1}+2)\theta],\\&\xi_{2,2}=\xi_r-\xi_\theta
\cos[(2N_2+{\gamma_1}+2)\theta]. \end{array}\right. \end{aligned}
\end{equation}
Then $\xi_{i,j}\in X^{2}_{\alpha}$ satisfies
\[\Delta
\xi_{i,j}+2e^{W_0}\xi_{i,j}=2Z_iZ_je^{W_0}.\]
In view of \eqref{ang2}, $(N_2+\frac{\gamma_1}{2}+1)>2$, and   radial symmetric property of $W_0$, we get
 \begin{equation}\label{pf_reduction1}
\begin{aligned}&-8\intr e^{W_0(x)} a_i\phi_{2,\e}^a(x)Z_{i}Z_{j}dx\\&= -4a_i\int_{\mathbb{R}^2}(\Delta \xi_{ij}+2e^{W_0(x)}\xi_{ij})\phi_{2,\e}^adx=  -4a_i\int_{\mathbb{R}^2}(\Delta \phi_{2,\e}^a+2e^{W_0(x)}\phi_{2,\e}^a)\xi_{ij}dx\\&=  -4a_i\int_{\mathbb{R}^2}\xi_{ij}\left( \frac{f'(V(\frac{x}{\e}))\phi_{1,\e}(\frac{x}{\e})}{2\e^2}+4\e^2e^{2W_0(x)}-2\e^2e^{W_0}\frac{A(x)}{|x|^2}+\sum_{i=1}^2c_{i,\e}^aZ_i\right)dx \\& = -4a_i\left[\int_{\mathbb{R}^2}  \frac{\xi_{ij}(\e x) f'(V(x))\phi_{1,\e}(x)}{2}+4\e^2\xi_{ij}e^{2W_0(x)}dx \right]
\\&=-4\delta_{ij}a_i\left[\intr \frac{f'(V(x)\phi_{1,\e}^a(x)}{8} + \frac{\e^2e^{2W_0(x)}}{(1+|x|^{2N_2+\gamma_1+2})^2}dx\right] \\&+O(\e^{2N_2+\gamma_1+2}+\e^{2\gamma_1-2N_1-2})\|\phi^a_{1,\e}\|_{L^\infty(\RN)} \\&=-4\delta_{ij}a_i\left[\intr \frac{f'(V(x)\phi_{1,\e}^a(x)(3-Z_{0}(\e x))}{16} + \frac{\e^2e^{2W_0(x)}}{(1+|x|^{2N_2+\gamma_1+2})^2}dx\right]\\&+O(\e^{2N_2+\gamma_1+2}+\e^{2\gamma_1-2N_1-2})\|\phi^a_{1,\e}\|_{L^\infty(\RN)} .\end{aligned}
\end{equation}To compute $\intr \frac{f'(V(x)\phi_{1,\e}^a(x)(3-Z_{0}(\e x))}{16} dx$, firstly,  we see from Lemma \ref{integrationbyparts} and \eqref{eqk}
\begin{equation}\label{reduction3}
\begin{aligned}0&= \int_{\mathbb{R}^2}\Big\{ \Delta
\phi^{a}_{2,\varepsilon}+2e^{W_{0}}\phi^{a}_{2,%
\varepsilon}-\frac{f'(V(\frac{x}{\varepsilon}))}{2\varepsilon^2}\phi^{a}_{1,\varepsilon}\Big(\frac{x}{\varepsilon}\Big)
-k^{a}_{2,\varepsilon}
-\sum^2_{i=1}c_{i,\varepsilon}^aZ_{i}\Big\}Z_{0}dx
\\&=\int_{\mathbb{R}^2} 2e^{W_{0}}\phi^{a}_{2,%
\varepsilon}-\frac{f'(V(x))}{2}\phi^{a}_{1,\varepsilon}(x)
-k^{a}_{2,\varepsilon}
-\sum^2_{i=1}c_{i,\varepsilon}^aZ_{i} dx\\&-\intr  \left\{\frac{f'(V(x))}{2}\phi^{a}_{1,\varepsilon}(x)Z_{0}(\e x)
+(k^{a}_{2,\varepsilon}
+\sum^2_{i=1}c_{i,\varepsilon}^aZ_{i})Z_{0}\right\}dx.\end{aligned}
\end{equation}
By $\phi_{1,\varepsilon}^{a}\in X_{\alpha}^1$,  Lemma \ref{w}, and Lemma \ref{estfory},  there is a constant $c_{\phi_{1,\e}^a}$ such that
\begin{equation*}\begin{aligned}
\phi_{1,\e}^a(x)&=c_{\phi_{1,\e}^a}-\frac{\ln|x|}{2\pi}\int_{\mathbb{R}^2}f'(V(y)) \phi_{1,\e}^a(y)-e^{W_a(y)}\phi_{2,\e}^a(y)-k_{1,\e}dy\\&+O(|x|^{-\frac{\alpha}{2}}\ln|x|\|\Delta \phi_{1,\e}^a\|_{Y_\alpha})\ \ \textrm{for}\ |x|\ge2.\end{aligned}
\end{equation*}Since $\phi_{1,\e}^a\in X_{\alpha}^1$, we have $\phi_{1,\e}^a\rho_1\in L^2(\RN)$ and thus
\begin{equation}\label{reduction4}\int_{\mathbb{R}^2}f'(V(y)) \phi_{1,\e}^a(y)-e^{W_a(y)}\phi_{2,\e}^a(y)-k_{1,\e}dy=0.\end{equation}
By \eqref{reduction3}-\eqref{reduction4}, we have
\begin{equation}\begin{aligned}\label{reduction5}  \int_{\mathbb{R}^2}f'(V ) \phi_{1,\e}^a \frac{(3-Z_{0}(\e y))}{16}dy &=\intr \frac{k_{1,\e}}{4}+\frac{k_{2,\e}}{8}(1+Z_{0})dy+O(|a| \| \phi^a_{2,\e}\rho_2\|_{L^2(\RN)})
\\&=\intr \frac{k_{1,\e}}{4}+\frac{\e^2e^{2W_0}}{(1+|y|^{2N_2+\gamma_1+2})}dy+O(|a| \e^{2-\frac{\alpha}{2}}).\end{aligned}\end{equation}
Then from \eqref{pf_reduction1} and \eqref{reduction5}, we prove the claim \eqref{step4_r}.

\medskip
\noindent
\textit{Step 4.} We claim  that for $j=1,2$,
 \begin{equation}
\begin{aligned}\label{eachterm3}&-\intr a_jk_{1,\e} dx=2a_j\e^2\intr e^{2W_0}dx+O(|a|\e^{3}+|a|^2\e^2).\end{aligned}
\end{equation} Indeed, by   Lemma \ref{mean},  \eqref{ang2},   and   radial symmetric property of $W_0$,   we see that
\begin{equation*}
\begin{aligned} &-\intr a_jk_{1,\e}dx
\\& =-a_j\intr \Big\{ e^{\tilde{W}_a(  x)+2\sum_{i=1}^{N_2}\ln|  x-\e q_i|+u_{2,\e}^*(  x)}- e^{W_a(  x)}  - 2\e^2e^{2\tilde{W}_a(  x)+4\sum_{i=1}^{N_2}\ln| x-\e q_i|+2u_{2,\e}^*(  x)}\Big\}dx
\\&=-a_j\intr \left( e^{ {W}_a(  x) }\frac{A(x)}{|x|^2}\e^2- 2\e^2e^{2{W}_a(  x)}\right)dx+O(|a|\e^3)
\\&=2a_j\e^2\intr e^{2W_0}dx+O(|a|\e^{3}+|a|^2\e^2).\end{aligned}
\end{equation*}
Now we complete the proof of the claim \eqref{eachterm3}.

\medskip
\noindent
\textit{Step 5.} In view of  \eqref{cl_1} and \eqref{eachterm5}-\eqref{eachterm3}, we get for $j=1,2$,
\begin{equation}\label{eachterm11}
\begin{aligned}  C_{j,\e}   \int_{\RN} Z_{j}^2dx
 &=2a_j\e^2\intr e^{2W_0}\left(1-\frac{2}{1+|x|^{2N_2+\gamma_1+2}}+16Z_{j}^2-\frac{2}{(1+|x|^{2N_2+\gamma_1+2})^2}\right)dx
\\&+O( \e^{3}+|a|\e^{2+\alpha}+|a|^2\e^{2-\frac{\alpha}{2}}).\end{aligned}
\end{equation}
By the change of variable $t=r^2$ and $\gamma_1>2N_1+2N_2+4$, we obtain that
\begin{equation*}
\begin{aligned}& \frac{1}{16\pi
 (N_2+\frac{\gamma_1}{2}+1)^4}\intr e^{2W_0}\left(1-\frac{2}{1+|x|^{2N_2+\gamma_1+2}}+16Z_{1}^2-\frac{2}{(1+|x|^{2N_2+\gamma_1+2})^2}\right)dx
\\&= 2\int^\infty_0\Big(\frac{8r^{4N_2+2{\gamma_1}}}{(1+r^{2N_2+%
{\gamma_1}+2})^5}-\frac{10r^{4N_2+2%
{\gamma_1}}}{ (1+r^{2N_2+{\gamma_1}+2})^6}-\frac{2r^{4N_2+2{\gamma_1}}}{(1+r^{2N_2+%
{\gamma_1}+2})^5}+\frac{r^{4N_2+2{\gamma_1}}}{(1+r^{2N_2+%
{\gamma_1}+2})^4}\Big)rdr\\&=  \int^\infty_0\Big(\frac{8t^{2N_2+\gamma_1}  }{(1+t^{ N_2+%
\frac{\gamma_1}{2}+1})^5}-\frac{10t^{2N_2+%
{\gamma_1}}}{ (1+t^{N_2+\frac{\gamma_1}{2}+1})^6}-\frac{2t^{2N_2+{\gamma_1}}}{(1+t^{N_2+%
\frac{\gamma_1}{2}+1})^5}+\frac{t^{2N_2+{\gamma_1}}}{(1+t^{N_2+%
\frac{\gamma_1}{2}+1})^4}\Big)dt\\&=  \int^\infty_0 \frac{-2t^{N_2+\frac{\gamma_1}{2}}}{(N_2+\frac{\gamma_1}{2}+1)}\frac{d}{dt}\Big[\frac{1}{(1+t^{ N_2+%
\frac{\gamma_1}{2}+1})^4}-\frac{1}{(1+t^{ N_2+%
\frac{\gamma_1}{2}+1})^5}\Big]-\frac{2t^{2N_2+{\gamma_1}}}{(1+t^{N_2+%
\frac{\gamma_1}{2}+1})^5}+\frac{t^{2N_2+{\gamma_1}}}{(1+t^{N_2+%
\frac{\gamma_1}{2}+1})^4} dt\\&=  \int^\infty_0\left(\frac{N_2+\frac{\gamma_1}{2}}{N_2+\frac{\gamma_1}{2}+1}-1\right)\frac{2t^{2N_2+%
{\gamma_1}}}{(1+t^{N_2+{\gamma_1}/2+1})^5}+\frac{t^{2N_2+%
{\gamma_1}}}{(1+t^{N_2+{\gamma_1}/2+1})^4}dt
\\&=  \int^\infty_0   \frac{\left(1-\frac{2}{N_2+\frac{\gamma_1}{2}+1}\right)t^{2N_2+%
{\gamma_1}} +t^{3N_2+\frac{3\gamma_1}{2}+1}}{(1+t^{N_2+{\gamma_1}/2+1})^5}dt:=T>0. \end{aligned}
\end{equation*}
We see that $C_{j,\e}^a\equiv0$ for $j=1,2,$ if $a=(a_{1},a_{2})\in \mathbb{R}^{2}$ satisfies
 \begin{equation}\label{finitered}
\begin{aligned}
\left(\begin{array}{ll} T \quad
\quad \  0 \\ \  \  0 \quad \quad
\quad T \end{array}\right) \left(\begin{array}{ll}a_1\\
a_2\end{array}\right)+O( \e  +|a|\e^{ \alpha}+|a|^2\e^{ -\frac{\alpha}{2}})=0. \end{aligned}
\end{equation}
Since $T>0$,   we can obtain $a_{\varepsilon}\in \mathbb{R}^{2}$ satisfying \eqref{finitered} and  $%
|a_{\varepsilon}|=O(\e)$, which implies  $C_{j,\e}^a
\equiv0$ for $j=1,2$.
Now we complete the proof of Proposition \ref{propgoal}.
\end{proof}
\subsection{Completion of Proof of Theorem \ref{T121}  and Corollary \ref{T122} }
By  Proposition \ref{existence_of_eta}-\ref{propgoal}, we obtain a
solution $(v_{1,\varepsilon },v_{2,\varepsilon })$ of \eqref{cs4} where
$
v_{1,\varepsilon }\left( x\right) =V_{1,\e}^a\left( x\right)
  +\eta _{1,\varepsilon }^{a}\left(
x\right),$ and $v_{2,\varepsilon }\left( x\right) =V_{2,\e}^a\left(x\right)
  +\eta _{2,\varepsilon }^{a}\left(
\e x\right).$
By using  \eqref{rmk0} and  $ \Vert \eta_{1,\varepsilon}^{a}\Vert_{L^{\infty}(\mathbb{R}%
^{2})}+\Big\|\frac{\eta_{2,\varepsilon}^a(x)}{\ln(2+|x|)}\Big\|_{L^\infty(\mathbb{R}^2)} =o(1),$   we complete the proof of  Theorem \ref{T121}.
Finally, Corollary
\ref{T122} follows from Theorem \ref{T121} since  when $N_1=1$, any non-topological (radial) solution of \eqref{cs1} is non-degenerate (see
\cite{CFL})
$\hfill\square$

\section{Invertibility of Linearized operator}\label{sectiona}

Recall the operator $\mathfrak{L}_0: X^2_\alpha\rightarrow Y_\alpha$
defined by $
\mathfrak{L}_0u=\Delta u +2e^{W_0(x)}u.
$ In  \cite{CI,CL}, it has been known  that the index   of a Fredholm operator $\mathfrak{L}_0 : X^2_\alpha\to Y_\alpha$ is  $1$,
\[\textrm{ind}\mathfrak{L}_0= \textrm{dim}(\textrm{ker} \mathfrak{L}_0) - \textrm{codim}(\textrm{ran}\mathfrak{L}_0)=1.\]
More precisely, we have the following result for the operator $Q\mathfrak{L}_0$.
\begin{lemma}\label{nondlemma}\cite{CI,CL}
\begin{equation*}
Q\mathfrak{L}_0:E_\alpha\rightarrow F_\alpha= \left \{
\begin{array}{l}
\{h\in Y_{\alpha}\ |\ \int_{\mathbb{R}^{2}}hZ_{i}dx=0,\ \ i=1,2\} \ \ %
\mbox{if}\ \frac{{\gamma_{1}}}{2}\in\mathbb{Z}, \\
Y_{\alpha} \ \ \mbox{if}\ \frac{{\gamma_{1}}}{2}\notin\mathbb{Z}.%
\end{array}
\right.
\end{equation*} is isomorphism.
\end{lemma}
\begin{proof}
(i) one-to-one: Suppose that $Q\mathfrak{L}_0u=Q[\Delta u +2e^{W_0}u]=0$ and $u\in E_\alpha$. By  Lemma \ref{w},   we have  $\int_{\mathbb{R}^2}Z_i(\Delta u +2e^{W_0}u)dx=0$ for $i=1,2$, and thus
 the definition of $Q$ implies that  \[\Delta u +2e^{W_0}u=0.\]
In view of  \cite[Lemma 2.4]{CI} and \cite[Lemma 2.2]{CL}, we have
\begin{equation*}
\left \{
\begin{array}{l}
\mbox{ker}\ \mathfrak{L}_0= \mbox{Span} \{Z_0, Z_1, Z_2\} \ \ \mbox{if}\ \frac{%
{\gamma_{1}}}{2}\in\mathbb{Z}, \\
\mbox{ker}\ \mathfrak{L}_0= \mbox{Span} \{Z_0\} \ \ \mbox{if}\ \frac{{\gamma_{1}}}{2}%
\notin\mathbb{Z},%
\end{array}
\right.
\end{equation*}
which implies
\begin{equation*}
u=\left \{
\begin{array}{l}
 c_0Z_0+c_1 Z_1+c_2 Z_2 \ \ \mbox{if}\ \frac{%
{\gamma_{1}}}{2}\in\mathbb{Z}, \\
c_0Z_0\ \ \mbox{if}\ \frac{{\gamma_{1}}}{2}%
\notin\mathbb{Z},%
\end{array}
\right.
\end{equation*}
for some constant $c_i\in\mathbb{R}$.
From $u\in E_\alpha$ and $\Delta Z_{i}+2e^{W_0}Z_i=0$, we see that
\begin{equation*}
0=\left \{
\begin{array}{l}\int_{\mathbb{R}^{2}}(-\Delta
Z_{i}+2e^{W_{0}}Z_{i})u dx= c_i\int_{\mathbb{R}^{2}} 4e^{W_{0}}Z_{i}^2 dx\ \ \mbox{for}\ \ i=0,1,2\  \ \mbox{if}\ \frac{%
{\gamma_{1}}}{2}\in\mathbb{Z}, \\
\int_{\mathbb{R}^{2}}(-\Delta
Z_{0}+2e^{W_{0}}Z_{0})u dx=c_0\int_{\mathbb{R}^{2}} 4e^{W_{0}}Z_{0}^2 dx\ \ \mbox{if}\ \frac{%
{\gamma_{1}}}{2}\notin\mathbb{Z},
\end{array}
\right.\end{equation*} which implies $c_i\equiv0$ and $u\equiv0$.
So we prove that $Q\mathfrak{L}_0$ is one-to-one from $E_\alpha$ to $F_\alpha$.

\medskip\noindent
(ii) onto:  In \cite[Proposition 2.2]{CI}, it was shown that
\begin{equation}\label{image}
\textrm{Im}\mathfrak{L}_0=F_{\alpha}= \left \{
\begin{array}{l}
\{h\in Y_{\alpha}\ |\ \int_{\mathbb{R}^{2}}hZ_{i}dx=0,\ \ i=1,2\} \ \ %
\mbox{if}\ \frac{{\gamma_{1}}}{2}\in\mathbb{Z}, \\
Y_{\alpha} \ \ \mbox{if}\ \frac{{\gamma_{1}}}{2}\notin\mathbb{Z}.%
\end{array}
\right.
\end{equation}In view of \eqref{image}, we see that for any $f\in F_\alpha$, there exists $u_f\in X_\alpha^2$ such that $(Q\mathfrak{L}_0)u_f=\mathfrak{L}_0u_f=\Delta u_f+2e^{W_0}u_f=f$. Let \[c_i:=\int_{\mathbb{R}^2}(-\Delta Z_i+2e^{W_0}Z_i)u_fdx\ \textrm{ for}\ \ i=1,2.\]
   We define \[\tilde{u}_f:=u_f-\sum_{i=0}^{2}\frac{c_i}{\int_{\mathbb{R}^2}4e^{W_0}Z_i^2dx}Z_i.\] Then we get that \[\tilde{u}_f\in E_\alpha, \ \ (Q\mathfrak{L}_0)\tilde{u}_f=\mathfrak{L}_0\tilde{u}_f=\Delta \tilde{u}_f+2e^{W_0}\tilde{u}_f=f.\]
   Thus we prove that $Q\mathfrak{L}_0$ is onto from $E_\alpha$ to $F_\alpha$. Now we complete the proof of Lemma \ref{nondlemma}.
\end{proof}
In \cite{DEM},  del Pino, Esposito, Musso recently studied  general linearized operators $%
\mathfrak{L}_a u=\Delta u +2e^{W_a}u$ for $a\neq0$. We also refer to \cite{LWY} for the $SU(n+1)$ Toda system.

\medskip

In order to prove Theorem \ref{thma}, firstly, we are going to   prove \eqref{importantineq}. To do it, we argue by contradiction and suppose that as $\e\to0$,
there are sequences  $a=a(\e)\to0$,    $(w_{1,\varepsilon},w_{2,\varepsilon})\in X_{\alpha}^{1}\times
E_{\alpha}$, and $(h_{1,\varepsilon},h_{2,\varepsilon})\in
Y_{\alpha}\times F_{\alpha}$ satisfying
$
\mathbb{QL}_{\varepsilon}^{a}(w_{1,\varepsilon},w_{2,\varepsilon
})=(h_{1,\varepsilon},h_{2,\varepsilon}),  $ and
\begin{equation}\left \{
\begin{array}{l}
 |\ln \e|^{-1}\left(\|w_{1,\e}\|_{L^{\infty}(\RN)}+ \Big\|\frac{w_{2,\e}}{\ln(2+|x|)}\Big\|_{L^{\infty}(\RN)}\right)+\sum_{i=1}^2 \Vert w_{i,\e}(x)\rho_{2}(x)\Vert_{L^2(\mathbb{R}^2)}=1,  \\ \lim_{\varepsilon\to0}
\left(\Vert h_{1,\varepsilon}\Vert_{Y_{\alpha}}+\Vert
h_{2,\varepsilon}\Vert_{Y_{\alpha}}\right)=0.  \label{eqto1}
\end{array}
\right.
\end{equation}There are
constants $ c_{i,\varepsilon} \in\mathbb{R}, i=1,2$ satisfying
\begin{equation}
\begin{aligned}\label{c} \left \{ \begin{array}{ll} &\Delta
w_{1,\varepsilon}+f'(V(x))w_{1,\varepsilon}-\varepsilon^2e^{W_a(\varepsilon
x)}w_{2,\varepsilon}(\varepsilon x)=h_{1,\varepsilon}(x), \\&\Delta
w_{2,\varepsilon}+2e^{W_0(x)}w_{2,\varepsilon}-\frac{f'(V(x/\varepsilon))}{2%
\varepsilon^2}w_{1,\varepsilon}(x/\varepsilon)=h_{2,\varepsilon}(x)+%
\sum_{i=1}^2c_{i,\varepsilon}Z_{i}(x),\ \textrm{and} \end{array}\right. \end{aligned}
\end{equation}
 \begin{equation}\label{intcfind}
\begin{aligned}&\int_{\mathbb{R}^2}\Big[\Delta
w_{2,\varepsilon}+2e^{W_0}w_{2,\varepsilon}-\frac{f'(V(x/\varepsilon))}{2%
\varepsilon^2}w_{1,\varepsilon}(x/\varepsilon)  -h_{2,\varepsilon}(x)-
\sum_{i=1}^2c_{i,\varepsilon}Z_{i}(x)\Big]Z_{j}(x)dx=0,\end{aligned}
\end{equation}for   $j=1,2.$ We note that   if $\frac{\gamma_1}{2}\notin\mathbb{N}$, then we let $c_{i,\e}=0$, $i=1,2$ since
  there is no projection and  $Q_0$ is the
identity map from $Y_\alpha$ to $Y_\alpha$ in this case.
The constant  $c_{i,\e}$, $i=1,2$ satisfies the following property.
\begin{lemma}\label{magofcc}
$\lim_{\e\to0}c_{i,\varepsilon}=0$  for  $i=1,2$.
\end{lemma}
\begin{proof}
The proof follows from Corollary \ref{magofc} and the assumption $\Vert w_{1,\e}\rho_{2}\Vert_{L^2(\mathbb{R}^2)}\le1$.\end{proof}
We have the following asymptotic behavior of $w_{i,\e}$, $i=1,2$ as $\e\to0$.
\begin{lemma}
\label{order}We have as $\varepsilon \to0$,

(i) $w_{1,\varepsilon}\to0\ \textrm{in}\  C^0_{\textrm{loc}}(\mathbb{R}^2)$  and $\lim_{\e\to0}\|f'(V )w_{1,\e} \|_{Y_\alpha}=0$,

(ii)
$w_{2,\varepsilon}\to0\ \textrm{in}\  C^0_{\textrm{loc}}(\mathbb{R}^2\setminus\{0\})$ and $\lim_{\e\to0}\|e^{W_0}w_{2,\e}\|_{Y_{\alpha}}=0$.
\end{lemma}
\begin{proof}
\textit{Step 1}. In this step, we are going to find the limit equations for $w_{1,\e}$ and $w_{2,\e}$.\\
 First, we note that \begin{equation}\label{elp001}
\begin{aligned}&\lim_{\e\to0} \|\e^2e^{W_a(\e x)}w_{2,\e}(\e x)\|_{L^2(\RN)}\le \lim_{\e\to0}\left(\e\sup_{x\in\RN}|e^{W_a(x)}\rho_2^{-1}(x)|\|w_{2,\e}\rho_2\|_{L^2(\RN)}\right)=0.
\end{aligned} \end{equation}Since $f'(V(x))\rho_2^{-1}(x)=O((1+|x|)^{-2\gamma_1+2N_1+1+\frac{\alpha}{2}})$, $\gamma_1>2N_1+2$, and $0<\alpha<\frac{1}{2}$, we see that  for any $\delta>0$,
\begin{equation}\label{elp0001}
\begin{aligned} &\lim_{\e\to0} \Big\| \frac{f'(V(\frac{x}{\e}))w_{1,\e}(\frac{x}{\e})}{ \e^2} \Big\|_{L^2(\RN\setminus B_{\delta}(0))}
\\&\le\lim_{\e\to0} \Big( \e^{-1}\sup_{x\in\RN\setminus B_{\frac{\delta}{\e}}(0)}|f'(V(x))\rho_2^{-1}(x)|\|w_{1,\e}\rho_2\|_{L^2(\RN)}\Big)=0.\end{aligned}
\end{equation}
There is a constant $c>0$, independent of $\e>0$, satisfying
\begin{equation}\label{elp01}
\begin{aligned} &\|\Delta w_{1,\e}\|_{L^2(\RN)}=\|f'(V)w_{1,\e}-\e^2e^{W_a(\e x)}w_{2,\e}(\e x)-h_{1,\e}\|_{L^2(\RN)}
\\&\le\sup_{x\in\RN}|f'(V(x))\rho_2^{-1}(x)|\|w_{1,\e}\rho_2\|_{L^2(\RN)}+\e\sup_{x\in\RN}|e^{W_a(x)}\rho_2^{-1}(x)|\|w_{2,\e}\rho_2\|_{L^2(\RN)}
\\&+\|h_{1,\e}\|_{L^2(\RN)}\le c. \end{aligned}
\end{equation}
For any $\delta>0$, there is a constant $c_{\delta}>0$, independent of $\e>0$, satisfying
\begin{equation}\label{elp02}
\begin{aligned} &\|\Delta w_{2,\e}\|_{L^2(\RN\setminus B_{\delta}(0))}  = \|2e^{W_0}w_{2,\e}-\frac{f'(V(\frac{x}{\e}))w_{1,\e}(\frac{x}{\e})}{2\e^2}-h_{2,\e}-\sum_{i=1}^2c_{i,\e}Z_{i} \|_{L^2(\RN\setminus B_{\delta}(0))}
\\&\le2\sup_{x\in\RN}|e^{W_0(x)}\rho_2^{-1}(x)|\|w_{2,\e}\rho_2\|_{L^2(\RN)} +
\Big\|h_{2,\e}-\sum_{i=1}^2c_{i,\e}Z_{i}\Big\|_{L^2(\RN)}\\&+(2\e)^{-1}\sup_{x\in\RN\setminus B_{\frac{\delta}{\e}}(0)}|f'(V(x))\rho_2^{-1}(x)|\|w_{1,\e}\rho_2\|_{L^2(\RN)}\le c_\delta.\end{aligned}
\end{equation}
Moreover, by \eqref{eqto1},  we have for any $R>0$,
\begin{equation}\label{elp03}\sum_{i=1}^2\|w_{i,\e}\|_{L^2(B_R(0))}\le (1+R)^{1+\frac{\alpha}{2}}\sum_{i=1}^2\|w_{i,\e}\rho_2\|_{L^2(\RN)}\le(1+R)^{1+\frac{\alpha}{2}}.\end{equation}
By
elliptic estimates and \eqref{elp001}-\eqref{elp03},   ${w}_{1,\varepsilon }\to {w}_{1}$ in
$C^{0}_{\text{loc}}(\mathbb{R}^{2})$ and ${w}_{2,\varepsilon }\to {w}_{2}$ in
$C^{0}_{\text{loc}}(\mathbb{R}^{2}\setminus\{0\})$, where
\begin{equation}\label{limiteqof12}
\begin{aligned}\left \{ \begin{array}{ll}
\Delta {w}_{1}+f'(V){w}_{1}=0\ \ \ \mbox{in}\ \ \ \mathbb{R}%
^{2},
\\ \Delta {w}_{2}+2e^{W_{0}}w_{2}=0\ \ \ \ \mbox{in}\ \ \ \mathbb{R}%
^{2}\setminus\{0\}.
 \end{array}\right. \end{aligned}
\end{equation}

\medskip
\noindent
\textit{Step 2}. We claim that $w_1\equiv0$ and $\lim_{\e\to0}\|f'(V(y))w_{1,\e}(y)\|_{Y_\alpha}=0.$   To prove this claim, we will show that $w_1\in L^\infty(\RN)$ and use the non-degeneracy of the first equation in \eqref{limiteqof12}.\\
By Fatou's Lemma and \eqref{eqto1},  we see that there is a constant $c>0$, satisfying
\begin{equation*}
\begin{aligned}&\|\Delta w_1\|_{Y_\alpha}=\|f'(V)(1+|x|)^{2+\alpha}w_1\rho_2\|_{L^2(\RN)}\le\sup_{x\in\RN}|f'(V)(1+|x|)^{2+\alpha}|\|w_1\rho_2\|_{L^2(\RN)}
\\&\le\sup_{x\in\RN}|f'(V)(1+|x|)^{2+\alpha}|\liminf_{\e\to0}\|w_{1,\e}\rho_2\|_{L^2(\RN)}\le c, \ \ \textrm{and}\end{aligned}
\end{equation*}
\begin{equation}\label{elp3}\|w_{1}\rho_2\|_{L^2(B_R(0))}\le \liminf_{\e\to0}\|w_{1,\e}\rho_2\|_{L^2(\RN)}\le1.\end{equation}
So  $w_1\in X_{\alpha}^2$. Together with Lemma \ref{w}, there is a constant $c_{w_1}$, independent of $x\in\RN$, such that
$
w_1(x)=c_{w_1}-\frac{1}{2\pi}\int_{\mathbb{R}^2}\ln|x-y|f'(V(y)) w_1(y)dy.$
By Lemma \ref{estfory}, we have if $|x|\ge 2$, then \begin{equation}\begin{aligned}\label{nondlinfty}
w_1(x)=c_{w_1}-\frac{\ln|x|}{2\pi}\int_{\mathbb{R}^2}f'(V(y)) w_1(y)dy+O(|x|^{-\frac{\alpha}{2}}\ln|x|\|f'(V)w_1\|_{Y_\alpha}).\end{aligned}
\end{equation}
By  $w_{1,\e}\in X_{\alpha}^1$  and  Lemma \ref{w}-\ref{estfory},         there is a  constant $c_{w_{1,\e}}$, independent of $x\in\RN$, such that
\begin{equation*}\begin{aligned}
w_{1,\e}(x)&=c_{w_{1,\e}}-\frac{\ln|x|}{2\pi}\int_{\mathbb{R}^2}f'(V(y)) w_{1,\e}(y)-e^{W_a(y)}w_{2,\e}(y)-h_{1,\e}dy +O(|x|^{-\frac{\alpha}{2}}\ln|x|\|\Delta w_{1,\e}\|_{Y_\alpha}).\end{aligned}
\end{equation*}Since $w_{1,\e}\in X_{\alpha}^1$, we have $w_{1,\e}\rho_1\in L^2(\RN)$, and thus
\begin{equation}\label{v1}\int_{\mathbb{R}^2}f'(V ) w_{1,\e} -e^{W_a}w_{2,\e}-h_{1,\e}dy=0,\end{equation} which implies
\begin{equation}\label{id1}\int_{\mathbb{R}^2}f'(V(y)) w_{1}(y)-e^{W_0(y)}w_{2}(y)dy=0.\end{equation}
Indeed, \eqref{id1} follows from $\lim_{\e\to0}\|h_{1,\e}\|_{L^1(\RN)}\le \lim_{\e\to0}\|\rho_2\|_{L^2(\RN)}\|h_{1,\e}\|_{Y_\alpha}=0$,
\begin{equation}\begin{aligned}\label{convw1}&\lim_{\e\to0}\intr f'(V )w_{1,\e} dy
=\intr f'(V)w_{1}dy,\   \textrm{and}\   \lim_{\e\to0}\intr e^{W_a }w_{2,\e} dy
=\intr e^{W_0}w_{2}dy.\end{aligned}
\end{equation}
By multiplying $Z_{0}$  on the second equation of \eqref{c}, we get from Lemma \ref{integrationbyparts} that
\begin{equation}\label{st1}
\begin{aligned}0&=\intr \Big(\Delta
w_{2,\varepsilon}+2e^{W_0}w_{2,\varepsilon}-\frac{f'(V(x/\varepsilon))}{2%
\varepsilon^2}w_{1,\varepsilon}(x/\varepsilon) -h_{2,\varepsilon} -
\sum_{i=1}^2c_{i,\varepsilon}Z_{i} \Big)Z_{0} dx
\\&=\intr \Big(2e^{W_0(x)}w_{2,\varepsilon}-\frac{f'(V(x/\varepsilon))}{2%
\varepsilon^2}w_{1,\varepsilon}(x/\varepsilon)-h_{2,\varepsilon}(x)-
\sum_{i=1}^2c_{i,\varepsilon}Z_{i}(x)\Big)dx
\\&-\intr \frac{f'(V(x))w_{1,\e}(x)}{2}Z_{0}(\e x)dx -\intr (h_{2,\varepsilon}(x)+
\sum_{i=1}^2c_{i,\varepsilon}Z_{i}(x))Z_{0}(x)dx.\end{aligned}
\end{equation}Here we note that
\begin{equation}\label{st2}
\begin{aligned}&\intr  f'(V(x))w_{1,\e}(x) Z_{0}(\e x)dx  =\intr f'(V (x))w_{1,\e}(x) \Big(1-
\frac{2|\e x|^{2N_2+\gamma_1+2}}{1+|\e x|^{2N_2+\gamma_1+2}}\Big)dx
\\&=\intr f'(V )w_{1,\e}  dx+O(\e^{2N_2+ \gamma_1+2}+
\e^{2\gamma_1-2N_1-2-\frac{\alpha}{2}})\|w_{1,\e}\rho_2\|_{L^2(\RN)}. \end{aligned}
\end{equation}
In view of \eqref{convw1}-\eqref{st2}, $\lim_{\e\to0}\|h_{2,\e}\|_{Y_\alpha}=0$, and $\lim_{\e\to0}c_{i,\e}=0$,  we get that
 \begin{equation}\label{st3}
\begin{aligned}0=\intr 2e^{W_0(x)}w_{2}(x)-f'(V(x))w_{1}(x)dx.\end{aligned}
\end{equation}By \eqref{id1} and \eqref{st3}, we have
  $\int_{\mathbb{R}^2}f'(V ) w_1 dx=\intr  e^{W_0 }w_{2} dx=0.$
Together with \eqref{nondlinfty}, we get   $w_1\in L^\infty(\RN)$.  The non-degeneracy condition for $V$ yields
$w_{1,\e}\to w_1\equiv0$ in $C_{\textrm{loc}}^0(\RN).$  Together with the decay of $e^V$ at infinity,
 we  get $\lim_{\e\to0}\|f'(V(y))w_{1,\e}(y)\|_{Y_\alpha}=0.$ This proves the claim.

\medskip
\noindent
\textit{Step 3}. We claim that $w_2\equiv0$ and $\lim_{\e\to0}\|e^{W_0}w_{2,\e}\|_{Y_{\alpha}}=0$. \\
For $2 \e \le |x|\le 2$ and $x_0\in\partial B_2(0)$, we see from Lemma \ref{estfory} that
\begin{equation}\begin{aligned}\label{smallarea}&w_{2,\e}(x)-w_{2,\e}(x_0)\\& =\frac{1}{2\pi}\intr \ln\frac{|x_0-y|}{|x-y|}(2e^{W_0}w_{2,\e}
-\frac{f'(V(\frac{y}{\e}))w_{1,\e}(\frac{y}{\e})}{2\e^2}-h_{2,\e}-\sum_{i=1}^2 c_{i,\e}Z_{i})dy
\\&=\frac{1}{2\pi}\ln\frac{|x|}{|x_0|}\intr
\frac{f'(V(y))w_{1,\e}(y)}{2}dy
\\&+\frac{1}{2\pi}\intr \Big(\ln\Big|\frac{x}{\e}-y\Big|-\ln\Big|\frac{x}{\e}\Big|+\ln\Big|\frac{x_0}{\e}\Big|-\ln\Big|\frac{x_0}{\e}-y\Big|\Big)
\frac{f'(V(y))w_{1,\e}(y)}{2}dy\\&+O(\|e^{W_0}w_{2,\e}\|_{Y_{\alpha}}+\|h_{2,\e}\|_{Y_{\alpha}}+\sum_{i=1}^2|c_{i,\e}|)
\\&=O(\|f'(V)w_{1,\e}\|_{Y_\alpha}+\|e^{W_0}w_{2,\e}\|_{Y_{\alpha}}+\|h_{2,\e}\|_{Y_{\alpha}}+\sum_{i=1}^2|c_{i,\e}|) +O(\ln|x|\|f'(V)w_{1,\e}\|_{Y_\alpha}).
\end{aligned}\end{equation}
By Step 2, there is a constant $c>0$, independent of $\e>0$ and $x\in B_1(0)\setminus\{0\}$, such that
$|w_2(x)|=\lim_{\e\to0}|w_{2,\e}(x)|\le c,$ which implies from  \eqref{limiteqof12}
that $\Delta {w}_{2}+2e^{W_{0}}w_{2}=0\  \mbox{in}\  \mathbb{R}%
^{2}.$
By Fatou's Lemma and \eqref{eqto1},  we see that there is a constant $c>0$, satisfying
\begin{equation*}
\begin{aligned}&\|\Delta w_2\|_{Y_\alpha}=\|2e^{W_{0}}(1+|x|)^{2+\alpha}w_2\rho_2\|_{L^2(\RN)}\le\sup_{x\in\RN}|2e^{W_{0}}(1+|x|)^{2+\alpha}|\|w_2\rho_2\|_{L^2(\RN)}
\\&\le\sup_{x\in\RN}|2e^{W_{0}}(1+|x|)^{2+\alpha}|\liminf_{\e\to0}\|w_{2,\e}\rho_2\|_{L^2(\RN)}\le c,  \end{aligned}
\end{equation*}  and $\|w_{2}\rho_2\|_{L^2(\RN)}\le \liminf_{\e\to0}\|w_{2,\e}\rho_2\|_{L^2(\RN)}\le1.$
So   ${w}_2\in X^{2}_{\alpha}$.
Then by  \eqref{kerinx2},    ${w}_{2}=\sum ^{2}_{i=0}c_{i}Z_{i}$ for some constants
$\{c_{i}\}_{i=0}^2$.
Since $w_{2,\varepsilon}\in E_{\alpha}$, $\Delta Z_{i}+2e^{W_0(x)}Z_{i}=0$,  and ${w}_{2,\varepsilon}\to {w}_{2}=\sum ^{2}_{i=0}c_{i}Z_{i}$ in $C^0_{\textrm{loc}}
(\mathbb{R}^2\setminus\{0\})$,  we get
\begin{equation*}
\begin{aligned} 0&=\lim_{\varepsilon \to0} \int_{\mathbb{R}^2}
{w}_{2,\varepsilon}(-\Delta Z_{i}+2e^{W_0(x)}Z_{i})dx
 =4\lim_{\varepsilon \to0}\int_{\mathbb{R}^2}
{w}_{2,\varepsilon}e^{W_0}Z_{i}dx=4c_i\int_{\mathbb{R}^2}
e^{W_0}Z_{i}^2dx,\ \ \ i=0,1,2, \end{aligned}
\end{equation*}
which implies  $\{c_{i}\}_{i=0,1,2}=0$ and ${w}_{2}\equiv0$.  By  the decay of $e^{W_0}$ at infinity and at zero, we  get
$\lim_{\e\to0}\|e^{W_0}w_{2,\e}\|_{Y_{\alpha}}=0$. This completes the proof.
\end{proof}
By using Lemma \ref{w} and Lemma \ref{order}, we have the following result.
\begin{lemma}\label{onetoone}$\lim_{\e\to0}\left[|\ln \e|^{-1}\left(\|w_{1,\e}\|_{L^{\infty}(\RN)}+ \Big\|\frac{w_{2,\e}}{\ln(2+|x|)}\Big\|_{L^{\infty}(\RN)}\right)+\sum_{i=1}^2 \Vert w_{i,\e}(x)\rho_{2}(x)\Vert_{L^2(\mathbb{R}^2)}\right]=0.$\end{lemma}\begin{proof}
For $|x|\le 2 \e$ and $x_0\in\partial B_2(0)$, we see from Lemma  \ref{w}-\ref{estfory} that
\begin{equation}\label{convt1}\begin{aligned}&w_{2,\e}(x)-w_{2,\e}(x_0)\\&=\frac{1}{2\pi}\intr \ln\frac{|x_0-y|}{|x-y|}(2e^{W_0}w_{2,\e}
-\frac{f'(V(\frac{y}{\e}))w_{1,\e}(\frac{y}{\e})}{2\e^2}-h_{2,\e}-\sum_{i=1}^2 c_{i,\e}Z_{i})dy
\\&=  \frac{1}{2\pi}\ln\frac{\e}{|x_0|}\intr
\frac{f'(V(y))w_{1,\e}(y)}{2}dy
\\&+\frac{1}{2\pi}\intr \Big(\ln\Big|\frac{x}{\e}-y\Big|+\ln\Big|\frac{x_0}{\e}\Big|-\ln\Big|\frac{x_0}{\e}-y\Big|\Big)
\frac{f'(V(y))w_{1,\e}(y)}{2}dy\\&+O(\|e^{W_0}w_{2,\e}\|_{Y_{\alpha}}+\|h_{2,\e}\|_{Y_{\alpha}}+\sum_{i=1}^2|c_{i,\e}|)
\\&=O(\|f'(V)w_{1,\e}\|_{Y_\alpha}|\ln\e|+\|e^{W_0}w_{2,\e}\|_{Y_{\alpha}}+\|h_{2,\e}\|_{Y_{\alpha}}+\sum_{i=1}^2|c_{i,\e}|).
\end{aligned}\end{equation}
For $|x|\ge 2 $ and $x_0\in\partial B_2(0)$, we see from Lemma \ref{w}-\ref{estfory} that
\begin{equation}\label{convt2}\begin{aligned}&w_{2,\e}(x)-w_{2,\e}(x_0)
 \\&=\frac{1}{2\pi}\intr \Big(\ln\Big|\frac{x}{\e}-y\Big|-\ln\Big|\frac{x}{\e}\Big|+\ln\Big|\frac{x_0}{\e}\Big|-\ln\Big|\frac{x_0}{\e}-y\Big|\Big)
\frac{f'(V(y))w_{1,\e}(y)}{2}dy
\\&+\frac{1}{2\pi}\ln\frac{|x|}{|x_0|}\intr\Big(
-2e^{W_0}w_{2,\e}+\frac{f'(V(y))w_{1,\e}(y)}{2}
 +h_{2,\e}+\sum_{i=1}^2 c_{i,\e}Z_{i}\Big)dy\\&+O(\|e^{W_0}w_{2,\e}\|_{Y_{\alpha}}+\|h_{2,\e}\|_{Y_{\alpha}}+\sum_{i=1}^2|c_{i,\e}|)
\\&=O(\|f'(V)w_{1,\e}\|_{Y_\alpha}+\|e^{W_0}w_{2,\e}\|_{Y_{\alpha}}+\|h_{2,\e}\|_{Y_{\alpha}}+\sum_{i=1}^2|c_{i,\e}|)(1+\ln|x|).
\end{aligned}\end{equation}
From Lemma \ref{order}, \eqref{convt1}-\eqref{convt2}, and \eqref{smallarea}, we get that\begin{equation}\label{for2}\lim_{\e\to0}\left[|\ln \e|^{-1}  \Big\|\frac{w_{2,\e}}{\ln(2+|x|)}\Big\|_{L^{\infty}(\RN)} + \Vert w_{2,\e} \rho_{2} \Vert_{L^2(\mathbb{R}^2)}\right]=0.
\end{equation}

For $2\le |x|\le \frac{2}{\e}$,  we see from Lemma \ref{w}-\ref{estfory} that
\begin{equation}\label{convtw1}\begin{aligned}&w_{1,\e}(x)-w_{1,\e}(0)\\& =\frac{1}{2\pi}\intr
\ln\frac{|y|}{|x-y|}(f'(V(y))w_{1,\e}(y)-\e^2e^{W_a(\e y)}w_{2,\e}(\e y)
-h_{1,\e})dy
\\&=\frac{1}{2\pi}\intr  \ln\frac{|x|}{|x-y|}(f'(V )w_{1,\e}
-h_{1,\e})dy +\frac{1}{2\pi}\intr \ln\frac{|\e x-y|}{|y|}e^{W_a(y)}w_{2,\e}(y)dy\\&-\frac{\ln|x|}{2\pi}\intr (f'(V )w_{1,\e}
-h_{1,\e})dy
+\frac{1}{2\pi}\intr \ln|y|(f'(V )w_{1,\e}
-h_{1,\e})dy
\\&=O(\|f'(V)w_{1,\e}\|_{Y_\alpha}+\|e^{W_a}w_{2,\e}\|_{Y_{\alpha}}+\|h_{1,\e}\|_{Y_{\alpha}})(1+\ln|x|).
\end{aligned}\end{equation}
For $|x|\ge \frac{2}{ \e}$, we see from Lemma \ref{w}-\ref{estfory} that
\begin{equation}\label{convtw2}\begin{aligned}&w_{1,\e}(x)-w_{1,\e}(0)
\\&=\frac{1}{2\pi}\intr (\ln|\e x-y|-\ln|\e x|-\ln|y|)e^{W_a(y)}w_{2,\e}(y)dy+\frac{\ln|\e x|}{2\pi}\intr e^{W_a(y)}w_{2,\e}(y)dy\\&-\frac{\ln|x|}{2\pi}\intr (f'(V )w_{1,\e}
-h_{1,\e})dy+O(\|f'(V)w_{1,\e}\|_{Y_\alpha}+\|h_{1,\e}\|_{Y_{\alpha}})
\\&=-\frac{\ln|x|}{2\pi}\intr (f'(V )w_{1,\e}-e^{W_a(y)}w_{2,\e}(y)
-h_{1,\e})dy\\&+
O(\|f'(V)w_{1,\e}\|_{Y_\alpha}+\|e^{W_a}w_{2,\e}\|_{Y_{\alpha}}+\|h_{1,\e}\|_{Y_{\alpha}})(|\ln\e|)\\&=O(\|f'(V)w_{1,\e}\|_{Y_\alpha}+\|e^{W_a}w_{2,\e}\|_{Y_{\alpha}}+\|h_{1,\e}\|_{Y_{\alpha}})(|\ln\e|),
\end{aligned}\end{equation}here we used \eqref{v1}, that is, $\intr (f'(V )w_{1,\e}-e^{W_a(y)}w_{2,\e}(y)
-h_{1,\e})dy=0$.
 Lemma \ref{order}   and \eqref{convtw1}-\eqref{convtw2} yield \begin{equation}\label{for1}\lim_{\e\to0}\left[|\ln \e|^{-1}\ \|w_{1,\e}\|_{L^{\infty}(\RN)}+  \Vert w_{1,\e} \rho_{2} \Vert_{L^2(\mathbb{R}^2)}\right]=0.
\end{equation}
In view of \eqref{for2} and \eqref{for1}, we complete  the proof of Lemma \ref{onetoone}.\end{proof}
\textbf{Proof of Theorem \ref{thma}:}
We see that  Lemma \ref{onetoone} contradict our assumption  \eqref{eqto1} and prove the inequality \eqref{importantineq}.
This   inequality \eqref{importantineq} implies that $\mathbb{QL}_{\varepsilon}^{a}$ is one-to-one from   $X_{\alpha }^{1}\times E_{\alpha }$ to $Y_{\alpha }\times
F_{\alpha }$.

Finally,   we   need to show that $\mathbb{QL}_{\varepsilon}^{a}$ is onto from $X_{\alpha }^{1}\times E_{\alpha }$ to $Y_{\alpha }\times
F_{\alpha }$.
Now  we define $\mathbb{QL}$ such that
$ \mathbb{QL}(w_1,w_2):=(\Delta w_1+f'(V)w_1, {Q}[\Delta
w_2+2e^{W_0}w_2]). $
In view of Lemma \ref{nondlemma},    $\mathbb{QL}$ is an isomorphism from $X_{\alpha }^{1}\times
E_{\alpha }$ to $Y_{\alpha }\times F_{\alpha }$, which implies $\textrm{ind}(\mathbb{QL})=0.$
Moreover, we see  \[
\mathbb{QL}_{\varepsilon}^{a}(w_1,w_2)=\mathbb{QL}(w_1,w_2)+ \mathbb{Q}(\mathbb{L}^a_{\e}-\mathbb{L}))(w_1,w_2).\]
Since $\mathbb{Q}(\mathbb{L}^a_{\e}-\mathbb{L})(w_{1},w_{2})$ is a compact operator,   Fredholm alternative implies \[\textrm{dim}(\textrm{ker} (\mathbb{QL}_{\varepsilon}^{a})) - \textrm{codim}(\textrm{ran}(\mathbb{QL}_{\varepsilon}^{a}))=\textrm{ind}(\mathbb{QL}_{\varepsilon}^{a})=\textrm{ind}(\mathbb{QL})=0.\]  By the inequality \eqref{importantineq}, we get that  $\textrm{dim}(\textrm{ker} (\mathbb{QL}_{\varepsilon}^{a}))=0$, and  $\textrm{codim}(\textrm{ran}(\mathbb{QL}_{\varepsilon}^{a}))=0.$ Now we get that $\mathbb{QL}_{\varepsilon}^{a}$ is onto from $X_{\alpha }^{1}\times E_{\alpha }$ to $Y_{\alpha }\times
F_{\alpha }$ and prove Theorem \ref{thma}.
$\hfill\square $
\section{Basic Estimates}\label{sectionb}Let $(v_1,v_2)$ be a solution of \eqref{cs4} and  $V$ satisfy  \eqref{chernsimonseq}.
We shall show that by a shift of origin,  \eqref{sumqeqc} holds, that is,
$
2\sum^{N_2}_{i=1}q_i+\overrightarrow{c}=0,$ where $\overrightarrow{c} =\frac{1}{2\pi}\int_{\mathbb{R}^2}
ye^{V(y)}(1-2e^{V(y)})dy.
$
 Indeed,  for any  fixed vector $y_0\in\mathbb{R}^2$, we consider the change of coordinate such that $\tilde{y}:=y-y_0$. Then $\tilde{v}_i(\tilde{y}):=v_i(\tilde{y}+y_0)=v_i(y), \ \tilde{p}_j:=p_j-y_0, \ \tilde{q}_j:=q_j-y_0$ satisfy
  \begin{equation*}
\left\{
\begin{array}{l}
\Delta \tilde{v}_{1}+2e^{\tilde{v}_{1}}\left( 1-2e^{\tilde{v}_{1}}\right) -e^{\tilde{v}_{2}}\left(
1-2e^{\tilde{v}_{2}}-e^{\tilde{v}_{1}}\right) =4\pi \sum_{j=1}^{N_{1}}\delta _{\tilde{p}_{j}} \\
\Delta \tilde{v}_{2}+2e^{\tilde{v}_{2}}\left( 1-2e^{\tilde{v}_{2}}\right) -e^{\tilde{v}_{1}}\left(
1-2e^{\tilde{v}_{1}}-e^{\tilde{v}_{2}}\right) =4\pi \sum_{j=1}^{N_{2}}\delta _{\tilde{q}_{j}}%
\end{array}%
\right. \mbox{in }\mathbb{R}^{2}.
\end{equation*}
The function  $\tilde{V}(\tilde{y}):=V(\tilde{y}+y_0)$  satisfies $
\Delta \tilde{V}+f(\tilde{V}(\tilde{y}))=4\pi\sum_{j=1}^{N_1}\delta_{\tilde{p}_j} $  and $
\int_{\mathbb{R}^2}f(\tilde{V}(\tilde{y}))d\tilde{y}= 4\pi\gamma_1.$
Let $\overrightarrow{\tilde{c}}:=\frac{1}{2\pi}\int_{\mathbb{R}^2}
\tilde{y}e^{\tilde{V}(\tilde{y})}(1-2e^{\tilde{V}(\tilde{y})})d\tilde{y}$.
Then $\overrightarrow{\tilde{c}} =\frac{1}{4\pi}\int_{\mathbb{R}^2}
(y-y_0)f(V(y))dy=\overrightarrow{c}-\frac{y_0}{4\pi}\int_{\mathbb{R}^2}f(V)dy.$
From the maximum principle,   $f(V(y))\ge0$ in $\mathbb{R}^2$, and  $2N_2+\frac{1}{4\pi}\int_{\mathbb{R}^2}f(V )dy>0.$
By choosing    $y_0=\frac{2\sum^{N_2}_{i=1} q_i +\overrightarrow{c}}{2N_2+\frac{1}{4\pi}\int_{\mathbb{R}^2}f(V )dy}$,  we get
$
2\sum^{N_2}_{i=1}\tilde{q_i}+\overrightarrow{\tilde{c}}=2\sum^{N_2}_{i=1} q_i +\overrightarrow{c}-y_0 (2N_2+\frac{1}{4\pi}\int_{\mathbb{R}^2}f(V )dy )=0
,$  which implies  \eqref{sumqeqc} with the   coordinate $\tilde{y}$.

Now we    show that $u_{2,\varepsilon}^{*}(x)-{\gamma_1} \ln|x|$ (resp. $%
u_{2,\varepsilon}^{*}$) on $[B_{\varepsilon R_{0}}(0)]^c$ (resp. $%
B_{\varepsilon R_{0}}(0)$) is enough small.

\begin{lemma}
\label{diffgreen} There are constants $c_1, c_2>0$, independent of $\e>0$, satisfying

{(i)} If $|x|\geq \varepsilon R_{0}$, then $u_{2,\varepsilon}^{*}(x)-%
{\gamma_1} \ln|x|=-\frac{x\cdot \overrightarrow{c}}{|x|^{2}}\varepsilon-d(x)%
\varepsilon^{2}+O\Big(\frac{\varepsilon^3}{|x|^3}(1+|\ln|x||)\Big)$ and
  $u_{2,\varepsilon}^*(x)-{\gamma_1}\ln|x|\le c_1$. Here $
d(x)=\frac{1}{8\pi}\int_{\mathbb{R}^2}\Big(\frac{2(x\cdot y)^2}{|x|^4}-%
\frac{|y|^2}{|x|^2}\Big)f(V(y))dy$ (see \eqref{definitionofdx}).

{(ii)} If $|x|\leq \varepsilon R_{0}$, then $u_{2,\varepsilon}^{*}(x)\leq
{\gamma_1} \ln \varepsilon+c_{2}$.
\end{lemma}

\begin{proof}
(i)  First,  let $N_1\ge 2$.   By Taylor's expansion theorem, we get that
\begin{equation}  \label{tayloroflog}
\begin{aligned} \ln|x-\varepsilon y|^2=\ln|x|^2-\frac{2x\cdot
y}{|x|^2}\varepsilon+\Big(\frac{|y|^2}{|x|^2}-\frac{2(x\cdot
y)^2}{|x|^4}\Big)\varepsilon^2+R(\varepsilon,x,y), \end{aligned}
\end{equation}
where $R(\varepsilon,x,y)=O\Big(\frac{|y|^3\varepsilon^3}{|x|^3}+\frac{%
|y|^4\varepsilon^4}{|x|^4}\Big)+\frac{(-2x\cdot
y\varepsilon+\varepsilon^2|y|^2)^3}{3(\theta_{x,y,\varepsilon}|x-\varepsilon
y|^2+(1-\theta_{x,y,\varepsilon})|x|^2)^3}$ for some $\theta_{x,y,%
\varepsilon}\in(0,1)$. By using \eqref{tayloroflog}, we get that
\begin{equation*}
\begin{aligned} &u_{2,\varepsilon}^*(x)-{\gamma_1} \ln|x|
=\frac{1}{8\pi}\int_{\mathbb{R}^2}(\ln|x-\varepsilon y|^2-\ln
|x|^2)f(V(y))dy \\&=\frac{1}{8\pi}\int_{\mathbb{R}^2}\Big\{-\frac{2x\cdot
y}{|x|^2}\varepsilon+\Big(\frac{|y|^2}{|x|^2}-\frac{2(x\cdot
y)^2}{|x|^4}\Big)\varepsilon^2 +R(\varepsilon,x,y)\Big\}f(V(y))dy
\\&=-\frac{x\cdot \overrightarrow{c}}{|x|^2}\varepsilon-d(x)\varepsilon
^2+\frac{1}{8\pi}\int_{\mathbb{R}^2}R(\varepsilon,x,y)f(V(y))dy.
\end{aligned}
\end{equation*}
By  $|f(V(y))|=(1+|y|)^{-2\gamma_1+2N_1}$,  $\gamma_1>2N_1+2$, and $N_1\ge2$, we get that  $\int_{\mathbb{R}^2}|y|^{6}|f(V(y))|dy$ is finite.
Moreover, we obtain  that for some constant $c>0$, independent of $\e>0$,
\begin{equation}
\begin{aligned}\label{greendiff0}
&\Big|\int_{\mathbb{R}^2}R(\varepsilon,x,y)f(V(y))dy\Big|
\le O \Big(  \Big|\int_{|x-\varepsilon y|\le
\frac{|x|}{2}}R(\varepsilon,x,y)f(V )dy\Big|
+ \frac{\varepsilon^3}{|x|^3}+\frac{\varepsilon^4}{|x|^4}+\frac{%
\varepsilon^6}{|x|^6}\Big). \end{aligned}
\end{equation}
By \eqref{tayloroflog} and the  change of variables, we also see that
\begin{equation*}
\begin{aligned} &\int_{|x-\varepsilon y|\le
\frac{|x|}{2}}R(\varepsilon,x,y)f(V )dy
=\int_{|x-y|\le
\frac{|x|}{2}}\Big\{ \ln\frac{|x- y|^2}{|x|^2}+\frac{2x\cdot
y-|y|^2}{|x|^2}+\frac{2(x\cdot y)^2}{|x|^4}\Big\}
\frac{f(V(\frac{y}{\varepsilon}))}{\varepsilon^2}dy \end{aligned}
\end{equation*}
Since $|x-y|\le
\frac{|x|}{2} \Rightarrow\frac{|x|}{2}\le|y|\le\frac{3|x|}{2}$, we see that if  $\varepsilon R_0 \le|x|$, then  $\frac{\varepsilon R_0}{2} \le\frac{|x|}{2}\le|y|\le\frac{3|x|}{2}$ and
\begin{equation*}
\begin{aligned} &\Big|\int_{|x-\varepsilon y|\le
\frac{|x|}{2}}R(\varepsilon,x,y)f(V )dy\Big| \le\Big|\int_{|x-y|\le \frac{|x|}{2}} \Big(\ln\frac{|x-
y|^2}{|x|^2}\Big) \frac {f(V(\frac{y}{\varepsilon}))}{\varepsilon^2}dy\Big|
\\& +\Big|\int_{|x-y|\le \frac{|x|}{2}}\Big(\frac{2x\cdot
y-|y|^2}{|x|^2}+\frac{2(x\cdot y)^2}{|x|^4}\Big)
\frac{f(V(\frac{y}{\varepsilon} ))}{\varepsilon^2}dy\Big|  \le   O\Big( \Big(\frac{\varepsilon}{|x|}\Big)^{2{\gamma_1}-2N_1-2}(1+|\ln|x||)\Big).
\end{aligned}
\end{equation*}
If $ \varepsilon R_0\le |x|$, then there are   constants  $c_2>0$ and $\theta_{x,y,\varepsilon}\in(0,1)$ such that
\begin{equation}\label{greendiff3}
\begin{aligned}
&u_{2,\varepsilon}^*(x)-{\gamma_1}\ln|x|\le\frac{1}{4\pi}\int_{\mathbb{R}^2}(\ln(|x|+\varepsilon |y|)-\ln|x|)f(V(y))dy
\\&\le \frac{1}{4\pi}\int_{\mathbb{R}^2}\frac{\varepsilon|y||f(V(y))|}{|x|+\theta_{x,y,\varepsilon}(\varepsilon|y|)}dy
\le\frac{\varepsilon}{4\pi|x|}\int_{\mathbb{R}^2}|y||f(V(y))|dy
\le c_2.
\end{aligned}
\end{equation}
By using \eqref{greendiff0}-\eqref{greendiff3}, we prove Lemma \ref%
{diffgreen}-$(i)$ for $N_1\ge2$.

Now we consider the case $N_1=1$.
We remind that
\begin{equation}\label{vatinfty}
V(x)=(-2\gamma_1+2)\ln|x -p_1|+ I_{p_1} + O(|x -p_1|^{-\sigma})\ \textrm{for}\ |x |\gg1.
\end{equation}
Since $N_1=1$ and $V$ is a radial function with respect to $p_1$, from \cite[Lemma 2.6]{CFL}, we have that $\sigma=2\gamma_1-4$.
By using $u_{2,\varepsilon}^*(x)= -\frac{V ( \frac{x}{\varepsilon} )-2\ln|\frac{x}{\varepsilon}-p_1|}{2}+
\gamma_1\ln\varepsilon+\frac{I_{p_1}}{2}$ and \eqref{vatinfty}, we see that  if $|x|\ge \e R_0$, then
from $\gamma_1>2N_1+2=4$ and \eqref{tayloroflog},
\begin{equation*}\begin{aligned}
u_{2,\varepsilon}^*(x)- \gamma_1\ln |x| &=\frac{\gamma_1}{2}(\ln | x-\e p_1|^2- \ln |x|^2)+ O\Big(\Big|\frac{x-\e p_1}{\e}\Big|^{4-2\gamma_1}\Big)\\&=-\frac{(x\cdot
p_1)\gamma_1}{|x|^2}\varepsilon+\frac{\gamma_1}{2}\Big(\frac{|p_1|^2}{|x|^2}-\frac{2(x\cdot
p_1)^2}{|x|^4}\Big)\varepsilon^2+O\Big(\frac{\e^3}{|x|^3 }\Big). \end{aligned}
\end{equation*}
Since $V$ is a radial function with respect to $p_1$, for   $y=(y_1,y_2)$, and $p_1=(p_{1,1},p_{1,2})$, we see that $
\overrightarrow{c}=\frac{1}{4\pi}\int_{\mathbb{R}^2}
(y +p_1)f(V(y+p_1))dy=p_1\gamma_1,$ and
\begin{equation*}\begin{aligned}
& d(x) =\frac{1}{8\pi}\int_{\mathbb{R}^2}\Big(\frac{(x_1^2-x_2^2)(y_1^2-y_2^2)+4x_1x_2y_1y_2+4(x_1y_1+x_2y_2)(x_1p_{1,1}+x_2p_{1,2})}{|x|^4}
\\&+\frac{2(x\cdot p_1)^2-|x|^2|p_1|^2-2|x|^2(y_1p_{1,1}+y_2p_{1,2}) }{|x|^4}\Big)f(V(y+p_1))dy
=\frac{\gamma_1}{2}\Big(\frac{2(x\cdot p_1)^2}{|x|^4}-\frac{|p_1|^2}{|x|^2}\Big).
\end{aligned}
\end{equation*}
Thus when $N_1=1$, if $|x|\ge\e R_0$, then for some constant $c_2$, independent of $\e>0$,
\[u_{2,\varepsilon}^{*}(x)-%
{\gamma_1} \ln|x|=-\frac{x\cdot \overrightarrow{c}}{|x|^{2}}\varepsilon-d(x)%
\varepsilon^{2}+O\Big(\frac{\varepsilon^3}{|x|^3}\Big)\le c_2.\]

(ii) If $|x|\le \varepsilon R_0$ and $N_1\ge2$,   there is a constant $c_1>0$, independent of $\e>0$,
satisfying
\begin{equation*}
\begin{aligned} u_{2,\varepsilon}^*(x)&\le
\frac{1}{4\pi}\int_{\mathbb{R}^2}\ln(|x|+\varepsilon |y|)f(V(y))dy  \le
\frac{1}{4\pi}\int_{\mathbb{R}^2}\ln[\varepsilon(R_0+|y|)]f(V(y))dy\le
{\gamma_1} \ln \varepsilon+c_1. \end{aligned}
\end{equation*}
If $|x|\le \e R_0$ and $N_1=1$, then there is a constant $c_1>0$, independent of $\e>0$, such that  $u_{2,\varepsilon}^*(x)= -\frac{V ( \frac{x}{\varepsilon} )-2\ln|\frac{x}{\varepsilon}-p_1|}{2}+
\gamma_1\ln\varepsilon+\frac{I_{p_1}}{2}\le \gamma_1\ln\varepsilon+c_1$. Here we used that $V (x )-2\ln|x-p_1|$ is smooth and $\Delta( V (x )-2\ln|x-p_1|)+f(V(x))=0$  in $\RN$. This completes the proof.
\end{proof}
Here we get some results for $\tilde{W}_a(\varepsilon x)$.
\begin{lemma}
\label{mean} (i) $|\tilde{W}_a(0)-\tilde{W}_a(\varepsilon x)|\le O( |\varepsilon x|^{2N_2+{\gamma_1}+2}+|a||\varepsilon
x|^{N_2+\frac{{\gamma_1}}{2}+1}). $

(ii) $ e^{\tilde{W}_a(x)+2\sum^{N_2}_{i=1}\ln|x-\varepsilon
q_i|+u_{2,\varepsilon}^*(x) }=\left\{
\begin{array}{ll}O(\varepsilon^{2N_2+{\gamma_1}})\quad \mbox{if}\
|x|\le\varepsilon R_0,\\   e^{W_a(x)}\Big\{
1+\frac{\varepsilon^2 A(x)}{|x|^2}+O ( \frac{\varepsilon^3}{|x|^3}
  (1+|\ln|x||)^3)\Big\}\ \mbox{if}\  |x|\ge\varepsilon R_0.
\end{array}\right. $
\end{lemma}
\begin{proof}
(i)  By  the mean value theorem, we have   for $\theta_{x}\in(0,1)$,
\begin{equation*}
\begin{aligned}&\Big|\frac{\tilde{W}_a(0)-\tilde{W}_a(\varepsilon x)}{2}\Big|
 =\Big|\ln \Big(\frac{1+|(\varepsilon
x)^{N_2+\frac{{\gamma_1}}{2}+1}+a|^2}{1+|a|^2}\Big)\Big|
\\&=\Big|\frac{|(\varepsilon
x)^{N_2+\frac{{\gamma_1}}{2}+1}+a|^2-|a|^2}{1+\theta_x|(\varepsilon
x)^{N_2+\frac{{\gamma_1}}{2}+1}+a|^2+(1-\theta_x)|a|^2}\Big|  \le O(
|\varepsilon x|^{2N_2+{\gamma_1}+2}+|a||\varepsilon x|^{N_2+\frac{{\gamma_1}}{2}+1}).
\end{aligned}
\end{equation*}
 (ii) By Lemma \ref{diffgreen}, we get     $e^{\tilde{W}%
_a(x)+2\sum^{N_2}_{i=1}\ln|x-\varepsilon
q_i|+u_{2,\varepsilon}^*(x)}=O(\varepsilon^{2N_2+{\gamma_1}})$ if $|x|\le
\varepsilon R_{0}$.

  By   Taylor expansion of $\prod^{N_{2}}_{i=1}|x-\varepsilon
q_{i}|^{2}$ with respect to $\varepsilon$ and \eqref{sumqeqc},
\begin{equation}
\begin{aligned}\label{taylorexpansion} &
\frac{\prod_{i=1}^{N_{2}}|x-\varepsilon q_{i}|^{2}}{|x|^{2N_{2}}}
=\Big\{1-\frac{2\sum^{N_2}_{i=1}q_i\cdot
x}{|x|^{2}}\varepsilon\\&+\Big(\sum_{i=1}^{N_{2}}|q_{i}|^{2}+\frac{(2\sum^{N_2}_{i=1}q_i
\cdot x)^{2}}{2|x|^{2}}-2\sum_{i=1}^{N_{2}}\frac{(x\cdot
q_{i})^{2}}{|x|^{2}}\Big)\frac{\varepsilon^{2}}{|x|^{2}}
+
\sum_{k=3}^{2N_{2}}O\Big(\frac{\varepsilon
^{k}}{|x|^{k}}\Big)\Big\}
\\&=\Big\{1+\frac{\overrightarrow{c}\cdot
x}{|x|^{2}}\varepsilon+\Big(\sum_{i=1}^{N_{2}}|q_{i}|^{2}+\frac{(%
\overrightarrow{c}\cdot x)^{2}}{2|x|^{2}}-2\sum_{i=1}^{N_{2}}\frac{(x\cdot
q_{i})^{2}}{|x|^{2}}\Big)\frac{\varepsilon^{2}}{|x|^{2}}+%
\sum_{k=3}^{2N_{2}}O\Big(\frac{\varepsilon
^{k}}{|x|^{k}}\Big)\Big\}.\end{aligned}
\end{equation}
By using \eqref{definitionofdx}, \eqref{taylorexpansion},  and Lemma \ref{diffgreen}, we also see that
if $  \varepsilon R_{0}\le |x|$, then
\begin{equation*}
\begin{aligned}&e^{\tilde{W}_a(x)+2\sum^{N_2}_{i=1}\ln|x-\varepsilon
q_i|+u_{2,\varepsilon}^*(x)}
\\&=\frac{e^{W_a(x)}\prod^{N_2}_{i=1}|x-\varepsilon q_i|^2}{|x|^{2N_2}}
\Big\{\sum_{k=0}^2\frac{(u_{2,\varepsilon}^*(x)-{\gamma_1}\ln|x|)^k}{k!}+O(|u_{2,\varepsilon}^*(x)-{\gamma_1}\ln|x||^3)
\Big\}\\&=e^{W_a(x)}\Big\{1+\frac{\overrightarrow{c}\cdot
x}{|x|^2}\varepsilon+\Big(\sum^{N_2}_{i=1}|q_i|^2
+\frac{(\overrightarrow{c}\cdot x)^2}{2|x|^2}-2\sum^{N_2}_{i=1}\frac{(x\cdot
q_i)^2}{|x|^2}\Big)\frac{\varepsilon^2}{|x|^2}+O\Big(\frac{%
\varepsilon^3}{|x|^3}\Big)\Big\} \\&\times
\Big\{1-\frac{\overrightarrow{c}\cdot
x}{|x|^2}\varepsilon-d(x)\varepsilon^2+\frac{(\overrightarrow{c}\cdot
x)^2}{2|x|^4}\varepsilon^2+O\Big(\frac{\varepsilon^3}{|x|^3}(1+|\ln|x||)^3%
\Big)\Big\},
\end{aligned}
\end{equation*}
which implies Lemma \ref{mean}-(ii).
\end{proof}

Now we recall the following  result.
\begin{lemma}
\label{w}\cite {CI,CFL}
(i) For any $w\in X^1_\alpha\cup X^2_\alpha$, there exists a constant $%
c_w\in\mathbb{R}$ such that
\begin{equation*}
w(x)=c_w+\frac{1}{2\pi}\int_{\mathbb{R}^2}\ln|x-y|\Delta w(y)dy.
\end{equation*}

(ii) There exists a constant $C>0$, independent of $w\in X^i_{\alpha}$,
satisfying
\begin{equation*}
|w(x)|\le C\| w \|_{X^{i}_{\alpha}}\ln(2+|x|)\ \ \mbox{ for all}\ \ x\in\mathbb{R}^2,\ w
\in X^{i}_{\alpha}, i=1,2.
\end{equation*}
 \end{lemma}
\begin{proof}See the proof for    \cite[Lemma 1.1]{CI} and    \cite[Theorem 4.1]{CFL}.
\end{proof}
For any $h\in Y_{\alpha}$, we have the following estimation.
\begin{lemma}\label{estfory}
   There is a constant $c >0$ such that

(i)  $\left|\intr (\ln|x-y|-\ln|x|)h(y)dy\right|\le c  |x|^{-\frac{\alpha}{2}}(\ln|x|+1)\|h\|_{Y_\alpha}$  for all $x\in\RN\setminus B_2(0)$ and $h\in Y_{\alpha}$;

 (ii) $\left|\intr  \ln|x-y| h(y)dy\right|\le c  \|h\|_{Y_\alpha}$  for all $x\in   B_2(0)$ and $h\in Y_{\alpha}$.

\end{lemma}
\begin{proof}(i) If $|y|\le\frac{|x|}{2}$, then  $|x-y|\ge |x|-|y|\ge \frac{|x|}{2}$, which implies  for some $\theta\in(0,1)$,
\begin{equation}\label{smally}
\begin{aligned}|\ln|x-y|-\ln|x||=\left|\frac{|x-y|-|x|}{\theta|x-y|+(1-\theta)|x|}\right|\le \frac{2|y|}{|x|}. \end{aligned}\end{equation}
We also see that if  $|y|\ge 2|x|\ge4$, then  \begin{equation}
\begin{aligned}\label{largey}\frac{3|y|}{2}\ge  |x-y|\ge  |y|-|x|\ge \frac{|y|}{2}\ge2.\end{aligned}\end{equation}
By    H\"{o}lder's inequality and \eqref{smally}-\eqref{largey}, we see that if $|x|\ge2$, then  there are constants $c_1, c_2>0$, independent of $x\in\RN\setminus B_2(0)$ and $h\in Y_{\alpha}$, such that  \begin{equation*}
\begin{aligned}
&\left|\intr (\ln|x-y|-\ln|x|)h(y)dy\right|\\&\le   c_1\Big(\int_{|y|\le\frac{|x|}{2}}\frac{|y|^{-\frac{\alpha}{2}}|y|^{1+\frac{\alpha}{2}}|h(y)|}{|x|}dy +\int_{\frac{|x|}{2}\le |y|\le 2|x|} (|\ln|x-y||+|\ln|x||)|y|^{-1-\frac{\alpha}{2}}|y|^{1+\frac{\alpha}{2}}|h(y)|dy\\&+\int_{|y|\ge2|x|}(|\ln|y||+|\ln|x||)|y|^{-1-\frac{\alpha}{2}}|y|^{1+\frac{\alpha}{2}}|h(y)|dy \Big)
\le c_2  |x|^{-\frac{\alpha}{2}}(\ln|x|+1)\|h\|_{Y_\alpha}.\end{aligned}\end{equation*}

(ii) If $|x|\le 2$ and $|y|\ge 4$, then $|y|\ge 2|x|$, which implies \eqref{largey}. By  using  H\"{o}lder's inequality again,  we see that if $|x|\le2$, then \begin{equation*}
\begin{aligned}
&\intr  |\ln|x-y| h(y)| dy \le   \int_{|y|\le4}|\ln|x-y|  h(y)|dy + \int_{|y|\ge4}\frac{|\ln|x-y||}{|y|^{1+\frac{\alpha}{2}}}|y|^{1+\frac{\alpha}{2}}| h(y)|dy  \le c \|h\|_{Y_\alpha},\end{aligned}\end{equation*}where $c>0$ is a  constant, independent of $x\in  B_2(0)$ and $h\in Y_{\alpha}$.
\end{proof}
If $w\in X^1_\alpha\cup X^2_\alpha$, then we cannot use integration by parts directly due to the behavior of $w$ at infinity. By using an approximate arguments with a smooth cut-off function, we have the following  simple result, but crucial for our next arguments.
\begin{lemma}\label{integrationbyparts}
For any  $w\in X^1_\alpha\cup X^2_\alpha$, we have

(i)  $\int_{\mathbb{R}^2}(\Delta  w)  Z_{i}dx =-2\int_{\RN}e^{W_0}Z_{i} wdx$ \ for $i=1,2$, and

(ii) $\int_{\mathbb{R}^2}(\Delta  w)  Z_{0}dx =-\int_{\RN}\Delta wdx-2\int_{\RN}e^{W_0}Z_{0} wdx$
 \end{lemma}
\begin{proof}  We recall that \begin{equation}
\begin{aligned}\left\{ \begin{array}{ll}\label{recallz0}
Z_{0}(x)+1=\frac{1-|x|^{2N_{2}+\gamma_1+2}}{1+|x|^{2N_{2}+\gamma_1+2}}+1=O(|x|^{-2N_{2}-\gamma_1-2}),
\\ Z_{1}(x)=\frac{|x|^{N_{2}+\frac{{\gamma_1}}{2}+1}\cos[(N_{2}+\frac{{\gamma_1}}{2}%
+1)\theta]}{1+|x|^{2N_{2}+\gamma_1+2}}=O(|x|^{-N_{2}-\frac{\gamma_1}{2}-1}),
\\ Z_{2}(x)=\frac{|x|^{N_{2}+\frac{{\gamma_1} }{2}+1}\sin[(N_{2}+\frac{{\gamma_1}}{2}%
+1)\theta]}{1+|x|^{2N_{2}+\gamma_1+2}}=O(|x|^{-N_{2}-\frac{\gamma_1}{2}-1}).
\end{array}\right. \textrm{as}\ |x|\to+\infty\end{aligned}
\end{equation}
 Let  $0\le\tilde{\chi}(x)\le 1$ be  a smooth cut-off function such that
$\tilde{\chi}(x)=1$ if $|x|\le1$ and $\tilde{\chi}(x)=0$ if $|x|\ge2$. For any $\delta>0$, we denote $\tilde{\chi}_{\delta}(x)=\tilde{\chi}(\delta x)$.  Since $w\in X^1_\alpha\cup X^2_\alpha$,  we have $\Delta w\in Y_{\alpha}$, and    if $|x|\ge \frac{1}{\delta}$, then there is a constant $C_a>0$, independent of $w\in X^1_\alpha\cup X^2_\alpha$ and $\delta>0$, such that
\begin{equation*}\begin{aligned}\int_{|x|\ge \frac{1}{\delta}}\sum_{i=1}^2|(\Delta w) Z_{i}|+|(\Delta w) (Z_{0}+1)|dx\le C_a\delta ^{\frac{\alpha}{2}}\|\Delta w\|_{Y_\alpha}.\end{aligned}\end{equation*}
Then we have \begin{equation}\label{intbypart1}\intr(\Delta w) Z_{i}dx=\lim_{\delta\to0}\intr( \Delta w) \tilde{\chi}_{\delta} Z_{i}dx\ \textrm{for}\ i=1,2,\ \ \textrm{and}\end{equation} \begin{equation}\intr(\Delta w) (Z_{0}+1)dx=\lim_{\delta\to0}\intr( \Delta w) \tilde{\chi}_{\delta} (Z_{0}+1)dx.\end{equation}
Then by integration by parts and \eqref{recallz0}, we see that \begin{equation}\begin{aligned} &\intr( \Delta w) \tilde{\chi}_{\delta} Z_{i}dx =\intr  w \Delta( \tilde{\chi}_{\delta} Z_{i})dx
\\&= \intr  w \Big((\Delta \tilde{\chi}_{\delta})Z_{i}+2 \nabla  \tilde{\chi}_{\delta}\cdot\nabla Z_{i}  -2\tilde{\chi}_{\delta} e^{W_0} Z_{i}\Big)dx\ \textrm{for}\ i=1,2,\ \textrm{and}\end{aligned}\end{equation}\begin{equation}\begin{aligned} &\intr( \Delta w) \tilde{\chi}_{\delta} (Z_{0}+1)dx =\intr  w \Delta\{ \tilde{\chi}_{\delta} (Z_{0}+1)\}dx
\\&= \intr  w \Big((\Delta \tilde{\chi}_{\delta})(Z_{0}+1)+2 \nabla  \tilde{\chi}_{\delta}\cdot\nabla Z_{0}  -2\tilde{\chi}_{\delta} e^{W_0} Z_{0}\Big)dx,\end{aligned}\end{equation}
 Since $w\in X^1_\alpha\cup X^2_\alpha$,  by Lemma \ref{w}, it is easy to see that $\|w\rho_2\|_{L^2(\RN)}$ is finite.
By using H\"{o}lder's inequality, we have  a constant $c_a>0$, independent of $w\in X^1_\alpha\cup X^2_\alpha$ and $\delta>0$, such that
\begin{equation}\begin{aligned}  & \intr \sum_{i=1}^2\Big( |w  (\Delta \tilde{\chi}_{\delta})Z_{i}| + | \nabla  \tilde{\chi}_{\delta}\cdot\nabla Z_{i}|\Big)+|w  (\Delta \tilde{\chi}_{\delta})(Z_{0}+1)| + | \nabla  \tilde{\chi}_{\delta}\cdot\nabla Z_{0}| dx \\&\le c_a \delta^{\frac{2-\alpha}{2}}\|w\rho_2\|_{L^2(\RN)}, \ \ \textrm{and}\end{aligned}\end{equation}  \begin{equation}\begin{aligned}\int_{|x|\ge \frac{1}{\delta}}\sum_{i=1}^2 |  w  e^{W_0} Z_{i}|+|  w  e^{W_0} Z_{0}| dx\le c_a\delta ^{2}\|w\rho_2\|_{L^2(\RN)}.\label{intbtpart2}\end{aligned}\end{equation}
By \eqref{intbypart1}-\eqref{intbtpart2}, we complete the proof of Lemma \ref{integrationbyparts}. \end{proof}
We recall the projection map ${Q}$ defined in \eqref{proj_q}. Now we have the following lemma.
\begin{lemma}
\label{pronorm} \begin{equation}\label{eeqref}
\Vert {Q}h\Vert_{Y_{\alpha}}\leq c\Vert h\Vert_{Y_{\alpha}},\end{equation}
 where $c>0$ is a constant which is independent of $h\in Y_\alpha$.
\end{lemma}
\begin{proof}If $\frac{{\gamma_1}}{2}\notin\mathbb{N}$, \eqref{eeqref} is trivial.  So we consider the case   $\frac{{\gamma_1}}{2}\in\mathbb{N}$.
By ${Q}h=h-\sum_{i=1}^2c_iZ_{i}\in F_\alpha$, we have for $i\neq j$,
$
0=\int_{\mathbb{R}^2} ({Q} h) Z_{i}dx=\int h Z_{i}-c_i Z_{i}^2-c_j Z_{i}Z_{j}dx.
$
By using H\"{o}lder inequality,  we get that
$
\Big| \int_{\mathbb{R}^2} hZ_{i}dx\Big|\le
 \|h\|_{Y_\alpha}\|Z_{i}(1+|x|)^{-1-\frac{\alpha}{2}}\|_{L^2(\mathbb
{R}^2 )}.
$
Moreover, we note that
\begin{equation*}
\begin{aligned}
& \left(\begin{array}{ll}\ \ \int_{\RN} Z_{1}^2dx  \quad \ \ \ \
\quad \ \int_{\RN} Z_{1}Z_{2}dx\\ \  \  \int_{\RN} Z_{1}Z_{2}dx \quad \quad
\quad \int_{\RN} Z_{2}^2dx \end{array}\right) \left(\begin{array}{ll}c_1\\
c_2\end{array}\right)=\left(\begin{array}{ll}\int_{\RN}hZ_{1}dx\\
\int_{\RN}hZ_{2}dx\end{array}\right). \end{aligned}
\end{equation*}
We see that  $\int_{\RN} Z_{i}Z_{j}=\int_{\RN}(Z_{i})^2dx \delta_{i,j}$ yields
$\|{Q}h\|_{Y_\alpha}\le\|h\|_{Y_\alpha}+\sum_{i=1}^2|c_i|\|Z_{i}\|_{Y_\alpha}
\le c\|h\|_{Y_\alpha},
$ where $c>0$ is a constant which is independent of $h\in Y_\alpha$.
\end{proof}
{\bf Acknowledgement}\\
The authors wish to thank an anonymous referee very much for careful reading
and valuable comments. Y. Lee was supported by the National Research Foundation of Korea(NRF) grant funded by the Korea government(MSIT) (No. NRF-2018R1C1B6003403).

\end{document}